\documentclass{tac}

\usepackage{amsmath,amssymb,amsfonts,amsbsy}
\usepackage[utf8]{inputenc}
\usepackage[mathcal]{euscript}
\usepackage{mathrsfs}
\usepackage{mathtools}
\usepackage{dsfont}
\usepackage[all]{xy}
\usepackage{verbatim,adjustbox}
\usepackage{marvosym}
\usepackage{stackrel}
\usepackage{graphicx}
\usepackage[svgnames]{xcolor}
\definecolor{blue(munsell)}{rgb}{0.0, 0.5, 0.69}
\usepackage{lipsum}
\usepackage{svg}
\usepackage{bbm}
\usepackage{caption}
\usepackage{enumitem}
\usepackage{epigraph}
\usepackage{fdsymbol}
\usepackage{relsize}
\usepackage{subfig}
\usepackage{framed}
\usepackage{hyperref}
\usepackage{quiver}

\usepackage{tikz}
\usepackage{tikz-cd}

\usepackage[all]{xy}

\input diagxy

\usepackage[colorlinks=true]{hyperref}
\hypersetup{allcolors=[rgb]{0.1,0.1,0.4}}

\usepackage{cleveref}

\author{Ivan Di Liberti, Gabriele Lobbia, and Lurdes Sousa}

\thanks{The first author was supported by the Swedish Research Council (SRC, Vetenskapsr{\aa}det) under Grant No.~2019-04545. The research has received funding from Knut and Alice Wallenbergs Foundation through the Foundation's program for mathematics.\\
The third author thanks the partial support by the Centre for Mathematics of the University of Coimbra
- UIDB/00324/2020, funded by the Portuguese Government through FCT/MCTES}
\address{
Ivan Di Liberti: \newline
Department of Mathematics\newline
Stockholm University\newline
Stockholm, Sweden
\\[5pt]
Gabriele Lobbia: \newline
Department of Mathematics and Statistics\newline
Masaryk University, Faculty of Sciences\newline
Kotl\'{a}\v{r}sk\'{a} 2, 611 37 Brno, Czech Republic
\\[5pt]
Lurdes Sousa: \newline
Polythecnic of Viseu, Portugal\newline
CMUC, University of Coimbra, Portugal
}

\title{KZ-pseudomonads and Kan injectivity}


\eaddress{diliberti.math@gmail.com\CR  lobbia@math.muni.cz\CR sousa@estv.ipv.pt}

\theoremstyle{plain}
\newtheorem{constr}[thm]{Construction}
\newtheorem{notat}[thm]{Notation}

\newcommand{\catk}{\mathcal{K}} 
\newcommand{\cata}{\mathcal{A}} 
\newcommand{\cc}{\mathbf{C}} 
\newcommand{\linj}{\mathbf{WLInj}} 
\newcommand{\slinj}{\mathbf{LInj}} 

\newcommand{\maps}{\mathcal{H}} 
\newcommand{\ord}{\mathbf{Ord}}
\newcommand{\Cat}{\mathbf{Cat}}

\newcommand{\bicolim}{\textrm{bicolim}}
\def\op{\mathrm{op}}

\newcommand\ch{\mathcal {H}}
\def\LInj{\mathbf{LInj}}

\newcommand\hcomp{\circ}
\newcommand\unit{\delta}
\newcommand\sat{\text{sat}}

\def\red{\color{red}}

\def\blue{\color{black}}
\def\black{\color{black}}

\newcommand{\compX}{\mathbf{x}}

  \usepackage{graphicx}

\begin{document}

\dedication{Dedicated to the memory of Marta Bunge (1938-2022)}

\maketitle

\begin{abstract}
We introduce the notion of Kan injectivity in 2-categories and study its properties. For an adequate $2$-category $\catk$,  we show that every set of morphisms $\maps$ induces a KZ-pseudomonad on $\catk$ whose $2$-category of pseudoalgebras is the locally full sub-2-category of all objects (left) Kan injective with respect to $\maps$ and morphisms preserving Kan extensions. The main ingredient is the construction of a (pseudo)chain whose appropriate ``convergence" is ensured by a small object argument.
\end{abstract}

\section*{Introduction}

A classical problem in category theory goes under the name of the \textit{orthogonal subcategory problem}. For $\maps$ a class of maps in a category $\cc$, we ask whether the full subcategory of orthogonal objects $\maps^\perp$ is reflective in $\cc$, that is, $\maps^\perp$ is the category of algebras of an idempotent monad.

There are several reasons to study orthogonal subcategories and their reflectivity, because many situations in mathematics can be reduced to an orthogonality class of objects. For example,
the set $\mathcal{H}$ of maps that specifies the orthogonality class can be understood as a set of axioms that the objects in the orthogonality class must satisfy (see \cite{adamek2006logic} for a theoretical approach to this motto, or \cite[1.33]{adamek_rosicky_1994}
for some practical examples of how it functions).
 Thus, orthogonality  offers a categorical tool to axiomatise convenient subcategories.

The orthogonal subcategory problem has a longstanding tradition and was approached by several authors.
Peter Freyd and Max Kelly \cite{freyd1972categories} provided what later became a standard reference on the topic. In \cite{kelly1980}, Kelly unified the work  of earlier authors, by providing a beautiful solution for this problem in a broad setting by means of the colimit of a transfinite sequence. This construction is quite in the same spirit of the celebrated \textit{small object argument}. In a more recent account, the work of Ji\v r\'{i} Ad\'amek and Ji\v r\'{i}  Rosick\'{y} \cite[Chap. 1.C]{adamek_rosicky_1994} gives a detailed description of the transfinite sequence in  locally presentable categories. Their technique  is very influential for our treatment.

\vskip1.5mm

The \textbf{aim of this paper} is to establish a similar result for a $2$-dimensional variation of the orthogonal subcategory problem which captures many relevant constructions of $2$-dimensional category theory. We will direct our study to  the interplay between Kan-injectivity and lax-idempotent pseudomonads (i.e. KZ-pseudomonads). They are natural substitutes for orthogonality and idempotent monads when working in 2-categories.

This work generalises the work of \cite{adamek2015kan} and introduces Kan injectivity in $2$-categories. An object $X$ is (left) Kan injective with respect to a map $h$ if every $f:\mbox{dom}(h)\to X$ can be extended to the codomain of  $h$  through a 2-cell

\[\begin{tikzcd}
	A & {A'} \\
	X
	\arrow["h", from=1-1, to=1-2]
	\arrow[""{name=0, anchor=center, inner sep=0}, "f"', from=1-1, to=2-1]
	\arrow[""{name=1, anchor=center, inner sep=0}, "{f/h}", from=1-2, to=2-1]
	\arrow["{\xi_f}", shorten <=5pt, shorten >=5pt, Rightarrow, from=0, to=1]
\end{tikzcd}\]
and such an extension is universal among the possible extensions; more precisely,  $(f/h, \xi_f)$ is the (left) Kan extension of $f$ along $h$.
Given a class $\mathcal{H}$ of 1-cells, we can form the locally full sub-2-category 
of all objects left Kan injective with respect to $\mathcal{H}$ and 1-cells preserving the corresponding Kan extensions.
There are two natural notions of Kan injectivity, the strongest one demanding that $\xi_f$ is invertible. We will show how they relate to each other in subsection \ref{comparison}, concluding that both notions give rise to the same Kan injective sub-2-categories. 
It is known that some relevant $2$-categories can be described via Kan injectivity (Example \ref{rexsinj}). 
We aim to push this observation and show that a vast class of interesting $2$-categories can be described via \textit{Kan injectivity axioms}. To do so, we will link Kan injective sub-2-categories to KZ-pseudomonads.

The concept of KZ-pseudomonad  in a 2-category (also known as lax-idempotent pseudomonad or KZ-doctrine), presented by Anders Kock in \cite{kock_1995} and by Zöberlein in \cite{zoberlein1976doctrines} generalises the one of idempotent monad in  ordinary categories. In \cite{bunge2006singular} Marta Bunge and Jonathan Funk characterised the 2-adjunctions giving rise to KZ-pseudomonads.
 In \cite{marm-wood_lax-idemp}, Francisco Marmolejo and Richard Wood showed that a KZ-pseudomonad in a 2-category and its algebras may be presented in terms of left Kan extensions. In particular,  their results can essentially be summarised as in Theorem \ref{thm:marm-wood}. This theorem was previously shown   for the particular case of order-enriched categories in \cite{carv-sousa_order-refl}, and is an important tool in the proof of  Theorem \ref{thm:linj-kz-adj}.

 \vskip1.5mm

Our \textbf{main result} (Theorem \ref{thm:linj-kz-adj}) shows that, in a locally small 2-category with small bicolimits, if all objects satisfy a convenient notion of smallness (see Definition \ref{def:small}), then for every set $\maps$ of morphisms of $\catk$ the inclusion of the corresponding Kan-injective sub-2-category $\LInj(\maps)$
$$\slinj(\maps)\hookrightarrow\catk$$
is the right adjoint 2-functor of a KZ-adjunction. To this end, we construct, for each object $X$, a transfinite (pseudo)chain (see Section~\ref{chain})  leading to the components of the unit of the KZ-adjunction. $\LInj(\mathcal{H})$ is then essentially the corresponding category of (pseudo)algebras. This chain generalizes the Kan injective reflection chain presented in \cite{adamek2015kan} for order-enriched categories. Here, the main factor allowing us to take off from the locally thin context of \cite{adamek2015kan} is the use of a special colimit, which we call \textit{coequinserter}.

\vskip1.5mm

The \textbf{structure of the paper} goes as follows. In \Cref{inizio} we start by introducing weak Kan injectivity and Kan injectivity (Definition \ref{definizionegenerale}). We put the notions into context, making the due comparison to the literature and proving the closedness under (bi)limits of $\LInj(\mathcal{H})$  (Proposition~\ref{prop:closed_bilimits}), a soundness result towards the main theorem.
In Proposition \ref{fromratorali}, we show that, in any $2$-category with bicocomma objects, for every class of maps $\mathcal{H}$ there exists a class of maps $\bar{\mathcal{H}}$ such that $\linj(\maps)=\slinj(\bar{\maps})$, where $\linj(\maps)$ refers to the weak Kan-injective sub-2-category. We also show that every class of morphisms saturated under Kan-injectivity contains
 all lari 1-cells and  is closed under composition, bicocomma, bipushouts and wide bipushouts (Proposition \ref{prop:saturation-irr}).

The subsequent three sections build the technology needed to prove our main theorem. In \Cref{sec2}, after recalling some results due to Marmolejo and Wood on the structure of KZ-pseudomonads, we formulate the result (Theorem \ref{thm:marm-wood}) which will serve as a basis for the proof of our main theorem. We finish this section with Corollary \ref{cor:Linj(H)} stating that,  for every class of 1-cells $\mathcal{H}$, if the inclusion $\slinj(\mathcal{H}) \hookrightarrow \catk$ is the right part of a KZ-adjunction, then the 2-category of pseudoalgebras of the corresponding KZ-monad is essentially $\slinj(\mathcal{H})$.

\Cref{chain} gives an explicit construction of a pseudochain (Construction \ref{kansmallobject}) which provides the candidate left pseudoadjoint to the forgetful functor $\slinj(\maps)\hookrightarrow\catk.$ It is shown that in this pseudochain $(x_{ij})_{i\leq j}$ every $x_{0i}$ belongs to the Kan injective saturation of $\mathcal{H}$.

\Cref{TheTheorem} contains our main theorem, which is the following:

\vskip1mm

\noindent {\sc Theorem} 
({\ref{thm:linj-kz-adj}})
{\em Let $\catk$ be a locally small 2-category with small bicolimits and such that all objects are small.  Then, for any set $\maps$ of 1-cells in $\catk$, the inclusion 2-functor $\slinj(\maps)\hookrightarrow\catk$ is the right part of a KZ-adjunction. Moreover, $\slinj(\maps)$ is the corresponding Eilenberg--Moore $2$-category, up to equivalence of 2-categories.}

\vskip1mm

\Cref{sec5}, our last section, applies the machinery developed in the paper to study a broad class of $2$-categories defined over Lex, the $2$-category of categories with finite limits. The main result (Theorem \ref{characteriation of exactness}) of the section relates Kan injectivity with the theory of lex-colimits by Garner and Lack \cite{lex} and offers an alternative characterization of $\Phi$-exactness.

\section{Kan injectivity} \label{inizio}

\subsection{Left Kan injectivity -- weak and strong} \label{onkinj}
Let $\catk$ be a 2-category, and $f\colon A\to X$ and $h\colon A\to A'$ two 1-cells in $\catk$. Recall that the \textbf{left Kan extension of $f$ along $h$} is defined as a 1-cell $f/h\colon A'\to X$ together with a 2-cell,
\begin{equation}\label{eq:Kan-ext}
\begin{tikzcd}
	A & {A'} \\
	X
	\arrow["h", from=1-1, to=1-2]
	\arrow[""{name=0, anchor=center, inner sep=0}, "f"', from=1-1, to=2-1]
	\arrow[""{name=1, anchor=center, inner sep=0}, "{f/h}", from=1-2, to=2-1]
	\arrow["{\xi_f}"{xshift=0.075cm,yshift=0.075cm}, shorten <=5pt, shorten >=3pt, Rightarrow, from=0, to=1]
\end{tikzcd}
\end{equation}
such that for any other 1-cell $g\colon A'\to X$ with a 2-cell $\alpha\colon f\Rightarrow g\circ h$ there exists a unique 2-cell $\overline{\alpha}\colon f/h\Rightarrow g$ such that we have the equality $\alpha=(\bar{\alpha}\hcomp h)\cdot \xi_f$:
\[\begin{tikzcd}
	A && {A'} && A && {A'} \\
	&&& {=} \\
	X &&&& X
	\arrow["h", from=1-1, to=1-3]
	\arrow[""{name=0, anchor=center, inner sep=0}, "f"', from=1-1, to=3-1]
	\arrow[""{name=1, anchor=center, inner sep=0}, "{f/h}"'{pos=0.4, xshift=0.175cm, yshift=0.06cm}, curve={height=12pt}, from=1-3, to=3-1]
	\arrow[""{name=2, anchor=center, inner sep=0}, "g", curve={height=-12pt}, from=1-3, to=3-1]
	\arrow[""{name=3, anchor=center, inner sep=0}, "g", curve={height=-12pt}, from=1-7, to=3-5]
	\arrow["h", from=1-5, to=1-7]
	\arrow[""{name=4, anchor=center, inner sep=0}, "f"', from=1-5, to=3-5]
	\arrow["{\xi_f}"{xshift=0.125cm, yshift=0.05cm}, shorten <=6pt, shorten >=6pt, Rightarrow, from=0, to=1]
	\arrow["{\bar{\alpha}}"{xshift=-0.05cm, yshift=0.03cm}, shorten <=4pt, shorten >=4pt, Rightarrow, from=1, to=2]
	\arrow["\alpha"{xshift=-0.1cm,yshift=0.025cm}, shorten <=10pt, shorten >=10pt, Rightarrow, from=4, to=3]
\end{tikzcd}\]
Of course, such a $1$-cell $f/h$ is defined up to isomorphism.

A 1-cell $p\colon X\to X'$ \textbf{preserves the left Kan extension} $(f/h,\xi_f)$ if the pair $(p(f/h), p\hcomp \xi_f)$ forms a  left Kan extension of $pf$ along $h$, i.e. there is an invertible 2-cell $(pf)/h\cong p\circ f/h$ satisfying the following equation. 
\[\begin{tikzcd}
	A && {A'} && A && {A'} \\
	X &&& {=} & X & {} \\
	{X'} &&&& {X'} \\
	\arrow["h", from=1-1, to=1-3]
	\arrow[""{name=0, anchor=center, inner sep=0}, "{(pf)/h}"'{pos=0.3, xshift=0.2cm}, curve={height=12pt}, from=1-3, to=3-1]
	\arrow[""{name=1, anchor=center, inner sep=0}, "{p\circ f/h}", curve={height=-12pt}, from=1-3, to=3-1]
	\arrow["h", from=1-5, to=1-7]
	\arrow[""{name=2, anchor=center, inner sep=0}, "f"', from=1-5, to=2-5]
	\arrow["p"', from=2-5, to=3-5]
	\arrow[""{name=3, anchor=center, inner sep=0}, "{f/h}", curve={height=-12pt}, from=1-7, to=2-5]
	\arrow["f"', from=1-1, to=2-1]
	\arrow["p"', from=2-1, to=3-1]
	\arrow["{\xi_f}"{xshift=-0.2cm,yshift=0.08cm}, shorten <=11pt, shorten >=11pt, Rightarrow, from=2, to=3]
	\arrow["{\xi_{pf}}"{xshift=0.2cm,yshift=0.08cm}, shorten <=2pt, shorten >=4pt, Rightarrow, from=2-1, to=0]
	\arrow["\cong", shorten <=4pt, shorten >=4pt, Rightarrow, from=0, to=1]
\end{tikzcd}\]

Throughout the paper we will make use of the notion of left Kan injectivity given below. We also present the notion of weakly left Kan injectivity, which will be discussed in this section.

\begin{definition} \label{definizionegenerale} 
\begin{enumerate} 
    \item An object $X\in\catk$ is \textbf{weakly left Kan injective} with respect to a family of $1$-cells $\mathcal{H}$ if, for all $h\colon A\to A'$ in $\maps$ and any $f\colon A\to X$ in $\catk$ the left Kan extension $(f/h,\xi_f)$ of $f$ along $h$ exists, see \eqref{eq:Kan-ext}.

By the general theory of Kan extensions, this amounts to say that the representable functor $\catk(-,X)\colon\catk\to \Cat$ maps every 1-cell of $\maps$ to a right adjoint 1-cell.

    \item We say that $X \in \mathcal{K}$ is \textbf{left Kan injective} with respect to $\maps$ if it is weakly left Kan injective and, moreover, the 2-cells $\xi_f$ are invertible. This amounts to say that the representable functor $\catk(-,X)$ maps every 1-cell of $\maps$ to a \textit{rali} in $\Cat$ (i.e. a right adjoint with invertible unit).\footnote{\textit{rali} stands for \textit{right adjoint left inverse}; analogously, \textit{lari} stands for \textit{left adjoint right inverse}, i.e. a left adjoint with invertible unit. Similarly, we write lali and rari for left/right adjoints with invertible counit.}

    \item A 1-cell  $p\colon X\to X'$ of $\catk$ is \textbf{(weakly) left Kan injective} with respect to  $\maps$ if its domain and codomain are so and $p$ preserves left Kan extensions along 1-cells in $\maps$.

\item  We can form a locally full sub-2-category
$\linj(\maps)$
 of $\catk$ with objects all weakly left Kan injectives with respect to $\maps$ and 1-cells between them which preserve left Kan extensions along maps in $\maps$. Similarly, we define $$\slinj(\maps)$$
restricting objects to left Kan injectives with respect to $\maps$.
\end{enumerate}
\end{definition}

Bunge and Funk \cite{BungeFunk1999} studied certain KZ-doctrines, called admissible, and characterised their algebras in terms of weakly left Kan injectivity, considering pointwise left Kan extensions (see also \cite{Street1981}). As we will see in the next section, we may characterise the algebras of any KZ-doctrine in terms of left Kan injectivity, and this fact is an important tool in our paper.

\begin{remark}\label{rem:Beck-Chev}Consider the diagram below, where $X$ and $X'$ are left Kan injective with respect to $h\colon A\to A'$, and $\overline{h_X}:=(-)/h$ is the left adjoint of $\catk(h,X)$. A 1-cell $p\colon X\to X'$ preserves left Kan extensions along $h$ if and only if it satisfies an appropriate
  Beck--Chevalley condition, namely,  the canonical natural transformation below is invertible.
  $$\xymatrix{\catk(A',X)\ar[d]_{\catk(A',p)}^{\hspace{16mm} \Longleftarrow }& &\catk(A,X)\ar[ll]_{\overline{h_X}}\ar[d]^{\catk(A,p)}&X\ar[d]^p\\
\catk(A',X')&&\catk(A,X')\ar[ll]^{\overline{h_{X'}}}&X'}$$
\end{remark}

This characterization concerning Kan injectivity leads to a nice behaviour of Kan injective sub-2-categories with respect to bilimits and pseudolimits, which we describe in the following proposition.

\begin{proposition}\label{prop:closed_bilimits}The inclusion 2-functor $\LInj(\ch)\hookrightarrow\catk$ creates bilimits and pseudolimits.
\end{proposition}

\begin{proof}
Let us consider a pseudofunctor $D\colon I\to \slinj(\maps)$ (with $I$ a small 2-category) and a weight $W\colon I\to\Cat$ (strict 2-functor).
\begin{enumerate}
\item For any object $i\in I$, $Di\in\slinj(\maps)$, i.e. $\catk(h,Di)=:{h_{Di}}^\ast$ is a rali (let us denote with $\overline{h_{Di}}\dashv {h_{Di}}^\ast$ the adjunction).

\item For any 1-cell $u\colon i\to j\in I$, $Du$ is Kan injective, i.e.
\[\begin{tikzcd}[ampersand replacement=\&]
	{\catk(A,D_i)} \& {\catk(B,D_i)} \\
	{\catk(A,D_j)} \& {\catk(B,D_j)}
	\arrow["{\overline{h_{Di}}}", from=1-1, to=1-2]
	\arrow[""{name=0, anchor=center, inner sep=0}, "{Du\circ-}"', from=1-1, to=2-1]
	\arrow[""{name=1, anchor=center, inner sep=0}, "{Du\circ-}", from=1-2, to=2-2]
	\arrow["{\overline{h_{Dj}}}"', from=2-1, to=2-2]
	\arrow["\cong"{description}, draw=none, from=0, to=1]
\end{tikzcd}\] 
Let us note that these isomorphisms make $\overline{h_{D-}}$ into a pseudonatural transformation, the composition and the other axioms follow by the universal property of Kan extensions.
\end{enumerate}

It is easy to check that also ${h_{D-}}^\ast$ is a pseudonatural transformation. In particular we have $\overline{h_{D-}}\dashv{h_{D-}}^\ast$ in the 2-category $[I,\Cat]$ of pseudofunctors, pseudonatural transformations and modifications (which also makes ${h_{D-}}^\ast$ a rali). It is well-known that hom-functors into $\Cat$ preserve adjunctions (see \cite[Proposition~I,6.3]{Gray:Adj}). Then, setting $H:=\overline{h_{D-}}\circ-$ and $H^\ast:={h_{D-}}^\ast\circ-$, we get an adjunction
\begin{center}
\begin{tikzcd}[ampersand replacement=\&]
	{[I,\Cat](W,\catk(A,D-))} \& {[I,\Cat](W,\catk(B,D-))}
	\arrow[""{name=0, anchor=center, inner sep=0}, "H", curve={height=-12pt}, from=1-1, to=1-2]
	\arrow[""{name=1, anchor=center, inner sep=0}, "{H^\ast}", curve={height=-12pt}, from=1-2, to=1-1]
	\arrow["\perp"{description}, draw=none, from=1, to=0]
\end{tikzcd}
in $\Cat$.
\end{center}
Now, let us assume that the $W$-weighted bi/pseudolimit of $D$ exists in $\catk$, i.e.
\begin{enumerate}
\item \textbf{Pseudolimit:} There exists an object $L_p\in\catk$ such that
$$\catk(A,L_p)\cong[I,\Cat](W,\catk(A,D-))$$
is an \textbf{isomorphism} of categories for any $A$.

\item \textbf{Bilimit:} There exists an object $L_b\in\catk$ such that
$$\catk(A,L_b)\simeq[I,\Cat](W,\catk(A,D-))$$
is an \textbf{equivalence} of categories for any $A$.
\end{enumerate}

Since both isomorphisms and equivalences of categories preserve adjunctions, in both cases we get a lali (for $L=L_p$ or $L=L_b$)
\[\begin{tikzcd}[ampersand replacement=\&]
	{\catk(A,L)} \& {\catk(B,L)}
	\arrow[""{name=0, anchor=center, inner sep=0}, "{\overline{h}}", curve={height=-12pt}, from=1-1, to=1-2]
	\arrow[""{name=1, anchor=center, inner sep=0}, "{-\circ h}", curve={height=-12pt}, from=1-2, to=1-1]
	\arrow["\perp"{description}, draw=none, from=1, to=0]
\end{tikzcd}\]

Let us consider projections $l^i_w\colon L\to Di$, i.e. the image of the object $w\in Wi$ under the $i$-component of the universal pseudonatural transformation $Wi\to\catk(L,Di)$ (the one corresponding to $1_L$). Let us consider the diagram below
\[\begin{tikzcd}[ampersand replacement=\&]
	{\catk(A,L)} \&\&\& {\catk(B,L)} \\
	\& {[W,\catk(A,D-)]} \& {[W,\catk(B,D-)]} \\
	{\catk(A,Di)} \&\&\& {\catk(B,Di)}
	\arrow[""{name=0, anchor=center, inner sep=0}, "{\overline{h_L}}", from=1-1, to=1-4]
	\arrow[""{name=1, anchor=center, inner sep=0}, "{l^i_w\circ-}"', from=1-1, to=3-1]
	\arrow[""{name=2, anchor=center, inner sep=0}, "{l^i_w\circ-}", from=1-4, to=3-4]
	\arrow["\sim"{description}, no head, from=1-1, to=2-2]
	\arrow["\sim"{description}, no head, from=2-3, to=1-4]
	\arrow[""{name=3, anchor=center, inner sep=0}, "{\overline{h_{Di}}}"', from=3-1, to=3-4]
	\arrow["{e_{i,w}}"{description}, from=2-3, to=3-4]
	\arrow["{e_{i,w}}"{description}, from=2-2, to=3-1]
	\arrow[""{name=4, anchor=center, inner sep=0}, "H", from=2-2, to=2-3]
	\arrow["\cong"{description}, draw=none, from=1, to=2-2]
	\arrow["\cong"{description}, draw=none, from=2-3, to=2]
	\arrow[""{name=5, anchor=center, inner sep=0}, "{(2)}"{description}, draw=none, from=4, to=3]
	\arrow["{(1)}"{description, pos=0.3}, draw=none, from=0, to=5]
\end{tikzcd}\]
where $e_{i,w}$ is the functor taking a pseudonatural transformation $\alpha\colon W\to\catk(A,D-)$ and evaluates its $i$-component at the object $w\in Wi$ (see below)
\[e_{i,w}\colon \alpha \longmapsto \alpha_i(w)\in\catk(A,Di)\]
and we wrote $[W,\catk(A,D-)]$ for the category $[I,\Cat](W,\catk(A,D-))$. Let us recall that $H$ sends a pseudonatural transformation $\alpha\colon W\to\catk(A,D-)$ to
\[\begin{tikzcd}[ampersand replacement=\&]
	W \& {\catk(A,D-)} \& {\catk(B,D-).}
	\arrow["\overline{{h_{D-}}}", from=1-2, to=1-3]
	\arrow["\alpha", from=1-1, to=1-2]
\end{tikzcd}\]
Hence, it is straightforward to check that the diagram (2) commutes and so the whole diagram above (since diagram (1) commutes by definition of $\overline{h_L}$). This shows that each $l^i_w$ is left Kan injective.
\end{proof}

All of this reasoning works also when we have $-\circ h$ only a left adjoint and not a  lali, hence also $\linj(\maps)$ is closed under weighted bi/pseudolimits. This also follows from the fact shown below that every weakly left Kan injective sub-2-category is left Kan injective (Proposition \ref{fromratorali}).

Next we list some examples concerning Kan injective sub-2-categories.

\begin{example}[Categories with finite colimits are weakly Kan injective] \label{rexinj}
Let $\textbf{Rex}$ be the $2$-category of small categories with finite colimits and functors preserving them. Define, in $\Cat$, \[\maps = \{\top\colon D \to 1 \mid D \text{ is a finite category}\}.\]
It then follows from \cite[X.7.1]{mac2013categories} that $\textbf{Rex} \simeq \linj(\maps)$. Of course, because finite colimits are generated by finite coproducts and coequalizers, $\maps$ can be reduced to three arrows $D \to 1$, with

\centerline{$D=0,\,
\begin{array}{|c|}
\hline
a \bullet \; \bullet b\\
\hline\end{array},\,
 \begin{array}{|c|}
\hline
\bullet {\renewcommand{\arraystretch}{0.65}\begin{array}{c} \rightarrow\\\rightarrow\end{array}}\bullet\\
\hline\end{array}$,}
 \noindent determining the existence of
an initial object, binary coproducts and coequalizers, respectively. For any class $\mathcal{D}$ of finite categories, using a similar argument, we can get the category of categories with any colimit of shape in $\mathcal{D}$. 
\end{example}

\begin{example}[Categories with finite colimits are Kan injective] \label{rexsinj}
Similarly to the previous discussion, we will describe $\textbf{Rex}$ as a  left Kan injective sub-2-category of $\Cat$. In order to do so, for all finite categories $D$, call $\hat{D}$ the category obtained from $D$ by freely adding a terminal object
 and call $\iota_D \colon D \to \hat{D}$ the canonical inclusion. Define,
\[\maps = \{\iota_D \colon D \to \hat{D} \mid D \text{ is a finite category}\}.\]

It then follows from \cite[3.1.8, 6.3.10]{riehl2017category} that $\textbf{Rex} \simeq \slinj(\maps)$. Naturally, as in the above example,  $\maps$ can be reduced to a class containing only three arrows.
\end{example}

{\blue
\begin{remark}[Pointwise Kan extension]
In the spirit of the previous example, a significant variation of this problem has been studied in the literature, for the special case of $2$-categories of prestacks  and \textit{pointwise} Kan extensions. In \cite[Theorem~9.3, Corollary~9.5]{Street1981}, Street provides a result in ${\cal K}=[{\cal C}^{\mbox{\scriptsize op}}, \Cat]$, for $\cal C$ a small bicategory, that is similar to our Theorem \ref{thm:linj-kz-adj}.
\end{remark}
}

\begin{example}[Orthogonality] In the context of ordinary categories, that is, locally discrete 2-categories, the notion of Kan-injectivity is just the classical definition of orthogonality. In this case, $\LInj(\ch)$ is the full subcategory of all objects orthogonal to $\ch$ usually denoted by $\ch^{\perp}$.
\end{example}

{\blue
\begin{example}[Other cameos: Factorization systems] 
In the general theory of algebraic weak factorization systems developed by Bourke and Garner, a variation of Kan injectivity emerges in a stricter form. Indeed, in \cite[Example~29~(iii)]{bourke2016algebraic} they essentially present our example \ref{rexinj}, but where there is a specified data of colimits. On a similar note, \cite[Section~9]{clementino2016lax} discusses strict extensions in the context of factorization systems.
\end{example}
}

\begin{example}[Fullness]  If $\ch$ is made of lax epimorphisms (i.e. for every $h\colon A\to A'$ in $\maps$ and every $X$, the functor $\catk(h,X)\colon\catk(A',X)\to \catk(A,X)$ is fully faithful), then $\LInj(\ch)$ is a full sub-2-category. Indeed, for every map $p$ between Kan injective objects, from the fact that $((pf)/h)h\cong pf \cong p(f/h)h$, it will follow that $(pf)/h\cong p(f/h)$. A detailed study on lax epimorphisms may be seen in \cite{LucatelliSousa2022}.
\end{example}

\begin{example}[Order-enriched categories] Known examples abound in the 2-category of posets and other order-enriched categories. For instance, in the category $\mathbf{Top}_0$ of $T_0$ topological spaces and continuous maps, the category of continuous lattices and maps preserving directed suprema and infima is $\mathbf{RInj}(\maps)$ for $\maps$ the class of (topological) embeddings, where $\mathbf{RInj}$ referes to right Kan injectivity in the expected sense. In the category $\mathbf{Loc}$ of locales and localic maps, the category of stably locally compact locales with convenient maps is $\LInj(\maps)$ for $\maps$ the class of flat embeddings (see \cite{elephant1}). These and other examples may be encountered in \cite{adamek2015kan} and \cite{carv-sousa_locales}.
\end{example}

\begin{remark}[A comparison with enriched weakness]
  In \cite{lack2012enriched}, Lack and Rosický introduce a very interesting notion of injectivity, which is parametric with respect to a class of maps. Let us recall it and briefly to compare it with our notion. Let $\mathcal{V}$ be a reasonably nice category to enrich in and let $\mathcal{E}$ be a class of maps in $\mathcal{V}$. Let $\mathcal{K}$ be a category enriched over $\mathcal{V}$ and $\mathcal{H}$ be a class of $1$-cells; then they define \[ \textbf{Inj}_{\mathcal{E}}(\mathcal{H})\] to be the full subcategory of $\mathcal{K}$ of those objects $X$ such that $\mathcal{K}(-,X)$ maps $\mathcal{H}$ to $\mathcal{E}$. This definition  resonates with ours. Indeed, let us consider the particular choice $\mathcal{V}=\Cat$ and $\mathcal{E}=\mathsf{ra},\mathsf{rali}$, where $\mathsf{ra}$ and $\mathsf{rali}$ stand for the classes of right adjoints  and of right adjoint left inverses, respectively.

It is clear that on the level of objects, $\textbf{Inj}_{\mathsf{ra}}(\mathcal{H})$ and $\textbf{Inj}_{\mathsf{rali}}(\mathcal{H})$ coincide with our notions of Kan injectives. Yet, there is a huge difference on the choice of the $1$-cells, which, in our case, leads to a, in general, non-full sub-2-category.
\end{remark}


\subsection{A comparison between weak Kan-injectivity and Kan-injectivity} \label{comparison}

The following proposition allows us to restrict to left Kan injectivity without losing generality.

\begin{proposition} \label{fromratorali}
Let $\maps$ be a class of maps in a 2-category $\catk$ with bicocomma objects, then there exists a class of maps $\overline{\maps}$ such that $\linj(\maps)=\slinj(\overline{\maps})$.
\end{proposition}

\noindent \textbf{The mapping cone trick.}
Concerning examples \ref{rexinj} and \ref{rexsinj}, we can guess a construction of $\overline{\maps}$ from $\maps$.
Indeed, in that case, from each arrow $\top: D \to 1$ one can obtain  the {\em mapping cone}  $\iota_D\colon D \to \hat{D}$  via the (bi-)cocomma object below.

\[\begin{tikzcd}
	D & 1 \\
	D & {\hat{D}}
	\arrow["\top", from=1-1, to=1-2]
	\arrow[""{name=0, anchor=center, inner sep=0}, "j", dashed, from=1-2, to=2-2]
	\arrow["{\iota_D}"', dashed, from=2-1, to=2-2]
	\arrow[""{name=1, anchor=center, inner sep=0}, Rightarrow, no head, from=1-1, to=2-1]
	\arrow["\rho", shorten <=10pt, shorten >=10pt, Rightarrow, from=1, to=0]
\end{tikzcd}\]
 We show in the proof of Proposition \ref{fromratorali} that this is an instance of a general property. A very similar idea and result appears in  \cite[Section~2]{street2014pointwise}.  

\begin{proof} [of Proposition \ref{fromratorali}]
For every map $h \in \maps$, we construct the mapping cone $C(h)$ over $h$ as the bicocomma object below.
\[\begin{tikzcd}
	A & {A'} \\
	A & {C(h)}
	\arrow["h", from=1-1, to=1-2]
	\arrow[""{name=0, anchor=center, inner sep=0}, "j", dashed, from=1-2, to=2-2]
	\arrow["{i_h}"', dashed, from=2-1, to=2-2]
	\arrow[""{name=1, anchor=center, inner sep=0}, Rightarrow, no head, from=1-1, to=2-1]
	\arrow["\rho", shorten <=12pt, shorten >=12pt, Rightarrow, from=1, to=0]
\end{tikzcd}\]
Then, we define  $\overline{\maps}$ to be the class of all $i_h$ with $h \in \maps$. Let us now show that an object $X$ is weakly left Kan injective with respect to $\maps$ if and only if it is left Kan injective with respect to $\overline{\maps}$. In particular, we will show that an object $X$ is weakly left Kan injective to a $h\in\maps$ if and only if it is left Kan injective to $i_h$.

\begin{enumerate}
\item We start by showing that if $X$ is left Kan injective with respect to $i_h$, then it is weakly left Kan injective with respect to $h$.

Let $f\colon A\to X$ be a 1-cell in $\catk$. Since $X$ is left injective with respect to $i_h$, there exists the left Kan extension $f/i_h$ with the associated 2-cell $\xi^{i_h}_f$ an isomorphism. Then, we can set $f/h:=f/i_h\circ j$ and $\xi^h_f$ as the pasting diagram below:
\[\begin{tikzcd}[ampersand replacement=\&]
	A \&\& {A'} \\
	\&\& {C(h)} \\
	X
	\arrow[""{name=0, anchor=center, inner sep=0}, "j", from=1-3, to=2-3]
	\arrow["h", from=1-1, to=1-3]
	\arrow[""{name=1, anchor=center, inner sep=0}, "{i_h}"{description, pos=0.3}, from=1-1, to=2-3]
	\arrow[""{name=2, anchor=center, inner sep=0}, "f"', from=1-1, to=3-1]
	\arrow["{f/i_h}", dashed, from=2-3, to=3-1]
	\arrow["\rho", shorten <=13pt, shorten >=13pt, Rightarrow, from=1, to=0]
	\arrow["{\xi^{i_h}_f}", shorten <=22pt, shorten >=22pt, Rightarrow, dashed, from=2, to=2-3]
\end{tikzcd}\]
Now, we will show that $f/h$ and $\xi^h_f$ defined in this way satisfy the universal property of the left Kan extension. Let $g\colon A'\to X$ be a 1-cell in $\catk$ together with a 2-cell
\begin{center}
\begin{tikzcd}[ampersand replacement=\&]
	A \& {A'} \\
	\& X
	\arrow["h", from=1-1, to=1-2]
	\arrow[""{name=0, anchor=center, inner sep=0}, "g", from=1-2, to=2-2]
	\arrow[""{name=1, anchor=center, inner sep=0}, "f"', from=1-1, to=2-2]
	\arrow["\beta", shorten <=6pt, shorten >=6pt, Rightarrow, from=1, to=0]
\end{tikzcd} \hspace{0.5cm}$=$\hspace{0.5cm}
\begin{tikzcd}[ampersand replacement=\&]
	A \& {A'} \\
	A \& X
	\arrow[""{name=0, anchor=center, inner sep=0}, Rightarrow, no head, from=1-1, to=2-1]
	\arrow["f"', from=2-1, to=2-2]
	\arrow["h", from=1-1, to=1-2]
	\arrow[""{name=1, anchor=center, inner sep=0}, "g", from=1-2, to=2-2]
	\arrow["\beta", shorten <=13pt, shorten >=13pt, Rightarrow, from=0, to=1]
\end{tikzcd}
\end{center}
By the universal property of the bicocomma object, the 2-cell $\beta$ is equivalent to
\begin{center}
     a 1-cell $\overline{g}\colon C(h)\to X$ and invertible 2-cells \\
\begin{tikzcd}
	&&& {A'} \\
	A & {C(h)} && {C(h)} \\
	&& X && X
	\arrow[""{name=0, anchor=center, inner sep=0}, "f"', curve={height=12pt}, from=2-1, to=3-3]
	\arrow["{\overline{g}}"', from=2-4, to=3-5]
	\arrow[""{name=1, anchor=center, inner sep=0}, "{\overline{g}}", from=2-2, to=3-3]
	\arrow["{i_h}", from=2-1, to=2-2]
	\arrow["j"', from=1-4, to=2-4]
	\arrow[""{name=2, anchor=center, inner sep=0}, "g", curve={height=-12pt}, from=1-4, to=3-5]
	\arrow["\cong"{description}, Rightarrow, draw=none, from=0, to=1]
	\arrow["\cong"{description}, Rightarrow, draw=none, from=2-4, to=2]
\end{tikzcd}
\end{center}
whose pasting with $\rho$ gives $\beta$.
Then, using the universal property of the left Kan extension $f/i_h$, we get that these data is equivalent to have
\begin{center}
    a 1-cell $\overline{g}\colon C(h)\to X$ such that $\overline{g}\circ j\cong g$ and \\
    a 2-cell $\overline{\beta}\colon f/i_h\to\overline{g}$ such that \\
\begin{tikzcd}[ampersand replacement=\&]
	A \&\& {C(h)} \\
	\\
	\&\& X
	\arrow[""{name=0, anchor=center, inner sep=0}, "{\overline{g}}", curve={height=-18pt}, from=1-3, to=3-3]
	\arrow[""{name=1, anchor=center, inner sep=0}, "{f/i_h}"', curve={height=18pt}, from=1-3, to=3-3]
	\arrow["{i_h}", from=1-1, to=1-3]
	\arrow["{\overline{\beta}}"{yshift=0.1cm}, shorten <=8pt, shorten >=8pt, Rightarrow, from=1, to=0]
\end{tikzcd}
\hspace{0.5cm}$=$\hspace{0.5cm}
\begin{tikzcd}[ampersand replacement=\&]
	A \&\& {C(h)} \\
	\\
	{C(h)} \&\& X
	\arrow[""{name=0, anchor=center, inner sep=0}, "f"{description}, from=1-1, to=3-3]
	\arrow["{i_h}", from=1-1, to=1-3]
	\arrow["{\overline{g}}", from=1-3, to=3-3]
	\arrow["{i_h}"', from=1-1, to=3-1]
	\arrow["{f/i_h}"', from=3-1, to=3-3]
	\arrow["{{\xi^{i_h}_f}^{-1}}"{xshift=0.1cm}, shorten <=11pt, shorten >=11pt, Rightarrow, from=3-1, to=0]
	\arrow["\cong"{description}, draw=none, from=0, to=1-3]
\end{tikzcd}
\end{center}
Let us notice that this means that $\overline{\beta}i_h$ is completely determined by the universal 2-cell $\xi_f^{i_h}$ and the isomorphism $f\cong\overline{g}\circ i_h$. Then, using the 2-dimensional property of the bicocomma object we get that $\overline{\beta}$ corresponds to the two 2-cells $\overline{\beta}i_h$ and $\overline{\beta}j$. Therefore, the data above corresponds to
\begin{center}
    a 1-cell $\overline{g}\colon C(h)\to X$ with $\overline{g}\circ j\cong g$ and \\
    a 2-cell $\overline{\beta}_j\colon f/i_h\circ j\to\overline{g}\circ j$ such that \\
    \adjustbox{scale=0.9}{
\begin{tikzcd}
	A && {A'} && A && {A'} \\
	\\
	A && {C(h)} & {=} & A && {C(h)} & {C(h)} \\
	\\
	& {C(h)} && X &&&& X
	\arrow[""{name=0, anchor=center, inner sep=0}, "{\overline{g}}", from=3-3, to=5-4]
	\arrow["{i_h}"{description}, from=3-1, to=3-3]
	\arrow[""{name=1, anchor=center, inner sep=0}, "f"{description}, from=3-1, to=5-4]
	\arrow["h", from=1-1, to=1-3]
	\arrow[""{name=2, anchor=center, inner sep=0}, Rightarrow, no head, from=1-1, to=3-1]
	\arrow[""{name=3, anchor=center, inner sep=0}, "j", from=1-3, to=3-3]
	\arrow["{i_h}"', from=3-1, to=5-2]
	\arrow["{f/i_h}"', from=5-2, to=5-4]
	\arrow[""{name=4, anchor=center, inner sep=0}, Rightarrow, no head, from=1-5, to=3-5]
	\arrow["h", from=1-5, to=1-7]
	\arrow[""{name=5, anchor=center, inner sep=0}, "j"{description}, from=1-7, to=3-7]
	\arrow["{i_h}"', from=3-5, to=3-7]
	\arrow["{f/i_h}"', from=3-7, to=5-8]
	\arrow["j", from=1-7, to=3-8]
	\arrow["{\overline{g}}", from=3-8, to=5-8]
	\arrow["{\overline{\beta}_j}", shorten <=5pt, shorten >=4pt, Rightarrow, from=3-7, to=3-8]
	\arrow["{{\xi^{i_h}_f}^{-1}}"{xshift=0.1cm,yshift=-0.2cm}, shorten <=6pt, shorten >=7pt, Rightarrow, from=5-2, to=1]
	\arrow["\cong", Rightarrow, draw=none, from=1, to=0]
	\arrow["\rho"{yshift=0.1cm}, shorten <=40pt, shorten >=40pt, Rightarrow, from=2, to=3]
	\arrow["\rho"{yshift=0.1cm}, shorten <=25pt, shorten >=25pt, Rightarrow, from=4, to=5]
\end{tikzcd} }\\
\end{center}
Putting together all of these steps we get that, given a 1-cell $g\colon A'\to X$ and a 2-cell $\beta\colon f\to g\circ h$, there exists a unique 2-cell $\Tilde{\beta}\colon f/h\to g$ (with $\tilde{\beta}$ the composition of $\bar{\beta}_j$ with the isomorphism $\bar{g}j\cong g$ above) such that
\[\begin{tikzcd}
	A && {A'} &&& A && {A'} \\
	\\
	A &&&& {=} & A && {C(h)} \\
	\\
	&& X &&&&& X.
	\arrow["f"', curve={height=12pt}, from=3-1, to=5-3]
	\arrow["h", from=1-1, to=1-3]
	\arrow[Rightarrow, no head, from=1-1, to=3-1]
	\arrow[""{name=0, anchor=center, inner sep=0}, Rightarrow, no head, from=1-6, to=3-6]
	\arrow["h", from=1-6, to=1-8]
	\arrow[""{name=1, anchor=center, inner sep=0}, "j"{description}, from=1-8, to=3-8]
	\arrow["{i_h}"', from=3-6, to=3-8]
	\arrow[""{name=2, anchor=center, inner sep=0}, "{f/i_h}"{description}, from=3-8, to=5-8]
	\arrow[""{name=3, anchor=center, inner sep=0}, "f"', curve={height=12pt}, from=3-6, to=5-8]
	\arrow[""{name=4, anchor=center, inner sep=0}, "g", curve={height=-40pt}, from=1-3, to=5-3]
	\arrow[""{name=5, anchor=center, inner sep=0}, "g", curve={height=-40pt}, from=1-8, to=5-8]
	\arrow["\rho"{yshift=0.1cm}, shorten <=26pt, shorten >=26pt, Rightarrow, from=0, to=1]
	\arrow["{{\xi^{i_h}_f}}"{xshift=0.2cm,yshift=0.075cm}, shorten <=12pt, shorten >=12pt, Rightarrow, from=3, to=2]
	\arrow["{\tilde{\beta}}"{xshift=-0.175cm}, shorten <=1pt, shorten >=5pt, Rightarrow, from=3-8, to=5]
	\arrow["\beta", shorten <=34pt, shorten >=34pt, Rightarrow, from=3-1, to=4]
\end{tikzcd}\]

\item Now we show that if $X$ is weakly left Kan injective with respect to $h$, then it is left Kan injective with respect to $i_h$.

Let $f\colon A\to X$ be a 1-cell in $\catk$. Since $X$ is weakly left injective with respect to $h$, there exists the left Kan extension $(f/h,\xi^{h}_f)$. Then, by the universal property of the bicocomma object, there exists a unique (up-to-isomorphism) $f/i_h$ such that
\[\begin{tikzcd}
	A && {A'} &&& A && {A'} \\
	\\
	A &&&& {=} & A && {C(h)} \\
	&&& X &&&&& X,
	\arrow["f"', curve={height=12pt}, from=3-1, to=4-4]
	\arrow["h", from=1-1, to=1-3]
	\arrow[""{name=0, anchor=center, inner sep=0}, Rightarrow, no head, from=1-1, to=3-1]
	\arrow[""{name=1, anchor=center, inner sep=0}, Rightarrow, no head, from=1-6, to=3-6]
	\arrow["h", from=1-6, to=1-8]
	\arrow[""{name=2, anchor=center, inner sep=0}, "j"{description}, from=1-8, to=3-8]
	\arrow["{i_h}"{description}, from=3-6, to=3-8]
	\arrow[""{name=3, anchor=center, inner sep=0}, "{\exists!f/i_h}"{description}, from=3-8, to=4-9]
	\arrow[""{name=4, anchor=center, inner sep=0}, "f"', curve={height=12pt}, from=3-6, to=4-9]
	\arrow[""{name=5, anchor=center, inner sep=0}, "{f/h}", curve={height=-18pt}, from=1-3, to=4-4]
	\arrow[""{name=6, anchor=center, inner sep=0}, "{f/h}", curve={height=-18pt}, from=1-8, to=4-9]
	\arrow["\rho"{yshift=0.1cm}, shorten <=26pt, shorten >=26pt, Rightarrow, from=1, to=2]
	\arrow["{\xi_f^{i_h}}"{xshift=0.1cm}, shorten <=12pt, shorten >=13pt, Rightarrow, from=4, to=3]
	\arrow["\cong", Rightarrow, draw=none, from=3-8, to=6]
	\arrow["{\xi^h_f}"{xshift=-0.2cm,yshift=0.1cm}, shorten <=39pt, shorten >=41pt, Rightarrow, from=0, to=5]
\end{tikzcd}\]
where also $\xi_f^{i_h}$ is an isomorphism. Let us prove now that $f/i_h$ and $\xi^{i_h}_f$ have the universal property of a left Kan extension. 

Let $t\colon C(h)\to X$ be a 1-cell. We want to show that to give a 2-cell $\gamma\colon(f/i_h)i_h\Rightarrow ti_h$ is equivalent to give a 2-cell $\overline{\gamma}\colon f/i_h\Rightarrow t$ with $\overline{\gamma}\hcomp i_h=\gamma$. By the universal property of the bicocomma object, to have a 2-cell   $\overline{\gamma}\colon f/i_h\Rightarrow t$ is equivalent to give 2-cells $\gamma_{i_h}(=\gamma)$ and $\gamma_j$
such that
\[\begin{tikzcd}
	A & {A'} &&& A & {A'} \\
	A & {C(h)} && {=} & A & {C(h)} & {C(h)} \\
	& {C(h)} & X &&&& X.
	\arrow["h", from=1-1, to=1-2]
	\arrow[""{name=0, anchor=center, inner sep=0}, Rightarrow, no head, from=1-1, to=2-1]
	\arrow[""{name=1, anchor=center, inner sep=0}, Rightarrow, no head, from=1-5, to=2-5]
	\arrow["h", from=1-5, to=1-6]
	\arrow[""{name=2, anchor=center, inner sep=0}, "j"', from=1-6, to=2-6]
	\arrow["{i_h}"', from=2-5, to=2-6]
	\arrow["{f/i_h}"', from=2-6, to=3-7]
	\arrow[""{name=3, anchor=center, inner sep=0}, "j", from=1-2, to=2-2]
	\arrow["{i_h}", from=2-1, to=2-2]
	\arrow[""{name=4, anchor=center, inner sep=0}, "t", from=2-2, to=3-3]
	\arrow[""{name=5, anchor=center, inner sep=0}, "{i_h}"', from=2-1, to=3-2]
	\arrow["{f/i_h}"', from=3-2, to=3-3]
	\arrow["j", from=1-6, to=2-7]
	\arrow["t", from=2-7, to=3-7]
	\arrow["{\gamma_j}", shorten <=3pt, shorten >=2pt, Rightarrow, from=2-6, to=2-7]
	\arrow["{\gamma_{i_h}}", shorten <=14pt, shorten >=14pt, Rightarrow, from=5, to=4]
	\arrow["\rho", shorten <=13pt, shorten >=13pt, Rightarrow, from=0, to=3]
	\arrow["\rho", shorten <=13pt, shorten >=13pt, Rightarrow, from=1, to=2]
\end{tikzcd}\]
We show that this is equivalent to give a 2-cell $\gamma=\gamma_{i_h}$ with $\overline{\gamma}i_h=\gamma$, by showing that $\gamma_j$ is determined by $\gamma_{i_h}$. This will complete the proof that $X$ is left Kan injective with respect to $i_h$.
Indeed, pasting with $\xi_f^{i_h}$, expanding the identity on $f/i_h\circ j$ through the isomorphism $f/i_h\circ j\cong f/h$,  and using the definition of $f/i_h$, we obtain the following equality
\[\begin{tikzcd}[ampersand replacement=\&]
	A \& {A'} \&\& A \& {A'} \\
	A \& {C(h)} \& {=} \& A \&\& {C(h)} \& {C(h)} \\
	\& {C(h)} \&\&\&\&\& {} \\
	\&\& X \&\& X
	\arrow["h", from=1-1, to=1-2]
	\arrow[""{name=0, anchor=center, inner sep=0}, Rightarrow, no head, from=1-1, to=2-1]
	\arrow[Rightarrow, no head, from=1-4, to=2-4]
	\arrow["h", from=1-4, to=1-5]
	\arrow[""{name=1, anchor=center, inner sep=0}, "j", from=1-2, to=2-2]
	\arrow["{i_h}", from=2-1, to=2-2]
	\arrow[""{name=2, anchor=center, inner sep=0}, "t", curve={height=-18pt}, from=2-2, to=4-3]
	\arrow["{i_h}"{description}, from=2-1, to=3-2]
	\arrow["{f/i_h}"{description}, from=3-2, to=4-3]
	\arrow["j", curve={height=-6pt}, from=1-5, to=2-7]
	\arrow["t", curve={height=-12pt}, from=2-7, to=4-5]
	\arrow["j"{description}, from=1-5, to=2-6]
	\arrow["{f/i_h}"{description}, from=2-6, to=4-5]
	\arrow[""{name=3, anchor=center, inner sep=0}, "{f/h}"{description}, from=1-5, to=4-5]
	\arrow["f"', curve={height=6pt}, from=2-4, to=4-5]
	\arrow[""{name=4, anchor=center, inner sep=0}, "f"', curve={height=40pt}, from=2-1, to=4-3]
	\arrow["{\gamma_j}"'{pos=0.3}, shorten <=5pt, shorten >=30pt, Rightarrow, from=2-6, to=3-7]
	\arrow["\rho", shorten <=15pt, shorten >=15pt, Rightarrow, from=0, to=1]
	\arrow["\cong"{description}, draw=none, from=3, to=2-6]
	\arrow["{\gamma_{i_h}}", shorten <=5pt, shorten >=9pt, Rightarrow, from=3-2, to=2]
	\arrow["{\xi^{i_h}_f}"{pos=0.6}, shorten <=4pt, shorten >=0pt, Rightarrow, from=4, to=3-2]
	\arrow["{\xi^h_f}"{pos=0.3}, shorten <=8pt, shorten >=13pt, Rightarrow, from=2-4, to=3]
\end{tikzcd}\]
showing that $\gamma_j$ is determined by $\gamma=\gamma_{i_h}$ via the universality of the left Kan extension of $f$ along $h$.
\end{enumerate}

Finally, using the description of the Kan extensions given above, we can see the equality for 1-cells as well.

 Let $p\colon X\to Y$ be left Kan injective with respect to $\overline{\maps}$. Then, for any $h\colon A\to A'\in\maps$ and $f\colon A\to X\in\catk$,
    \begin{center}
    \begin{tabular}{rlr}
         $p\circ f/h$
         & $\cong p\circ f/i_h\circ j$
         & (by construction above)\\
         &  $\cong (pf)/i_h\circ j$
         & (because $p\in\slinj(\overline{\maps})$)\\
         & $\cong (pf)/h$
         & (by construction above).\\
    \end{tabular}
    \end{center}
     On the other hand, let us consider $p\colon X\to Y\in\linj(\maps)$. For any $i_h\in\overline{\maps}$ and any $f\colon A\to X\in\catk$, through the universal property of the co-comma object $C(h)$,
    \begin{center}
      $p\circ f/i_h$ corresponds to $p\circ f/h$ \\
      and $(pf)/i_h$ to $(pf)/h$.
    \end{center}
    Since $p\in\linj(\maps)$, we get $p\circ f/h\cong(pf)/h$ and therefore $p\circ f/i_h\cong (pf)/i_h$.
\end{proof}

\subsection{Saturated classes}\label{subsec:saturation}
Kan injectivity determines a Galois connection between locally full sub-2-categories and classes of 1-cells. More precisely, given a locally full sub-2-category $\mathcal{A}$, denote by $\cata^{\LInj}$ the class of all 1-cells with respect to which all objects and 1-cells of $\cata$ are left Kan injective. Then,  we have that $\mathcal{A}\subseteq \mathcal{B}$ implies  $\mathcal{B}^{\slinj}\subseteq \mathcal{A}^{\LInj}$; we also have that $\maps\subseteq \mathcal{I}$ implies $\LInj(\mathcal{I})\subseteq \LInj(\maps)$, and \[\cata^{\LInj}\subseteq \maps \text{ if and only if } \cata\subseteq \LInj(\maps).\] These considerations justify the definition below.

\begin{definition}
The \textit{saturation of $\maps$} with respect to Kan-injectivity is defined by,
$$\mathcal{H}^{\text{sat}}:=\left(\LInj(\mathcal{H})\right)^{\LInj}.$$
\end{definition}

It follows from the previous discussion that we have $\LInj(\mathcal{H}^{\text{sat}})=\LInj(\mathcal{\mathcal{H}})$. The following proposition shows that $\mathcal{H}^{\text{sat}}$ is closed under certain constructions. This result will be used along the paper, in particular, in Lemma \ref{lem:ext-to-P} and Proposition \ref{prop:distr-KZ-KZ}.

\begin{proposition}
\label{prop:saturation-irr}
$\maps^{\text{sat}}$ is closed under the following constructions:
\begin{enumerate}

    \item  (Laris) Any lari 1-cell $l\colon A\to B$ belongs to $\maps^\sat$.
    
    \item (Isomorphisms) If $h\in\maps^\sat$ and there exists an isomorphism $h\cong h'$, then $h'\in\maps^\sat$.
    
    \item (Compositions) Given a pair of composable 1-cells $f\colon A\to B$ and $g\colon B\to C$, if $f,g\in\maps^\sat$, then $gf\in\maps^\sat$.

    \item (Reflections) If $h\in\maps^{\text{sat}}$ and there are pseudocommutative squares 
    \begin{equation}\label{ralis}
\begin{tikzcd}[ampersand replacement=\&]
	A \& B \\
	{A'} \& {B'}
	\arrow[""{name=0, anchor=center, inner sep=0}, "h", from=1-2, to=2-2]
	\arrow[""{name=1, anchor=center, inner sep=0}, "s"', from=1-1, to=2-1]
	\arrow["{l_1}"', from=1-2, to=1-1]
	\arrow["{l_2}", from=2-2, to=2-1]
	\arrow["\cong"{description}, draw=none, from=1, to=0]
\end{tikzcd} \hspace{0.5cm}\mbox{and} \hspace{0.5cm}
\begin{tikzcd}[ampersand replacement=\&]
	A \& B \\
	{A'} \& {B'}
	\arrow[""{name=0, anchor=center, inner sep=0}, "h", from=1-2, to=2-2]
	\arrow[""{name=1, anchor=center, inner sep=0}, "s"', from=1-1, to=2-1]
	\arrow["{r_1}", from=1-1, to=1-2]
	\arrow["{r_2}"', from=2-1, to=2-2]
	\arrow["\cong"{description}, draw=none, from=1, to=0]
\end{tikzcd}
    \end{equation}  
    where $l_1$ and $l_2$ are laris with right adjoints $r_1$ and $r_2$, respectively,
     then $s\in\maps^{\text{sat}}$.

\item (Bicocomma objects and  bipushouts).
If in
\begin{equation}\label{square}
\begin{tikzcd}
	A & {A'} \\
	B & {C}
	\arrow["h", from=1-1, to=1-2]
	\arrow["{\overline{h}}"', from=2-1, to=2-2]
	\arrow[""{name=0, anchor=center, inner sep=0}, "r"', from=1-1, to=2-1]
	\arrow[""{name=1, anchor=center, inner sep=0}, "s", from=1-2, to=2-2]
	\arrow[shorten <=13pt, shorten >=13pt, Rightarrow, from=0, to=1]
\end{tikzcd}\end{equation}
 $h\in\maps^{\text{sat}}$,  then $\overline{h}\in\mathcal{H}^{\text{sat}}$, provided that \eqref{square} is a bicocomma object or an invertible 2-cell forming a bipushout.
 \item  (Wide bipushouts). If the diagram
 $$\xymatrix{A\ar[dr]_h^{\hspace{12.5pt}\cong}\ar[rr]^{h_i}&&A_i\ar[ld]^{d_i}\\&B&}$$
 represents a wide bipushout of a family of 1-cells $h_i$ with all of them in $\maps^{\text{sat}}$, then $h\in\maps^{\text{sat}}$.
\end{enumerate}
\end{proposition}

\begin{proof}
\begin{enumerate}

    \item \textbf{Laris:} For any $X\in\catk$ and any lari 1-cell $l\colon A\to B$, we want to show that $X$ is Kan injective with respect to $l$, i.e. $\catk(l,X)$ is rali. This is true because the 2-functor $\catk(-,X)$ send lari 1-cells to rali 1-cells (see \cite[Remark~I,6.5]{Gray:Adj}). 
    
    \item \textbf{Isomorphisms:} Clearly, for any $X\in\catk$, if $h\cong h'$, then also $\catk(h,X)\cong\catk(h',X)$. Hence, if $X$ is Kan injective with respect to $h$, then $X$ is also Kan injective with respect to $h'$. 
    
    \item \textbf{Composition:} This follows since the composition of ralis is a rali.

    \item \textbf{Reflections:} Let us consider the pseudocommutative squares (\ref{ralis}),
 and let $X$ be left Kan injective with respect to $h$, i.e. $h^\ast:=\catk(h,X)$ is a rali. We want to show that  $\catk(s,X)$ is a rali as well. Applying $\catk(-,X)$ to the pseudocommutative square with $l_1$ and $l_2$ we get the pseudocommutative square below. 
\[\begin{tikzcd}[ampersand replacement=\&]
	{\catk(A,X)} \&\& {\catk(B,X)} \\
	\\
	{\catk(A',X)} \&\& {\catk(B',X)}
	\arrow[""{name=0, anchor=center, inner sep=0}, "{h^\ast}", from=3-3, to=1-3]
	\arrow[""{name=1, anchor=center, inner sep=0}, "{s^\ast}", from=3-1, to=1-1]
	\arrow[""{name=2, anchor=center, inner sep=0}, "{{l_1}^\ast}"', from=1-1, to=1-3]
	\arrow[""{name=3, anchor=center, inner sep=0}, "{{l_2}^\ast}", from=3-1, to=3-3]
	\arrow[""{name=4, anchor=center, inner sep=0}, "{(-)/h=:\overline{h}}", curve={height=-24pt}, from=1-3, to=3-3]
	\arrow[""{name=5, anchor=center, inner sep=0}, "{{r_1}^\ast}"', curve={height=24pt}, from=1-3, to=1-1]
	\arrow[""{name=6, anchor=center, inner sep=0}, "{{r_2}^\ast}", curve={height=-24pt}, from=3-3, to=3-1]
	\arrow["{?}"', curve={height=24pt}, dashed, from=1-1, to=3-1]
	\arrow["\cong"{description}, draw=none, from=1, to=0]
	\arrow["\vdash"{description}, draw=none, from=0, to=4]
	\arrow["\top"{description}, draw=none, from=3, to=6]
	\arrow["\perp"{description}, draw=none, from=2, to=5]
\end{tikzcd}\]
    We want to find a left adjoint to $s^\ast$ with invertible unit. We claim that  ${r_2}^\ast\circ \overline{h}\circ {l_1}^\ast$ is the required left adjoint. Let us consider two maps $g\colon A\to X$ and $g'\colon A'\to X$, then,
    \begin{center}
    {\renewcommand{\arraystretch}{1.5}
    \begin{tabular}{rcl}
        ${r_2}^\ast\circ \overline{h}\circ {l_1}^\ast g$
        & $\longrightarrow$ & $g'$ \\
        \hline
        $\overline{h}{l_1}^\ast g$
        & $\longrightarrow$ & ${l_2}^\ast g'$ \\
        \hline
        ${l_1}^\ast g$
        & $\longrightarrow$ & $h^\ast{l_2}^\ast g'$ \\
        \hline
        ${l_1}^\ast g$
        & $\longrightarrow$ & ${l_1}^\ast s^\ast g'$ \\
        \hline
        $g$
        & $\longrightarrow$ & $s^\ast g'$ \\
    \end{tabular}
    }%
    \end{center}
    We note that in this chain of bijections we used two adjunctions, the isomorphisms  ${l_1}^\ast s^\ast\cong h^\ast{l_1}^\ast$ and that since  $r_1$ is rari, then ${l_1}^\ast$ is fully faithful. Clearly this bijection is natural, so we have left to check only that the unit of this adjunction is invertible.
     Setting $g':={r_2}^\ast\overline{h}{l_1}^\ast g$, following the bijections above we obtain
   \begin{center}
       \begin{tabular}{rlr}
         $g\to s^\ast{r_2}^\ast\overline{h}{l_1}^\ast g$
         & $\cong {r_1}^\ast h^\ast\overline{h}{l_1}^\ast g$
         & (by $r_2s\cong hr_1$)\\
         & $\cong {r_1}^\ast {l_1}^\ast g$
         & (by $h^\ast$ rali)\\
         & $\cong g$
         & (by $r_1$ rari).
    \end{tabular}
   \end{center}

    \item \textbf{Bipushouts:}
Consider the diagram below, we want to show that if $X$ is Kan injective with respect to $h$, and the square \eqref{square} is a bipushout, then $X$ is also Kan injective with respect to $k$. The diagram shows how to construct the candidate Kan extension of $s$ using the universal property of the bipushout.

\[\begin{tikzcd}
	A & \bullet && A & \bullet && A & \bullet \\
	{A'} & \bullet && {A'} &&& {A'} & \bullet \\
	&& X &&& X &&& X
	\arrow[""{name=0, anchor=center, inner sep=0}, "h"', from=1-1, to=2-1]
	\arrow["f", from=1-1, to=1-2]
	\arrow[""{name=1, anchor=center, inner sep=0}, "{k}", from=1-2, to=2-2]
	\arrow["{g}"', from=2-1, to=2-2]
	\arrow["s"{description}, curve={height=-12pt}, from=1-2, to=3-3]
	\arrow["f", from=1-4, to=1-5]
	\arrow["h"', from=1-4, to=2-4]
	\arrow[""{name=2, anchor=center, inner sep=0}, "s"{description}, curve={height=-12pt}, from=1-5, to=3-6]
	\arrow["{(sf)/h}"', curve={height=12pt}, dashed, from=2-4, to=3-6]
	\arrow[""{name=3, anchor=center, inner sep=0}, "{k}", from=1-8, to=2-8]
	\arrow[""{name=4, anchor=center, inner sep=0}, "h"', from=1-7, to=2-7]
	\arrow[from=2-7, to=2-8]
	\arrow["f", from=1-7, to=1-8]
	\arrow[""{name=5, anchor=center, inner sep=0}, "{(sf)/h}"{description}, curve={height=12pt}, from=2-7, to=3-9]
	\arrow[""{name=6, anchor=center, inner sep=0}, "s"{description}, curve={height=-12pt}, from=1-8, to=3-9]
	\arrow["{s/k}",dashed, from=2-8, to=3-9]
	\arrow["\cong"{description}, Rightarrow, draw=none, from=0, to=1]
	\arrow["\cong"{description}, Rightarrow, draw=none, from=4, to=3]
	\arrow["\cong"{description}, Rightarrow, draw=none, from=2-8, to=6]
	\arrow["\cong"{description}, Rightarrow, draw=none, from=2-8, to=5]
	\arrow["\cong"{description}, Rightarrow, draw=none, from=2-4, to=2]
\end{tikzcd}\]

If we follow this approach to show the univeral property of the Kan extension the proof would be very technical. Instead, we follow a more formal approach. In the diagram below, the situation above is formulated in terms of $h^*$ having a left adjoint. Recall that the diagram in the middle must be a bipullback, and we can thus construct the dashed functor on the right.

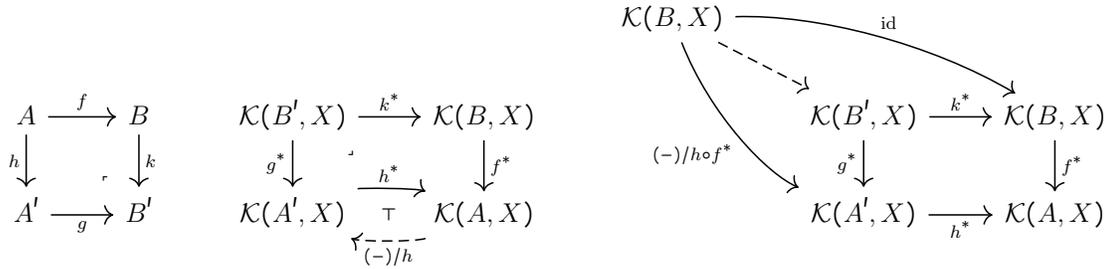
\[  \adjustbox{scale=0.85}{\begin{tikzcd}[ampersand replacement=\&]
	\&\&\&\& {\catk(B,X)} \\
	A \& B \& {\catk(B',X)} \& {\catk(B,X)} \&\& {\catk(B',X)} \& {\catk(B,X)} \\
	{A'} \& {B'} \& {\catk(A',X)} \& {\catk(A,X)} \&\& {\catk(A',X)} \& {\catk(A,X)}
	\arrow["h"', from=2-1, to=3-1]
	\arrow["f", from=2-1, to=2-2]
	\arrow["k", from=2-2, to=3-2]
	\arrow["g"', from=3-1, to=3-2]
	\arrow[""{name=0, anchor=center, inner sep=0}, "{h^*}", curve={height=-12pt}, from=3-3, to=3-4]
	\arrow["{k^*}", from=2-3, to=2-4]
	\arrow["{f^*}", from=2-4, to=3-4]
	\arrow["{g^*}"', from=2-3, to=3-3]
	\arrow[""{name=1, anchor=center, inner sep=0}, "{(-)/h}", curve={height=-12pt}, dashed, from=3-4, to=3-3]
	\arrow["\lrcorner"{anchor=center, pos=0.125, rotate=180}, draw=none, from=3-2, to=2-1]
	\arrow["\lrcorner"{anchor=center, pos=0.125}, draw=none, from=2-3, to=3-4]
	\arrow["{\text{id}}", curve={height=-18pt}, from=1-5, to=2-7]
	\arrow["{h^*}"', from=3-6, to=3-7]
	\arrow["{f^*}", from=2-7, to=3-7]
	\arrow["{k^*}", from=2-6, to=2-7]
	\arrow["{g^*}"', from=2-6, to=3-6]
	\arrow["{(-)/h \circ f^*}"', curve={height=18pt}, from=1-5, to=3-6]
	\arrow[dashed, from=1-5, to=2-6]
	\arrow["\dashv"{anchor=center, rotate=90}, draw=none, from=1, to=0]
\end{tikzcd}}\]

We now want to show that the dashed arrow provides a left adjoint for $k^*$. We shall call $(-)//k$ the dashed functor. By the universal property of the bipullback, we already have the invertible map $\text{1} \to k^* \circ (-)//k$, which will be our unit. To construct the counit, we consider the diagram below, and use the $2$-dimensional part of the universal property of the bipullback to obtain the desired $2$-cell $(-)//k \circ k^*\to \text{1}$.

  \[\adjustbox{scale=0.7}
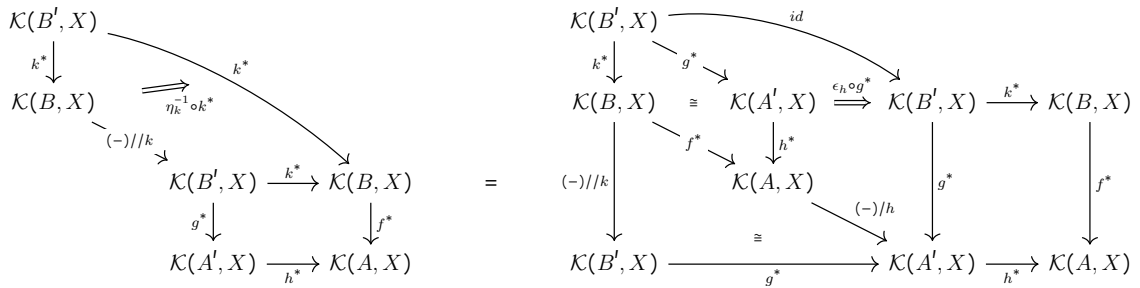
{\begin{tikzcd}[ampersand replacement=\&]
	{\catk(B',X)} \&\&\&\& {\catk(B',X)} \\
	{\catk(B,X)} \&\&\&\& {\catk(B,X)} \& {\catk(A',X)} \& {\catk(B',X)} \& {\catk(B,X)} \\
	\& {\catk(B',X)} \& {\catk(B,X)} \& {=} \&\& {\catk(A,X)} \\
	\& {\catk(A',X)} \& {\catk(A,X)} \&\& {\catk(B',X)} \&\& {\catk(A',X)} \& {\catk(A,X)}
	\arrow[""{name=0, anchor=center, inner sep=0}, "{k^\ast}", curve={height=-18pt}, from=1-1, to=3-3]
	\arrow["{h^*}"', from=4-2, to=4-3]
	\arrow["{f^*}", from=3-3, to=4-3]
	\arrow["{k^*}", from=3-2, to=3-3]
	\arrow["{g^*}"', from=3-2, to=4-2]
	\arrow["{k^\ast}"', from=1-1, to=2-1]
	\arrow["{(-)//k}"{description}, from=2-1, to=3-2]
	\arrow["id", curve={height=-24pt}, from=1-5, to=2-7]
	\arrow["{k^*}", from=2-7, to=2-8]
	\arrow["{f^*}",from=2-8, to=4-8]
	\arrow["{g^*}", from=2-7, to=4-7]
	\arrow[""{name=1, anchor=center, inner sep=0}, "{k^\ast}"', from=1-5, to=2-5]
	\arrow["{g^*}"', from=4-5, to=4-7]
	\arrow["{(-)//k}"', from=2-5, to=4-5]
	\arrow["{f^*}"{description}, from=2-5, to=3-6]
	\arrow[""{name=2, anchor=center, inner sep=0}, "{(-)/h }", from=3-6, to=4-7]
	\arrow["{g^*}"{description}, from=1-5, to=2-6]
	\arrow[""{name=3, anchor=center, inner sep=0}, "{h^*}", from=2-6, to=3-6]
	\arrow["{\epsilon_h\circ g^\ast}"{yshift=0.075cm}, shorten <=4pt, shorten >=4pt, Rightarrow, from=2-6, to=2-7]
	\arrow["{h^*}"', from=4-7, to=4-8]
	\arrow["{\eta_k^{-1}\circ k^\ast}"'{xshift=-0.1cm,yshift=-0.05cm}, shorten <=20pt, shorten >=20pt, Rightarrow, from=2-1, to=0]
	\arrow["\cong"{description}, draw=none, from=4-5, to=2]
	\arrow["\cong"{description}, draw=none, from=1, to=3]
\end{tikzcd}}\]

Moreover, given a 1-cell $p:X\to X'$ which is left Kan injective with respect to $h$, using the construction above of $(-)/k:=(-)//k$ and Remark \ref{rem:Beck-Chev}, we conclude that $p$ is also left Kan injective with respect to $k$.

   \noindent \textbf{Bicocomma objects:} We follow the same argument of the second part of  Proposition \ref{fromratorali}. Indeed, in the notation of that proposition, if $X$ was Kan injective with respect to $h$ (as opposed to weak Kan injective) the result is true a fortiori. Also, in the proof we never use the fact that the $1$-cell $A \to A$ is the identity, it could be any $1$-cell. This delivers the proof.

   \item     \textbf{Wide bipushouts:}
       The proof is completely similar to the one for bipushouts. Using the left Kan injectivity of $X$ with respect to all $h_i$ by means of the hom-functor $\catk(-,X)$, we obtain a wide bipullback and, as a consequence, a left adjoint of $\catk(h,X)$ making it a rali:
       $$\xymatrix{\catk(A_i,X)\ar[rr]^{\catk(h_i,X)}&&\catk(A,X)\\
       &\catk(B,X)\ar[ru]^{\catk(d_i,X)}\ar[lu]_{\catk(h,X)}&
   \\&\catk(A,X)\ar@/^4ex/[luu]^{(-)/h_i}\ar@/_4ex/[ruu]_{\text{id}}\ar@{-->}[u]_{(-)/h}&
       }$$
       That is, for each $s:B\to X$, the 1-cell $s/h$ is obtained by the universality of the wide bipushout:
   \begin{equation}\label{DD2}
\begin{tikzcd}
	{A} && {A_i} \\
	& {B} \\
	& X
	\arrow[""{name=0, anchor=center, inner sep=0}, "{h_i}", from=1-1, to=1-3]
	\arrow[""{name=1, anchor=center, inner sep=0}, "{s}"', curve={height=12pt}, from=1-1, to=3-2]
	\arrow[""{name=2, anchor=center, inner sep=0}, "{s/h_i}", curve={height=-12pt}, from=1-3, to=3-2]
	\arrow["{h}"{description}, from=1-1, to=2-2]
	\arrow["{s/h}"{description}, from=2-2, to=3-2]
	\arrow["{d_i}"{description}, from=1-3, to=2-2]
	\arrow["\cong"{description}, Rightarrow, draw=none, from=0, to=2-2]
	\arrow["\cong"{description}, Rightarrow, draw=none, from=1, to=2-2]
	\arrow["\cong"{description}, Rightarrow, draw=none, from=2-2, to=2]
\end{tikzcd}
\end{equation}
\end{enumerate}
\end{proof}


\section{KZ-pseudomonads presented via Kan-injectivity} \label{sec2}

Idempotent monads over a category $\cc$ are precisely those whose categories of algebras are full reflective subcategories of $\cc$. Thus, an idempotent monad may be presented by orthogonality with respect to the family  $(\delta_X\colon X\to \bar{X})_{X\in\cc}$ of reflections into the corresponding reflective subcategory. In this section, we see that, analogously, a KZ-monad may be presented by left Kan injectivity with respect to a family of 1-cells $(\delta_X\colon X\to \bar{X})_{X\in\catk}$, where every $\bar{X}$ is essentially a pseudoalgebra. These facts will have an important role in Section \ref{TheTheorem}.

We recall from \cite{kock_1995} that a KZ-pseudomonad, also known as a lax-idempotent pseudomonad or KZ-doctrine, can be described as a pseudomonad with unit $\unit$ and  multiplication $\mu$
 such that $\mu$ is a right adjoint to $T\unit$ (and a left adjoint to $\unit_T$) with convenient coherence relations. By a \textit{KZ-adjunction} we mean a biadjunction whose induced pseudomonad is a KZ-pseudomonad.

The next theorem is essentially contained in \cite{marm-wood_lax-idemp}, as we explain in the proof. For the particular case of order-enriched categories, it was given in \cite[Theorem~3.4]{carv-sousa_order-refl}. {\blue Concerning the notion of density for  1-cells used in that result, we give the following lemma.}


\blue
\begin{lemma}
\label{lem:iso-dense-iff-dense}
    Let $\catk$ be a 2-category and $f\colon X\to Y$ a 1-cell in $\catk$. Then, the following are equivalent:
    \begin{enumerate}[label=(\arabic*)]
        \item There exists an invertible 2-cell $\xi\colon f \Rightarrow f$ making $(1_Y,\xi)$ a Kan extension of $f$ along itself.
        \item The pair $(1_Y,1_f)$ is a Kan extension of $f$ along itself. 
    \end{enumerate} 
\end{lemma}

\begin{proof} 

    One can check that $(1)$ is saying that there exists an invertible 2-cell $\xi$ making the composite function below bijective, while $(2)$ means that the first function is such. 
\[\begin{tikzcd}[ampersand replacement=\&]
	{\catk(Y,Y)[1_Y,g]} \& {\catk(X,Y)[f,gf]} \& {\catk(X,Y)[f,gf]}
	\arrow["{-\circ f}", from=1-1, to=1-2]
	\arrow["{-\cdot \xi}", from=1-2, to=1-3]
\end{tikzcd}\]
Clearly $(2)$ implies $(1)$. For the other implication, we just need to notice that, since $\xi$ is invertible, the function $-\cdot\xi$ is a bijection. From this it follows that $-\circ f$ is bijective as it can be written as composite of bijective maps. \black  
\end{proof}


\begin{definition}\label{def:dense} We call a 1-cell \textbf{dense} if it satisfies any of the two equivalent conditions in Lemma~\ref{lem:iso-dense-iff-dense}.\end{definition}

\black 

\begin{theorem} \label{thm:marm-wood}
\begin{enumerate}

    \item[(1)] Let $\cata$ be a locally full (and locally replete) sub-2-category of the 2-category $\catk$, and let
\begin{center}
    $d_X\colon X\to DX$,  $X\in\catk$,
\end{center}
be a family of 1-cells with $\cata\subseteq \slinj(\lbrace d_X\colon X\to DX\mid X\in\catk\rbrace)$ and such that:
\begin{enumerate}
    \item[(a)] For all $X\in\catk$, $DX\in\cata$, and,  for every $f\colon X\to A$ with $A\in\cata$, $f/d_X\in\cata$.
    \item[(b)] Every $d_X$ is dense. 
\end{enumerate}
Then, the inclusion $\cata\hookrightarrow\catk$  is the right part of a  KZ-adjunction in $\catk$.
    \item[(2)] Conversely, every KZ-pseudomonad $\mathbb{D}$ may be induced by the data in (1) where $d\colon\text{Id}_{\catk}\to D$ is the unit.
\end{enumerate}

\end{theorem}

\begin{remark}\label{rem:dense}
  Under assumption (a), condition (b) is equivalent to the following condition used in  \cite{carv-sousa_order-refl} in the Ord-enriched case:
\begin{enumerate}
\item[(b$^{\prime}$)] $(f d_X)/d_X \cong f$ for all $f\colon DX\to A$ in $\cata$.
\end{enumerate}
Indeed, assuming (b), with $f\in \cata$, we have that $(fd_X)/d_X\cong f(d_X/d_X)\cong f$.
\end{remark}

\begin{proof}[of Theorem \ref{thm:marm-wood}]
\begin{enumerate}
    \item[(1)] Recall, from  \cite[Definition~3.1]{marm-wood_lax-idemp}, that a \emph{left Kan pseudomonad}  $\mathbb{D}$ consists of the following data:
\begin{enumerate}
\item
for every  $X\in\catk$, a 1-cell
$$d_X\colon X\to DX;$$
    \item
for every 1-cell $f\colon X\to DY$, a left Kan extension of $f$ along $d_X$
\[\begin{tikzcd}[ampersand replacement=\&]
	X \&\& DX \\
	\&\& DY
	\arrow["{d_X}", from=1-1, to=1-3]
	\arrow[""{name=0, anchor=center, inner sep=0}, "f"', from=1-1, to=2-3]
	\arrow[""{name=1, anchor=center, inner sep=0}, "{f^\mathbb{D}}", from=1-3, to=2-3]
	\arrow["{\mathbb{D}_f}", shorten <=12pt, shorten >=12pt, Rightarrow, from=0, to=1]
\end{tikzcd}\]
with $\mathbb{D}_f$ invertible;
    \item for every $f\colon X\to Y$ and $g\colon Z\to DX$, $(f^\mathbb{D}\circ g)^{\mathbb{D}}\cong f^\mathbb{D}\circ g^{\mathbb{D}}$;
    \item every $d_X$ is dense. 
\end{enumerate}

Marmolejo and Wood proved in \cite[Theorem~4.1]{marm-wood_lax-idemp} that this data induces a KZ-pseudomonad $\mathbb{D}=(D,d,m)$.\footnote{Marmolejo and Wood studied the dual situation: \emph{right} Kan and \emph{colax-idempotent} pseudomonads.} Following the proof of their theorem, we see that the given $D$ is extended to the endo-pseudofunctor $D\colon\catk\to \catk$, and $d$ is extended to a strong transformation which is the unit of the pseudomonad.
It is clear that, under the hypotheses of our  Theorem~\ref{thm:marm-wood}, the family $d_X, X\in \catk$, fulfils the conditions defining a left Kan pseudomonad, where $f^{\mathbb{D}}$ is  an existing left Kan extension $f/d_X$.  {\blue For (c), observe that $(f^{\mathbb{D}}\circ g)^{\mathbb{D}}\cong (f/d_X\circ g)/d_Z\cong (f/d_X)\circ (g/d_Z)\cong f^{\mathbb{D}}\circ g^{\mathbb{D}}$, the second isomorphism due to  $f/d_X$ belonging to $\cata\subseteq \LInj(\{d_X\mid X\in \catk\})$.} The pseudofunctor $D\colon \catk\to \catk$ is defined on 1-cells by $Df=(d_Y\circ f)^{\mathbb{D}}\cong (d_Y\circ f)/d_X$, which lies in $\cata$. Thus, $D$ admits a corestriction $D_{\cata}$ to $\cata$. Moreover, from Remark~\ref{rem:dense}, for every $f\colon X\to A$ with $A\in \cata$, the morphism  $f/d_X\colon DX\to A$ is the unique 1-cell of $\cata$, up to isomorphism, such that $(f/d_X)\circ d_X \cong f$. 
 {\blue Since, for every $f\colon X\to A$ with $A\in \cata$, $f/d_X$ belongs to $\cata$, we obtain, via Kan extensions, an adjunction between the hom-categories $\catk(X,A)$ and $\cata(DX, A)$. The unit and counit of this adjunction are invertible, the first one by definition of Kan injectivity, the second one by condition (b$^{\prime}$) described in Remark~\ref{rem:dense}.  Thus, the adjunction is  indeed an equivalence (pseudonatural in $X$ and $A$), and we  have a biadjunction between $\cata$ and $\catk$. More precisely,} the inclusion functor of $\cata$ into $\catk$ is the right 2-functor of a KZ-adjunction
\[\begin{tikzcd}
	\cata && \catk
	\arrow[""{name=0, anchor=center, inner sep=0}, curve={height=12pt}, hook, from=1-1, to=1-3]
	\arrow[""{name=1, anchor=center, inner sep=0}, "{D_{\cata}}"', curve={height=12pt}, from=1-3, to=1-1]
	\arrow["\perp"{description}, Rightarrow, draw=none, from=1, to=0]
\end{tikzcd}\]
whose induced pseudomonad is $\mathbb{D}$.

\item[(2)]  This is \cite[Theorem 4.2]{marm-wood_lax-idemp}. 
\end{enumerate} 
\end{proof}

\begin{remark}
    \blue 
    We underline that in \cite[Definition~3.1]{marm-wood_lax-idemp} Marmolejo and Wood use the definition of density with the identity 2-cell as part of the Kan extension, whereas in our main result (Theorem~\ref{thm:linj-kz-adj}) we are going to use condition (2) of Lemma~\ref{lem:iso-dense-iff-dense}.
\end{remark}

The next theorem describes the category of pseudoalgebras of a KZ-pseudomonad by means of left Kan injectivity. {\blue We will make use of the following lemma} (proved in \cite[Proposition 2.13]{carv-sousa_order-refl} for the particular case of order-enriched categories), which
shows how Kan injectivity interacts with lali 1-cells.

\begin{lemma}
\label{lem:left-inj-vs-lali}
Every sub-2-category $\LInj(\maps)$ is closed under lalis, that is:
 for any pseudocommutative diagram
\[\begin{tikzcd}
	A & B \\
	X & {Y,}
	\arrow["f", from=1-1, to=1-2]
	\arrow["g"', from=2-1, to=2-2]
	\arrow[""{name=0, anchor=center, inner sep=0}, "{l_1}"', from=1-1, to=2-1]
	\arrow[""{name=1, anchor=center, inner sep=0}, "{l_2}", from=1-2, to=2-2]
	\arrow["\cong"{description}, Rightarrow, draw=none, from=0, to=1]
\end{tikzcd}\]
with $f$ a 1-cell of $\slinj(\maps)$ and $l_1,l_2$ lalis,
    then also $g$ belongs to $\slinj(\maps)$.
\end{lemma}

\begin{proof} 
We first show that $X$ belongs to $\LInj(\mathcal{H})$. Given any $h\colon C\to C'$ in $\maps$ and any $p\colon C\to X$, we need to prove that there exists a Kan extension $p/h$ with an invertible universal 2-cell. Since $A$ is left Kan injective with respect to $h$ we can consider the following 2-cell, where  $l:=l_1\dashv r$   and $\epsilon$ is the counit of the adjunction (which is an isomorphism since $l$ is a lali).
\[\begin{tikzcd}
	C &&& {C'} \\
	& X \\
	&& A \\
	X
	\arrow["h", from=1-1, to=1-4]
	\arrow["p"{description}, from=1-1, to=2-2]
	\arrow["r"{description}, from=2-2, to=3-3]
	\arrow["p"', from=1-1, to=4-1]
	\arrow[""{name=0, anchor=center, inner sep=0}, Rightarrow, no head, from=2-2, to=4-1]
	\arrow["l", curve={height=-6pt}, from=3-3, to=4-1]
	\arrow[""{name=1, anchor=center, inner sep=0}, "{(rp)/h}", curve={height=-6pt}, from=1-4, to=3-3]
	\arrow["{\epsilon^{-1}}", shorten <=12pt, shorten >=12pt, Rightarrow, from=0, to=3-3]
	\arrow["{\xi^h_{rp}}", shorten <=14pt, shorten >=14pt, Rightarrow, from=2-2, to=1]
\end{tikzcd}\]
    The pasting diagram makes $l\circ (rp)/h$ a left Kan extension of $p$ along $h$ (with universal 2-cell invertible, since $\xi^h_{rp}$ is so):
    \begin{center}
    \begin{tabular}{rlr}
         $l\circ (rp)/h$
         & $\cong (lrp)/h$
         & (since left adjoints preserves left Kan extensions)\\
         & $\cong p/h$
         & (since $l$ lali, so $lr\cong1$).
    \end{tabular}
    \end{center}

Let us now consider the pseudocommutative square (with $l_i\dashv r_i$ for $i=1,2$)
\[\begin{tikzcd}
	A & B \\
	X & {Y.}
	\arrow["f", from=1-1, to=1-2]
	\arrow["g"', from=2-1, to=2-2]
	\arrow[""{name=0, anchor=center, inner sep=0}, "{l_1}"', from=1-1, to=2-1]
	\arrow[""{name=1, anchor=center, inner sep=0}, "{l_2}", from=1-2, to=2-2]
	\arrow["\cong"{description}, Rightarrow, draw=none, from=0, to=1]
\end{tikzcd}\]
    By the first part we already know that $X$ and $Y$ are left Kan injective with respect to $\maps$. We have left to prove that $g$ preserves Kan extensions, i.e. for any $h\colon C\to C'$ in $\maps$ and any $t\colon C\to X$, then $g\circ t/h\cong(gt)/h$. Indeed,
    \begin{center}
    \begin{tabular}{rlr}
         $g\circ t/h$
         & $\cong g\circ l_1\circ (r_1t)/h$
         & (by first part applied to $X$)\\
         & $\cong l_2\circ f\circ (r_1t)/h$
         & (by the pseudocommutativity of the square)\\
         & $\cong l_2\circ  (fr_1t)/h$
         & (because $f\in\slinj(\maps)$)\\
         & $\cong (l_2fr_1t)/h$
         & (because left adjoints preserve Kan extension)\\
         & $\cong (gl_1r_1t)/h$
         & (by $gl_1\cong l_2f$)\\
         & $\cong (gt)/h$
         & (by $l_1r_1\cong1$).
    \end{tabular}
    \end{center}
\end{proof}

\begin{theorem}{\em (\cite{carv-sousa_order-refl}, \cite{marm-wood_lax-idemp}, see also \cite{kock1977doctrines} and \cite{BungeFunk1999})}\label{thm:units}
The 2-category of pseudoalgebras and homomorphisms of a KZ-pseudomonad is, up to 2-equivalence, the sub-2-category $\LInj(\mathcal{U})$  where $\mathcal{U}$ is made of all components of the unit of the pseudomonad.
\end{theorem}

\begin{proof}  For order-enriched categories, this was proven in \cite{carv-sousa_order-refl}. For the general context it follows from \cite{marm-wood_lax-idemp}, by combining the description made by Marmolejo and Wood, in Section 3 of that paper, of the 2-category of algebras $\mathbb{D}$-$\overline{\text{Alg}}$ for $\mathbb{D}$ a left Kan pseudomonad, and the fact, given by them in Section 5, Theorem 5.1, that it is, up 2-equivalence, the 2-category of algebras of the \textit{lax}-idempotent pseudomonad determined by $\mathbb{D}$. {\blue Indeed, we can rephrase the description of $\mathbb{D}$-$\overline{\text{Alg}}$ given in  \cite[Section 3]{marm-wood_lax-idemp} (after Remark 3.2), using left Kan extensions instead of right ones, as follows:
\begin{enumerate}
\item[(i)] the objects of $\mathbb{D}$-$\overline{\text{Alg}}$  are all $X$ in $\cal K$ which are left Kan injective with respect to $\cal U$ and such that, for every $A$ and every $f\colon A\to X$, the 1-cell data $f/d_A$ of the Kan extension belongs to $\LInj(\mathcal{U})$;
    \item[(ii)] the morphisms of $\mathbb{D}$-$\overline{\text{Alg}}$ are precisely those in $\LInj(\mathcal{U})$.
    \end{enumerate}

   Thus, we just need to show that, in (i), the condition that $f/d_A$ belongs to $\LInj(\mathcal{U})$ is redundant.}%

   We start proving that, given $X\in\slinj(\{d_X\mid X\in \catk\})$, then $1_X/d_X\colon DX\to X$ is a lali, in particular $1_X/d_X\dashv d_X$. We set  $\epsilon$ as the inverse of the universal 2-cell
\begin{center}
    $\epsilon^{-1}:=$
\begin{tikzcd}[ampersand replacement=\&]
	X \& {DX} \\
	X
	\arrow[""{name=0, anchor=center, inner sep=0}, Rightarrow, no head, from=1-1, to=2-1]
	\arrow["d_X", from=1-1, to=1-2]
	\arrow[""{name=1, anchor=center, inner sep=0}, "{1_X/d_X}", from=1-2, to=2-1]
	\arrow["\cong"{xshift=-0.15cm,yshift=0.025cm}, shorten <=5pt, shorten >=5pt, Rightarrow, from=0, to=1]
\end{tikzcd}
\end{center}
Moreover, we can define $\eta\colon 1_{DX}\Rightarrow d_X\circ 1_X/d_X$ using that $1_{DX}$ is a Kan extension (since $d_X$ is dense). More precisely, we define $\eta$ as the 2-cell corresponding to the 2-cell $\eta'$ defined below.
\begin{center}
\begin{tikzcd}[ampersand replacement=\&]
	X \& {DX} \\
	{DX}
	\arrow["d_X", from=1-1, to=1-2]
	\arrow[""{name=0, anchor=center, inner sep=0}, "d_X"', from=1-1, to=2-1]
	\arrow[""{name=1, anchor=center, inner sep=0}, "{d_X\circ1/d_X}", from=1-2, to=2-1]
	\arrow["{\eta'}"{xshift=-0.075,yshift=0.075}, shorten <=6pt, shorten >=6pt, Rightarrow, from=0, to=1]
\end{tikzcd} $:=$
\begin{tikzcd}[ampersand replacement=\&]
	X \& {DX} \\
	\& X \& {DX}
	\arrow["{d_X}", from=1-1, to=1-2]
	\arrow[""{name=0, anchor=center, inner sep=0}, "{1/d_X}", from=1-2, to=2-2]
	\arrow["{d_X}", from=2-2, to=2-3]
	\arrow[""{name=1, anchor=center, inner sep=0}, curve={height=6pt}, Rightarrow, no head, from=1-1, to=2-2]
	\arrow["{\epsilon^{-1}}"{xshift=0.15cm,yshift=0.075cm}, shorten <=5pt, shorten >=5pt, Rightarrow, from=1, to=0]
\end{tikzcd}
\end{center}
One triangle identity follows directly from the definitions of $\epsilon$ and $\eta$ and the second one from the (2-dimensional) universal property of $1_X/d_X$. Then, since  $DX\in\slinj(\cal U)$, by Lemma \ref{lem:left-inj-vs-lali} we get that also $X\in\slinj(\cal U)$ and $1_X/d_X$ is a morphism of $\slinj(\cal U)$.

{\blue Finally, let us consider $f\colon A\to X$ and show that $f/d_X$ belongs to $\slinj(\cal U)$. Indeed, since $1_X/d_X$ does, we have that 
$$f/d_A\cong (1_Xf)/d_A\cong (1_X/d_X \circ d_X\circ f)/d_A\cong (1_X/d_X)\circ(\,(d_Xf)/d_A\,).$$ By Theorem \ref{thm:marm-wood}, we know that $(d_Xf)/d_A$ lies in $\slinj(\cal U)$, thus, being isomorphic to the composition of two morphisms in $\slinj(\cal U)$, $f/d_A$ is also in $\slinj(\cal U)$.
}%
\end{proof}

We have just seen that the 2-category of pseudoalgebras of a KZ-pseudomonad is essentially a Kan injective sub-2-category of $\catk$. A natural question is: When is a Kan injective sub-2-category 2-equivalent to the 2-category of pseudoalgebras for a KZ-pseudomonad? For ordinary categories this reduces to the famous Orthogonal Subcategory Problem (introduced in \cite{freyd1972categories}) asking when is an orthogonal subcategory the category of algebras of an idempotent monad. For order-enriched categories, an answer of the \textit{Kan Injective Subcategory Problem} was given in \cite{adamek2015kan}. The next two sections are dedicated to give an answer in the general 2-dimensional context.

We end this section by showing that a Kan injective sub-2-category of $\catk$ whose inclusion into $\catk$ is the right part of a KZ-adjunction is always KZ-monadic, that is, the 2-category of pseudoalgebras of the corresponding KZ-pseudomonad, up to 2-equivalence.

\begin{corollary}\label{cor:Linj(H)}For every class of 1-cells $\mathcal{H}$, if the inclusion $\slinj(\mathcal{H}) \hookrightarrow \catk$ is the right part of a KZ-adjunction, then the 2-category of pseudoalgebras of the corresponding KZ-pseudomonad is 2-equivalent to $\slinj(\mathcal{H})$.
\end{corollary}

\begin{proof}
By Theorems \ref{thm:marm-wood} and \ref{thm:units}, and using their notation, we have just to prove that $\slinj(\{d_X\mid X\in \catk\})$ is contained in $\slinj(\mathcal{H})$.
{\blue For every object $X$ in $\slinj(\{d_X\})$, the morphism $1_X/d_X\colon DX\to X$ is a lali, as shown in the proof of Theorem \ref{thm:units}. By Lemma~\ref{lem:left-inj-vs-lali}, since $DX$ belongs $\slinj(\mathcal{H})$, we get that also $X$ belongs to $\slinj(\mathcal{H})$.
}%
Moreover, given $u\colon X\to Y$ in $\slinj(\lbrace d_X\rbrace)$, we can consider the diagrams
\[\begin{tikzcd}
	X & Y && {DX} & Y \\
	{DX} & Y && X & Y
	\arrow[""{name=0, anchor=center, inner sep=0}, Rightarrow, no head, from=1-2, to=2-2]
	\arrow[""{name=1, anchor=center, inner sep=0}, "{d_X}"', from=1-1, to=2-1]
	\arrow["u", from=1-1, to=1-2]
	\arrow["{u/d_X}"', from=2-1, to=2-2]
	\arrow["{u/d_X}", from=1-4, to=1-5]
	\arrow[""{name=2, anchor=center, inner sep=0}, "{1_X/d_X}"', from=1-4, to=2-4]
	\arrow["u"', from=2-4, to=2-5]
	\arrow[""{name=3, anchor=center, inner sep=0}, Rightarrow, no head, from=1-5, to=2-5]
	\arrow["\cong"{description}, Rightarrow, draw=none, from=1, to=0]
	\arrow["\cong"{description}, Rightarrow, draw=none, from=2, to=3]
\end{tikzcd}\]
which are mates. Then, since $u/d_X\in\slinj(\maps)$, using again Lemma \ref{lem:left-inj-vs-lali},  we get that also $u\in\slinj(\maps)$.
\end{proof}

\begin{remark}For a KZ-pseudomonad, let $\mathcal{U}$ be the class of the units. Between the sub-2-category of all pseudoalgebras and its full sub-2-category  consisting of all free algebras  we may encounter several relevant sub-2-categories. This is the topic of the paper \cite{hofmann-sousa_aaa}, dealing with the order-enriched context.
\end{remark}

\section{The pseudochain construction} \label{chain}

The transfinite chain described here  is a $2$-dimensional enhancement of the \textit{orthogonal reflection construction} \cite[1.37]{adamek_rosicky_1994}. The $\mathsf{Pos}$-enriched version analogue of this chain was presented in \cite[Construction 5.2]{adamek2015kan}.

The archetype of a transfinite construction of this kind is the one of Quillen's Small Object Argument. A deep general study on transfinite constructions of free algebras on ordinary categories was made in \cite{kelly1980}.  In the transfinite construction of \cite{adamek2015kan}, besides the conical colimits used in the ordinary case, coinserters were applied. Here, we use  a new  ingredient, named coequinserter, whose definition (in its strict version) is given next. It is a special 2-colimit which may be obtained as the composition of a coinserter with a coequifier.

\begin{definition}Given a 2-cell
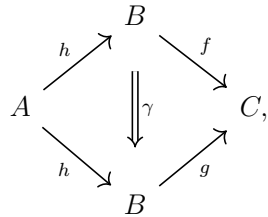
\[\adjustbox{scale=0.9}{\begin{tikzcd}
	& B \\
	A && C, \\
	& B
	\arrow["h", from=2-1, to=1-2]
	\arrow["h"', from=2-1, to=3-2]
	\arrow["f", from=1-2, to=2-3]
	\arrow["g"', from=3-2, to=2-3]
	\arrow["\gamma", shorten <=15pt, shorten >=15pt, Rightarrow, from=1-2, to=3-2]
\end{tikzcd}}\]
a \textbf{coequinserter} of $\gamma$ consists of a 1-cell $i\colon C\to Q$ and a 2-cell
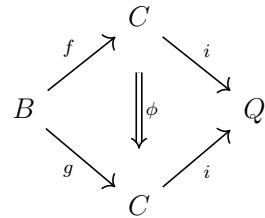
\[\adjustbox{scale=0.9}{\begin{tikzcd}
	& C \\
	B && Q \\
	& C
	\arrow["f", from=2-1, to=1-2]
	\arrow["g"', from=2-1, to=3-2]
	\arrow["i", from=1-2, to=2-3]
	\arrow["i"', from=3-2, to=2-3]
	\arrow["\phi", shorten <=15pt, shorten >=15pt, Rightarrow, from=1-2, to=3-2]
\end{tikzcd}}\]
such that
\[\adjustbox{scale=0.9}{\begin{tikzcd}
	& B &&&&&& C \\
	A && C & Q & {=} & A & B && Q \\
	& B &&&&&& C
	\arrow["f", from=2-7, to=1-8]
	\arrow["g"', from=2-7, to=3-8]
	\arrow["i", from=1-8, to=2-9]
	\arrow["i"', from=3-8, to=2-9]
	\arrow["\phi", shorten <=10pt, shorten >=10pt, Rightarrow, from=1-8, to=3-8]
	\arrow["h", from=2-6, to=2-7]
	\arrow["i", from=2-3, to=2-4]
	\arrow["h", from=2-1, to=1-2]
	\arrow["h"', from=2-1, to=3-2]
	\arrow["f", from=1-2, to=2-3]
	\arrow["g"', from=3-2, to=2-3]
	\arrow["\gamma", shorten <=15pt, shorten >=15pt, Rightarrow, from=1-2, to=3-2]
\end{tikzcd}}\]
with the following universal properties:
\begin{enumerate}
    \item[(1)] For any other 1-cell $u\colon C\to R$ and 2-cell $\epsilon\colon uf\Rightarrow ug$ such that $u\gamma=\epsilon h$, there exists a unique $t\colon Q\to R$ such that $ti=u$ and $t\phi=\epsilon$.

    \item[(2)] For any pair of 1-cells $u,v\colon Q\to R$ and 2-cell $\theta\colon ui\Rightarrow vi$ such that
\[\adjustbox{scale=0.9}{\begin{tikzcd}
	& C &&&&& C & Q \\
	B && Q & R & {=} & B && Q & R, \\
	& C & Q &&&& C
	\arrow["i"{description}, from=3-2, to=2-3]
	\arrow["i"', from=3-2, to=3-3]
	\arrow["u", from=2-3, to=2-4]
	\arrow["v"', from=3-3, to=2-4]
	\arrow["g"', from=2-1, to=3-2]
	\arrow["f", from=2-1, to=1-2]
	\arrow["i", from=1-2, to=2-3]
	\arrow["\phi"{xshift=0.1cm}, shorten <=15pt, shorten >=15pt, Rightarrow, from=1-2, to=3-2]
	\arrow["\theta"{xshift=0.1cm}, shorten <=3pt, shorten >=3pt, Rightarrow, from=2-3, to=3-3]
	\arrow["f", from=2-6, to=1-7]
	\arrow["g"', from=2-6, to=3-7]
	\arrow["i"{description}, from=1-7, to=2-8]
	\arrow["i"', from=3-7, to=2-8]
	\arrow["\phi"{xshift=0.1cm}, shorten <=15pt, shorten >=15pt, Rightarrow, from=1-7, to=3-7]
	\arrow["i", from=1-7, to=1-8]
	\arrow["u", from=1-8, to=2-9]
	\arrow["v"', from=2-8, to=2-9]
	\arrow["\theta"{xshift=0.1cm}, shorten <=3pt, shorten >=3pt, Rightarrow, from=1-8, to=2-8]
\end{tikzcd}}\]
    then there exists a unique 2-cell $\overline{\theta}\colon u\Rightarrow v$ with $\overline{\theta}i=\theta$.
\end{enumerate}
\end{definition}

\begin{remark}[Coequinserters from coinserters and coequifiers]
In a 2-category with (bi)coinserters and (bi)coequifiers, we can construct a (bi)coequinserter as follows. First, we consider the (bi)coinserter of $f,g\colon B\to C$,
\[\adjustbox{scale=0.9}{\begin{tikzcd}
	& C \\
	B && D. \\
	& C
	\arrow["f", from=2-1, to=1-2]
	\arrow["g"', from=2-1, to=3-2]
	\arrow["e", from=1-2, to=2-3]
	\arrow["e"', from=3-2, to=2-3]
	\arrow["\chi"{xshift=0.1cm}, shorten <=15pt, shorten >=15pt, Rightarrow, from=1-2, to=3-2]
\end{tikzcd}}\]
Then, let $q\colon C\to Q$ be the (bi)coequifier  of $\chi\circ h$ and $e\circ \gamma$:
\[\begin{tikzcd}
	& B & C \\
	A &&& D. \\
	& B & C
	\arrow["g"', from=3-2, to=3-3]
	\arrow["e", from=1-3, to=2-4]
	\arrow["{e\circ \gamma}"{xshift=0.1cm}, shorten <=15pt, shorten >=15pt, Rightarrow, from=1-3, to=3-3]
	\arrow["f", from=1-2, to=1-3]
	\arrow["e"', from=3-3, to=2-4]
	\arrow["h", from=2-1, to=1-2]
	\arrow["h"', from=2-1, to=3-2]
	\arrow["{\chi\circ h}"'{xshift=-0.1cm}, shorten <=15pt, shorten >=15pt, Rightarrow, from=1-2, to=3-2]
\end{tikzcd}\]
    One can check that the (bi)coequinserter is given by the 1-cell $qe\colon C\to Q$ and the 2-cell $q\circ\chi\colon (qe)f\Rightarrow (qe)g$.
\end{remark}

    \begin{notat}[Pseudochains] \label{pseudochain}
    For any ordinal $i$, let $\mathbf{i}$ be the ordered set of  all ordinals $j< i$ looked as a locally discrete 2-category. By an \textit{$i$-pseudochain} in a 2-category $\catk$ we mean a normal pseudofunctor
    $$X\colon \mathbf{i}\to \catk\,.$$
    We denote $X(j\leq k)$ by ${x_{j,k}}\colon X_j\rightarrow X_k$; in particular, by the normality assumption,  $x_{jj} = 1_{X_j}$. For any $j\leq k\leq l$, we denote with 
\[\begin{tikzcd}[ampersand replacement=\&]
	\& {X_k} \\
	{X_j} \&\& {X_l}
	\arrow["{x_{j,k}}", from=2-1, to=1-2]
	\arrow["{x_{k,l}}", from=1-2, to=2-3]
	\arrow[""{name=0, anchor=center, inner sep=0}, "{x_{j,l}}"', from=2-1, to=2-3]
	\arrow["{\compX^{k}_{j,l}}", shorten <=6pt, shorten >=6pt, Rightarrow, from=1-2, to=0]
\end{tikzcd}\]
    the compositor isomorphism given by the pseudofunctoriality of $X$\blue; in particular, by the normality assumption,  $\mathbf{x}^j_{jk} = 1_{x_{j,k}}=\mathbf{x}^k_{jk}$.\black 

    \black 

  \blue  When clear from the context, we might omit super/subscripts and write $x\colon X_j\to X_k$ for $X(j\leq k)$ and $\compX=\compX^{k}_{j,l}\colon x_{k,l}x_{j,k} \Rightarrow x_{j,l}$ for the compositor isomorphisms.

For any $j\leq k_1\leq\ldots\leq k_n\leq l$, the axioms of a pseudofunctor ensure that all possible pasting of compositor isomorphisms are equal. We will denote any of these with the following notation. 
\[\begin{tikzcd}[ampersand replacement=\&]
	\& {X_{k_1} } \& \ldots \& {X_{k_n}} \\
	{X_j} \&\&\&\& {X_l}
	\arrow["{x_{j,k_1}}", from=2-1, to=1-2]
	\arrow[""{name=0, anchor=center, inner sep=0}, "{x_{j,l}}"', from=2-1, to=2-5]
	\arrow["{x_{k_n,l}}", from=1-4, to=2-5]
	\arrow["{x_{k_1,k_2}}", from=1-2, to=1-3]
	\arrow["{x_{k_{n-1},k_n}}", from=1-3, to=1-4]
	\arrow["{\mathbf{x}^{k_1,\ldots,k_n}_{j,l}=\mathbf{x}}", shorten <=6pt, shorten >=6pt, Rightarrow, from=1-3, to=0]
\end{tikzcd}\] \black

     Analogously, we may consider a pseudochain indexed by all ordinals, considering a pseudofunctor from the category $\mathbf{Ord}$.
    \end{notat}

    In a 2-category $\catk$ with (weighted) bicolimits, given a set of 1-cells $\maps$, we are going to construct, for every object $X\in \catk$, a pseudochain which will allow us (in Section \ref{TheTheorem}) to obtain the free pseudoalgebras of a KZ-pseudomonad induced by the inclusion $\LInj(\maps)\hookrightarrow \catk$.

\begin{constr}[The Kan injective pseudochain] \label{kansmallobject} Let $\mathcal{K}$ be a locally small $2$-category with small (weighted) bicolimits and let $\maps$ be a set of $1$-cells in $\mathcal{K}$. Given an object $X$ we construct a pseudochain  of objects $X_i$ ($i\in\ord$). \blue As in Notation \ref{pseudochain}, \black we denote the connecting maps by $x_{ji}\colon X_j\to X_i$  \blue and the compositor isomorphisms by $\compX_{jl}^k$, for all $j\leq k\leq i$. 
In each step $i$, we obtain an (i+1)-pseudochain which extends the previous (j+1)-pseudochains, $j<i$.\black

The first step is the given object $X_0:=X$.

\vskip1mm

\blue {\em Limit steps.} For $i$ a limit ordinal, $X_i$ is a bicolimit of the  $i$-pseudochain $(X_{j})_{j<i}$. The  connecting maps $x_{ji}\colon X_j\to X_i$ are the corresponding coprojections, and the 2-cells $\compX_{ji}^k$ are the isomorphisms given by the definition of bicolimit. \black

\vskip1mm

{\em Isolated steps.} Given $X_i$ with $i$ even, we define both $X_{i+1}$ and $X_{i+2}$. The idea is that the $i+1$ step approximates the 1-dimensional property of a Kan injective object and the $i+2$ step the 2-dimensional one.
\begin{enumerate}
    \item\label{kansmall:i+1} To define $X_{i+1}$ and the connecting map $x_{i,i+1}\colon X_i\to X_{i+1}$, consider \blue the diagram 
 \begin{equation}\label{diag:span}
 \begin{tikzpicture}
\node (a) at (0,2.5) {$A$};
\node (b) at (2,2.5) {$A'$};
\node (c) at (0,0) {$X_i$};
\node (a') at (1,1.5) {$\bullet$};
\node (b') at (3,1.5) {$\bullet$};
\draw[->] (a) to  node[scale=.7] (f) [above] {$h\in\maps$} (b);
\draw[dotted] (a) to (a');
\draw[->] (a') to  node[scale=.7] (f) [above] {$\in\maps$} (b');
\draw[->] (a) to  node[scale=.7] (f) [left] {$f\in\catk$} (c);
\draw[->] (a') to  node[scale=.7] (f) [right] {$\in\catk$} (c);
\end{tikzpicture}
\end{equation}
indexed by  all  spans $X_i\xleftarrow{f}A\xrightarrow{h}A'$ with $h\in \maps$ and $f$ arbitrary. 
  \black   We take the conical  bicolimit of this diagram. It may be obtained as a wide bipushout of all  bipushouts of $f$ along  $h$, each $(f,h)$  as in \eqref{diag:span}. 

    We set $x_{i,i+1}$ and $f//h$ the coprojections of the bicolimit, and the 1-cell $x_{i,i+1}$ is the required new connecting map in the pseudochain. \blue We denote by $\tilde{\xi}_f^h$ the corresponding isomorphisms:\black 
\begin{equation}\label{diag:w-bipu}
\begin{tikzcd}[ampersand replacement=\&]
	A \& {A'} \\
	{X_i} \& {X_{i+1}}
	\arrow[""{name=0, anchor=center, inner sep=0}, "f"', from=1-1, to=2-1]
	\arrow["h", from=1-1, to=1-2]
	\arrow[""{name=1, anchor=center, inner sep=0}, "{f//h}", from=1-2, to=2-2]
	\arrow["{x_{i,i+1}}"', from=2-1, to=2-2]
	\arrow["{\tilde{\xi}_f^h}", shorten <=13pt, shorten >=13pt, Rightarrow, from=0, to=1]
\end{tikzcd}
\end{equation}
\blue For every $j< k<i+1$, the new 2-cells $\compX_{j,i+1}^k$  are given by the composition of the identity on $x_{i,i+1}$ with $\compX_{j,i}^k$. In particular, $\mathbf{x}_{j,i+1}^i$ is the identity.\black

    \item\label{kansmall:i+2} Here we define
    $X_{i+2}$ and the connecting map $x_{i+1,i+2}\colon X_{i+1}\to X_{i+2}$. \blue The 2-dimensional property of a Kan injective object involves existence and uniqueness. Accordingly, we make use of two bicolimit constructions: bicoequinserters for the existence (in (a)) and bicoequifiers for the uniqueness (in (b)). Then, part (c) will put everything together. 
    \black
    
    \vskip2mm
    \begin{enumerate}
        \item[(a)] For every 2-cell $\gamma$ of the form
\[\begin{tikzcd}[ampersand replacement=\&]
	A \& {A'} \\
	{X_j} \& {X_{i+1}}
	\arrow["h", from=1-1, to=1-2]
	\arrow[""{name=0, anchor=center, inner sep=0}, "f"', from=1-1, to=2-1]
	\arrow[""{name=1, anchor=center, inner sep=0}, "g", from=1-2, to=2-2]
	\arrow["{x_{j,i+1}}"', from=2-1, to=2-2]
	\arrow["\gamma"{yshift=0.1cm}, shorten <=15pt, shorten >=15pt, Rightarrow, from=0, to=1]
\end{tikzcd}\]
with $h\in\maps$ and $j$ even, we consider the 2-cell
\[\begin{tikzcd}[ampersand replacement=\&]
	A \&\& {A'} \\
	{A'} \& {X_j} \&\& {X_{i+1}} \\
	\& {X_{j+1}}
	\arrow["h", from=1-1, to=1-3]
	\arrow[""{name=0, anchor=center, inner sep=0}, "f", from=1-1, to=2-2]
	\arrow[""{name=1, anchor=center, inner sep=0}, "g", from=1-3, to=2-4]
	\arrow["{x_{j,i+1}}"{description}, from=2-2, to=2-4]
	\arrow["{x_{j+1,i+1}}"', curve={height=12pt}, from=3-2, to=2-4]
	\arrow["h"', from=1-1, to=2-1]
	\arrow[""{name=2, anchor=center, inner sep=0}, "{x_{j,j+1}}"{description}, from=2-2, to=3-2]
	\arrow["{f//h}"', from=2-1, to=3-2]
	\arrow["{(\tilde{\xi}_f^h)^{-1}}"{pos=0.3}, shorten <=8pt, shorten >=8pt, Rightarrow, from=2-1, to=2-2]
	\arrow["{\compX^{j+1}_{j,i+1}}"'{xshift=-0.1cm}, shorten <=35pt, shorten >=35pt, Rightarrow, from=2, to=2-4]
	\arrow["\gamma", shorten <=40pt, shorten >=40pt, Rightarrow, from=0, to=1]
\end{tikzcd}\]
and its bicoequinserter  
$\xymatrix{X_{i+1}\ar[rr]^{c_{\gamma}}&&C_{\gamma}}$
with universal 2-cell
\[\adjustbox{scale=0.8}{\begin{tikzcd}[ampersand replacement=\&]
	\& {X_{j+1}} \&\& {X_{i+1}} \\
	{A'} \&\&\&\& {C_\gamma} \\
	\&\& {X_{i+1}}
	\arrow["g"', from=2-1, to=3-3]
	\arrow["{f//h}", from=2-1, to=1-2]
	\arrow[""{name=0, anchor=center, inner sep=0}, "{x_{j+1,i+1}}", from=1-2, to=1-4]
	\arrow["{c_\gamma}", from=1-4, to=2-5]
	\arrow["{c_\gamma}"', from=3-3, to=2-5]
	\arrow["{\chi_\gamma}"{xshift=0.1cm}, shorten <=17pt, shorten >=17pt, Rightarrow, from=0, to=3-3]
\end{tikzcd}}.\]

        \item[(b)]  \blue For every $\gamma=\{\sigma,\tau\}$\footnote{\blue We use here the same letter $\gamma$ that was used in the previous part for a 2-cell, although referring to a different situation. This will be useful in part (c) and in later proofs.} where $\sigma$ and $\tau$ are 2-cells as below such that $\sigma \circ h=\tau\circ h$,
\[\begin{tikzcd}[ampersand replacement=\&]
	\& {} \& {X_{j+1}} \& {} \\
	A \& {A'} \&\& {X_{i+1}} \\
	\& {} \& {} \& {}
	\arrow["{x_{j+1,i+1}}", curve={height=-6pt}, from=1-3, to=2-4]
	\arrow["{f//h}", curve={height=-6pt}, from=2-2, to=1-3]
	\arrow["g"{below}, curve={height=24pt}, from=2-2, to=2-4]
	\arrow["h", from=2-1, to=2-2]
	\arrow[draw=none, from=2-2, to=3-3]
	\arrow[draw=none, from=3-3, to=2-4]
	\arrow["\sigma"{xshift=0.1cm}, shift left=15, shorten <=18pt, shorten >=22pt, Rightarrow, from=1-2, to=3-2]
	\arrow["\tau"'{xshift=-0.1cm}, shift right=15, shorten <=18pt, shorten >=22pt, Rightarrow, from=1-4, to=3-4]
\end{tikzcd}\]
 we consider the bicoequifier of $\sigma$ and $\tau$ denoted with the 1-cell 

\centerline{$\xymatrix{X_{i+1}\ar[rr]^{c_{\gamma}}&&C_{\gamma}}$.}

        \item[(c)] 
\blue We define the morphism $x_{i+1,i+2}\colon X_{i+1}\to X_{i+2}$ through the wide bipushout of all $c_{\gamma}$ with $\gamma\in\Gamma$, where $\Gamma$ consists of all $\gamma$ described either in part (a) or (b). Hence $x_{i+1,i+2}$ comes equipped with canonical isomorphisms $\delta_\gamma$ as below, for any $\gamma\in\Gamma$. \black 
\begin{equation}\label{diag:w-bipu0}
\begin{tikzcd}[ampersand replacement=\&]
	{X_{i+1}} \& {C_\gamma} \\
	\& {X_{i+2}}
	\arrow["{c_\gamma}", from=1-1, to=1-2]
	\arrow["{d_\gamma}", from=1-2, to=2-2]
	\arrow[""{name=0, anchor=center, inner sep=0}, "{x_{i+1,i+2}}"', from=1-1, to=2-2]
	\arrow["{\delta_\gamma}", shift right=3, shorten <=6pt, shorten >=2pt, Rightarrow, from=0, to=1-2]
\end{tikzcd}
\end{equation}
{\blue Similarly to the previous isolated step, we define the new 2-cells $\compX_{j,i+2}^k$ by composition of $\mathbf{x}_{j,i+1}^k$ with the identity on $x_{i+1,i+2}$.}
\end{enumerate}

\end{enumerate}
    
\end{constr}



In the following lemma, which is going to be useful in the proof of Theorem \ref{thm:linj-kz-adj}, we show that, for every ordinal $i$, $x_{0i}$ belongs to the Kan injectivity saturation $\maps^{\text{sat}}$ (see Subsection \ref{subsec:saturation}).

\begin{lemma}\label{lem:ext-to-P}
Let $\maps$ be a set of 1-cells in a locally small 2-category with small bicolimits. In the Kan injective pseudochain, for every ordinal $i$,  the sub-$2$-category $\slinj(\maps)$ is left Kan injective with respect to $x_{0i}\colon X_0\to   X_i$, i.e.
$$\slinj(\maps)\subseteq \slinj(\lbrace x_{0i}\mid X \in \catk\rbrace).$$
This determines, for each $p_0\colon X_0\to P$ with $P\in \LInj(\maps)$, a pseudococone \blue given by 1-cells $p_i\colon X_i\to P$ and invertible 2-cells $\mathbf{p}_{i,j}\colon p_i\Rightarrow p_jx_{i,j}$ such that $(p_i, \mathbf{p}_{0i})$  is a Kan extension of $p_0$ along $x_{0i}$, i.e. with our notation \black 
$$p_i\cong p_0/x_{0i}.$$
\end{lemma}
\begin{proof} The proof is by transfinite induction on ordinals.

\vskip2mm

\noindent \emph{Limit step}.
 Assume the property holds for all $i< \kappa$, where $\kappa$ is a limit ordinal. 
 Then, by construction of $X_\kappa$ \blue (as a bicolimit)\black, there is a unique (up-to-iso) $1$-cell $p_\kappa\colon X_\kappa \to P$ \blue equipped with, for any $i< \kappa$, invertible 2-cells $\mathbf{p}_{i,\kappa}\colon p_i\Rightarrow p_\kappa x_{i,\kappa}$ such that the equation below holds. \black   
 \begin{equation}\label{eq:def-p_k}
\begin{tikzcd}[ampersand replacement=\&]
	{X_i} \&\& {X_\kappa} \&\& {X_i} \&\& {X_\kappa} \\
	\&\&\& {=} \&\& {X_j} \\
	\& P \&\&\&\& P
	\arrow[""{name=0, anchor=center, inner sep=0}, "{p_i}"', curve={height=18pt}, from=1-1, to=3-2]
	\arrow[""{name=1, anchor=center, inner sep=0}, "{p_\kappa}", curve={height=-18pt}, from=1-3, to=3-2]
	\arrow[""{name=2, anchor=center, inner sep=0}, "{p_i}"', curve={height=18pt}, from=1-5, to=3-6]
	\arrow[""{name=3, anchor=center, inner sep=0}, "{p_j}"{description, pos=0.3}, from=2-6, to=3-6]
	\arrow["{x_{i,\kappa}}"', from=2-6, to=1-7]
	\arrow[""{name=4, anchor=center, inner sep=0}, "{x_{i,j}}"', from=1-5, to=2-6]
	\arrow[""{name=5, anchor=center, inner sep=0}, "{p_\kappa}", curve={height=-18pt}, from=1-7, to=3-6]
	\arrow[""{name=6, anchor=center, inner sep=0}, "{x_{i,\kappa}}", from=1-5, to=1-7]
	\arrow["{x_{i,\kappa}}", from=1-1, to=1-3]
	\arrow["{\mathbf{p}_{i,j}}"'{xshift=-0.1cm,yshift=-0.1cm}, shorten <=6pt, shorten >=6pt, Rightarrow, from=2, to=2-6]
	\arrow["{\mathbf{x}_{i,\kappa}^j}"', shift right=1, shorten <=8pt, shorten >=8pt, Rightarrow, from=4, to=6]
	\arrow["{\mathbf{p}_{j,\kappa}}"'{pos=0.6,xshift=-0.3cm,yshift=-0.1cm}, shift left=3, shorten <=19pt, shorten >=6pt, Rightarrow, from=3, to=5]
	\arrow["{\mathbf{p}_{i,\kappa}}", shorten <=26pt, shorten >=26pt, Rightarrow, from=0, to=1]
\end{tikzcd}
\end{equation}

    We want to show that $(p_\kappa, \mathbf{p}_{0\kappa})$ \blue is a Kan extension of $p_0$ along $x_{0\kappa}$\black. Given a $1$-cell $r$ and a $2$-cell $\alpha$ as below, for every $i < \kappa$, since $p_i \cong p_0/x_{0i}$ by inductive hypothesis, we have a unique $2$-cell $\alpha_i\colon p_i \Rightarrow rx_{i\kappa}$ such that the following equality holds. 

\begin{center}
\begin{tikzcd}[ampersand replacement=\&]
	\& {X_i} \\
	{X_0} \&\& {X_\kappa} \\
	P
	\arrow[""{name=0, anchor=center, inner sep=0}, "{p_0}"', from=2-1, to=3-1]
	\arrow["r", curve={height=-12pt}, from=2-3, to=3-1]
	\arrow["{x_{0i}}", from=2-1, to=1-2]
	\arrow[""{name=1, anchor=center, inner sep=0}, "{x_{i\kappa}}", from=1-2, to=2-3]
	\arrow[""{name=2, anchor=center, inner sep=0}, "{x_{0\kappa}}"{description}, from=2-1, to=2-3]
	\arrow["\alpha"{xshift=-0.1cm}, shift right=3, shorten <=25pt, shorten >=40pt, Rightarrow, from=0, to=2-3]
	\arrow["{(\compX^i_{0,\kappa})^{-1}}"{pos=0.3, xshift=0.15cm}, shift left=1, shorten <=10pt, shorten >=10pt, Rightarrow, from=2, to=1]
\end{tikzcd}
\hspace{0.5cm}$=$\hspace{0.5cm}
\begin{tikzcd}[ampersand replacement=\&]
	{X_0} \&\& {X_i} \& {X_\kappa} \\
	P
	\arrow["{x_{0i}}", from=1-1, to=1-3]
	\arrow[""{name=0, anchor=center, inner sep=0}, "{p_0}"', from=1-1, to=2-1]
	\arrow[""{name=1, anchor=center, inner sep=0}, "{p_i}"{description}, from=1-3, to=2-1]
	\arrow[""{name=2, anchor=center, inner sep=0}, "r", curve={height=-12pt}, from=1-4, to=2-1]
	\arrow["{x_{i\kappa}}", from=1-3, to=1-4]
	\arrow["{\alpha_i}", shorten <=8pt, shorten >=8pt, Rightarrow, from=1, to=2]
	\arrow["{\mathbf{p}_{0,i}}"{yshift=0.1cm}, shorten <=13pt, shorten >=13pt, Rightarrow, from=0, to=1]
\end{tikzcd}
\end{center}
\blue Moreover, using the universal property of $p_j$ as a Kan extension, we can prove that for any $j< i$ we get the following equality. 
\begin{center}
\begin{tikzcd}[ampersand replacement=\&]
	\& {X_i} \\
	{X_j} \&\& {X_\kappa} \\
	P
	\arrow[""{name=0, anchor=center, inner sep=0}, "{p_j}"', from=2-1, to=3-1]
	\arrow["r", curve={height=-12pt}, from=2-3, to=3-1]
	\arrow["{x_{ji}}", from=2-1, to=1-2]
	\arrow[""{name=1, anchor=center, inner sep=0}, "{x_{i\kappa}}", from=1-2, to=2-3]
	\arrow[""{name=2, anchor=center, inner sep=0}, "{x_{j\kappa}}"{description}, from=2-1, to=2-3]
	\arrow["({\compX^i_{j,\kappa})^{-1}}"{pos=0.3,xshift=0.15cm}, shift left=1, shorten <=10pt, shorten >=10pt, Rightarrow, from=2, to=1]
	\arrow["{\alpha_j}"{xshift=-0.1cm}, shift right=3, shorten <=25pt, shorten >=40pt, Rightarrow, from=0, to=2-3]
\end{tikzcd}
\hspace{0.5cm}$=$\hspace{0.5cm}
\begin{tikzcd}[ampersand replacement=\&]
	{X_j} \&\& {X_i} \& {X_\kappa} \\
	P
	\arrow["{x_{ji}}", from=1-1, to=1-3]
	\arrow[""{name=0, anchor=center, inner sep=0}, "{p_j}"', from=1-1, to=2-1]
	\arrow[""{name=1, anchor=center, inner sep=0}, "{p_i}"{description}, from=1-3, to=2-1]
	\arrow[""{name=2, anchor=center, inner sep=0}, "r", curve={height=-12pt}, from=1-4, to=2-1]
	\arrow["{x_{i\kappa}}", from=1-3, to=1-4]
	\arrow["{\alpha_i}", shorten <=8pt, shorten >=8pt, Rightarrow, from=1, to=2]
	\arrow["{\mathbf{p}_{j,i}}"{yshift=0.1cm}, shorten <=13pt, shorten >=13pt, Rightarrow, from=0, to=1]
\end{tikzcd}
\end{center}

\black 
    This implies that, using the 2-dimensional aspect of the universality of the bicolimit of the pseudochain $(X_i)_{i < \kappa}$, there exists a unique $2$-cell $\overline{\alpha}\colon p_\kappa \Rightarrow r$, with
    \begin{center}
        $\alpha_i=$
\adjustbox{scale=0.8}{
\begin{tikzcd}[ampersand replacement=\&]
	{X_i} \&\& {X_\kappa} \\
	\\
	P
	\arrow[""{name=0, anchor=center, inner sep=0}, "{p_i}"', from=1-1, to=3-1]
	\arrow["{x_{i\kappa}}", from=1-1, to=1-3]
	\arrow[""{name=1, anchor=center, inner sep=0}, "{p_\kappa}"{description}, from=1-3, to=3-1]
	\arrow[""{name=2, anchor=center, inner sep=0}, "r", curve={height=-24pt}, from=1-3, to=3-1]
	\arrow["{\overline{\alpha}}", shorten <=6pt, shorten >=4pt, Rightarrow, from=1, to=2]
	\arrow["{\mathbf{p}_{i,\kappa}}"{yshift=0.1cm}, shorten <=13pt, shorten >=13pt, Rightarrow, from=0, to=1]
\end{tikzcd}}
    for every $i < \kappa$
    and in particular $\alpha=$
\adjustbox{scale=0.8}{
\begin{tikzcd}[ampersand replacement=\&]
	{X_0} \&\& {X_\kappa} \\
	\\
	P
	\arrow[""{name=0, anchor=center, inner sep=0}, "{p_0}"', from=1-1, to=3-1]
	\arrow["{x_{0\kappa}}", from=1-1, to=1-3]
	\arrow[""{name=1, anchor=center, inner sep=0}, "{p_\kappa}"{description}, from=1-3, to=3-1]
	\arrow[""{name=2, anchor=center, inner sep=0}, "r", curve={height=-24pt}, from=1-3, to=3-1]
	\arrow["{\overline{\alpha}}", shorten <=6pt, shorten >=4pt, Rightarrow, from=1, to=2]
	\arrow["{\mathbf{p}_{0,\kappa}}"{yshift=0.1cm}, shorten <=13pt, shorten >=13pt, Rightarrow, from=0, to=1]
\end{tikzcd}}.
    \end{center}
     For the unicity of $\overline{\alpha}$ \blue we notice that this 2-cell was uniquely determined by the $\alpha_i$'s which are defined using only $\alpha$.\footnote{\blue More precisely: Let $\beta$ be another 2-cell such that $\alpha=(\beta\circ x_{0\kappa})\cdot \mathbf{p}_{0\kappa}$. Then, by the universal property of the Kan extension $p_i$, we have that $(\beta\circ x_{i\kappa})\cdot \mathbf{p}_{0\kappa}=\alpha_i$ since they both correspond to $\alpha$ (modulo pasting with $\mathbf{x}^i_{0,\kappa}$). Hence, by the universal property of the bicolimit, $\beta=\overline{\alpha}$.\black } \black   Consequently, $(p_\kappa, \mathbf{p}_{0\kappa})$ \blue is a Kan extension of $p_0$ along $x_{0\kappa}$.
     
     \black
    Moreover, let us consider $u\colon P\to Q$ in $\slinj(\maps)$ and $p_0\colon X_0\to P$ in $\catk$. By induction hypothesis we assume that $u$ is left Kan injective with respect to $x_{i\kappa}$ for all $i<k$, i.e. $(up)/x_{0i}\cong u(p/x_{0i})$. We want to show that $u$ is also in $\slinj(\lbrace x_{0\kappa}\rbrace)$.
    \begin{center}
    \begin{tabular}{rlr}
         $((up)/x_{0k})x_{ik}$& $\cong (up)_\kappa x_{ik}\cong (up)_i\cong (up)/x_{0i}$  & (by part above for $up$)\\
         &  $\cong u(p/x_{0i})$ & (by induction hypothesis) \\
         & $\cong u(p/x_{0k})x_{ik}$ & (by part above for $p$).
    \end{tabular}
    \end{center}
    Hence, by the universal property of the bicolimit,  $(up)/x_{0k}\cong u(p/x_{0k})$.
    
    \vskip2.5mm

 \noindent \emph{Isolated step}. Let $i$ be an even ordinal such that every $x_{0j}$, with ${j \leq i}$, belongs to the Kan injective saturation of $\mathcal{H}$. We treat the two cases of the construction, $i+1$ and $i+2$ for $i$ even,  separately.

\begin{enumerate}
    \item[(1)] As seen in the construction of the pseudochain, $x_{i,i+1}\colon X_i\to X_{i+1}$ is a wide bipushout of bipushouts of morphisms along 1-cells of $\mathcal{H}$. Then, by Proposition \ref{prop:saturation-irr}, $P$ is Kan injective with respect to $x_{i,i+1}$. Here we see in detail how to prove that $p_{i+1}\cong p_i/x_{i,i+1}$.  Combining this with the inductive hypothesis on $x_{0i}$, we get  \blue $p_{i+1}\cong p_0/x_{0,i+1}$ as required.\black 

        Recall the bicolimit diagram (\ref{diag:w-bipu}) used in the construction of the pseudochain. \blue Since for any $h\in\maps$ and $f\colon A\to X_i\in\catk$ we have an isomorphism, given by the Kan extension,
\[\begin{tikzcd}[ampersand replacement=\&]
	A \& {A'} \\
	{X_i} \\
	\&\& P,
	\arrow["h", from=1-1, to=1-2]
	\arrow["f"', from=1-1, to=2-1]
	\arrow["{p_i}"', curve={height=12pt}, from=2-1, to=3-3]
	\arrow[""{name=0, anchor=center, inner sep=0}, "{(p_if)/h}"{pos=0.6}, curve={height=-12pt}, from=1-2, to=3-3]
	\arrow["{\xi_{p_if}^h}", shorten <=20pt, shorten >=20pt, Rightarrow, from=2-1, to=0]
\end{tikzcd}\]
        then, by the universal property of $X_{i+1}$, we get a unique (up-to-iso) 1-cell $p_{i+1}\colon X_{i+1}\to P$ equipped with invertible 2-cells $\mathbf{p}_{i,i+1}\colon p_i\Rightarrow p_{i+1}x_{i,i+1}$ and $\pi\colon p_{i+1}f//h\Rightarrow (p_if)/h$ such that
\begin{equation}\label{diag:xi_pif}
\xi_{p_if}^h=
\begin{tikzcd}[ampersand replacement=\&]
	A \& {A'} \\
	{X_i} \& {X_{i+1}} \\
	\&\& P.
	\arrow["h", from=1-1, to=1-2]
	\arrow[""{name=0, anchor=center, inner sep=0}, "f"', from=1-1, to=2-1]
	\arrow[""{name=1, anchor=center, inner sep=0}, "{f//h}"{description}, from=1-2, to=2-2]
	\arrow["{x_{i,i+1}}"{description}, from=2-1, to=2-2]
	\arrow[""{name=2, anchor=center, inner sep=0}, "{p_i}"', curve={height=12pt}, from=2-1, to=3-3]
	\arrow[""{name=3, anchor=center, inner sep=0}, "{p_{i+1}}"{description}, dashed, from=2-2, to=3-3]
	\arrow[""{name=4, anchor=center, inner sep=0}, "{(p_if)/h}"{pos=0.6}, curve={height=-12pt}, from=1-2, to=3-3]
	\arrow["{\mathbf{p}_{i,i+1}}", shorten <=10pt, shorten >=10pt, Rightarrow, from=2, to=3]
	\arrow["\pi", shorten <=5pt, shorten >=5pt, Rightarrow, from=2-2, to=4]
	\arrow["{\tilde{\xi}_f^h}", shorten <=16pt, shorten >=16pt, Rightarrow, from=0, to=1]
\end{tikzcd} 
\end{equation}\black 
        We want to show that $(p_{i+1},\mathbf{p}_{i,i+1})$ is a left Kan extension of $p_i$ along $x_{i,i+1}$, i.e. $p_{i+1}\cong p_i/x_{i,i+1}$.

To do so, consider a $1$-cell $r\colon X_{i+1}\to P$ and a 2-cell $\alpha\colon p_i\Rightarrow rx_{i,i+1}$. \blue Let  $\tilde{\alpha}$ be the composition
\[\tilde{\alpha}:=\left(\begin{tikzcd}[ampersand replacement=\&]
	{p_{i+1}x_{i,i+1}} \& {p_i} \& {rx_{i,i+1}}
	\arrow["{{\mathbf{p}_{i,i+1}}^{-1}}"{yshift=0.1cm}, Rightarrow, from=1-1, to=1-2]
	\arrow["\alpha"{yshift=0.1cm}, Rightarrow, from=1-2, to=1-3]
\end{tikzcd}\right).\] \black 
For every span $(h,f)$, let $\bar{\alpha}_{hf}$ be the unique 2-cell for which we have the following equality, determined by the universality of the Kan extension $(\xi^h_{p_if},(p_if)/h)$,
\begin{center}
\begin{tikzcd}[ampersand replacement=\&]
	A \& {A'} \\
	{X_i} \&\& {X_{i+1}} \\
	P
	\arrow["h", from=1-1, to=1-2]
	\arrow[""{name=0, anchor=center, inner sep=0}, "f"', from=1-1, to=2-1]
	\arrow["{p_i}"', from=2-1, to=3-1]
	\arrow[""{name=1, anchor=center, inner sep=0}, "{(p_if)/h}"{description, pos=0.7}, from=1-2, to=3-1]
	\arrow["{f//h}", curve={height=-6pt}, from=1-2, to=2-3]
	\arrow["r", curve={height=-6pt}, from=2-3, to=3-1]
	\arrow["{\bar{\alpha}_{hf}}", shorten <=20pt, shorten >=20pt, Rightarrow, from=1, to=2-3]
	\arrow["{\xi^h_{p_if}}", shorten <=10pt, shorten >=10pt, Rightarrow, from=0, to=1]
\end{tikzcd} \hspace{1cm}$=$\hspace{1cm}
\begin{tikzcd}[ampersand replacement=\&]
	A \& {A'} \\
	{X_i} \& {X_{i+1}} \\
	P
	\arrow["h", from=1-1, to=1-2]
	\arrow[""{name=0, anchor=center, inner sep=0}, "f"', from=1-1, to=2-1]
	\arrow["{x_{i,i+1}}", from=2-1, to=2-2]
	\arrow[""{name=1, anchor=center, inner sep=0}, "{f//h}", from=1-2, to=2-2]
	\arrow[""{name=2, anchor=center, inner sep=0}, "{p_i}"', from=2-1, to=3-1]
	\arrow[""{name=3, anchor=center, inner sep=0}, "r", from=2-2, to=3-1]
	\arrow["\alpha"{yshift=0.15cm}, shorten <=6pt, shorten >=6pt, Rightarrow, from=2, to=3]
	\arrow["{\tilde{\xi}_f^h}", shorten <=16pt, shorten >=16pt, Rightarrow, from=0, to=1]
\end{tikzcd}
\end{center}

and set
\begin{center}
    $\tilde{\alpha}_{hf}:=$
\begin{tikzcd}[ampersand replacement=\&]
	\& {X_{i+1}} \\
	{A'} \&\& P \\
	\& {X_{i+1}}
	\arrow[""{name=0, anchor=center, inner sep=0}, "{(p_if)/h}"{description}, from=2-1, to=2-3]
	\arrow["{f//h}"', from=2-1, to=3-2]
	\arrow["r"', from=3-2, to=2-3]
	\arrow["{f//h}", from=2-1, to=1-2]
	\arrow["{p_{i+1}}", from=1-2, to=2-3]
	\arrow["{\bar{\alpha}_{hf}}"{xshift=0.1cm}, shorten <=10pt, shorten >=6pt, Rightarrow, from=0, to=3-2]
	\arrow["\pi"{xshift=0.1cm}, shorten <=6pt, shorten >=10pt, Rightarrow, from=1-2, to=0]
\end{tikzcd}
\end{center}
The 2-cells $\tilde{\alpha}$ and $\tilde{\alpha}_{hf}$
 satisfy the conditions under which we can apply the 2-dimensional aspect of the universality of the bicolimit given by (\ref{diag:w-bipu}). Consequently, there is a unique $\bar{\alpha}\colon p_{i+1}\Rightarrow r$ with $\bar{\alpha} x_{i,i+1}=\tilde{\alpha}$ and $\bar{\alpha} (f//h)=\tilde{\alpha}_{hf}$.

Hence, \blue taking into account the definition of $\tilde{\alpha}$ 
 and using the property of $\bar{\alpha}_{hf}$, \black we see that $\overline{\alpha}$ satisfies also the following equation
\begin{center}
    $\alpha=$
\begin{tikzcd}[ampersand replacement=\&]
	{X_i} \&\& {X_{i+1}} \\
	\\
	P
	\arrow[""{name=0, anchor=center, inner sep=0}, "{p_{i+1}}"{description}, curve={height=12pt}, from=1-3, to=3-1]
	\arrow[""{name=1, anchor=center, inner sep=0}, "r", curve={height=-18pt}, from=1-3, to=3-1]
	\arrow["{x_{i,i+1}}", from=1-1, to=1-3]
	\arrow["{p_i}"', from=1-1, to=3-1]
	\arrow["{\overline{\alpha}}", shorten <=10pt, shorten >=6pt, Rightarrow, from=0, to=1]
	\arrow["{\mathbf{p}_{i,i+1}}", shorten <=4pt, shorten >=11pt, Rightarrow, from=1-1, to=0]
\end{tikzcd}
\end{center}
and that it is unique.

 Concerning 1-cells, let $u\colon P\to Q$ be in $\slinj(\maps)$ and set $q:=up$. Adding $u$ to diagram \eqref{diag:xi_pif}, we have $(up_i)/h\cong u(p_i/h)$, and $up_{i}\cong q_{i}$. Thus $up_{i+1}$ and $q_{i+1}$ take isomorphic values when composed with $x_{i,i+1}$ and $f//h$. Consequently, $q_{i+1}\cong up_{i+1}$, i.e. $(up)/x_{0,i+1}\cong u(p/x_{0,i+1})$.

    \item[(2)] \blue First, we consider each $\gamma \in \Gamma$ (see part 2.(c) of Construction~\ref{kansmallobject}) and we construct a 1-cell $p_{\gamma}\colon C_{\gamma}\to P$ equipped with an invertible 2-cell $\mathbf{p}^c_\gamma\colon p_{i+1}\Rightarrow p_{\gamma}\circ c_{\gamma}$ associated to $\gamma$. We divide the two different kinds of $\gamma$ in the parts (a) and (b) below. 
    \black
    \begin{enumerate}
        \item[(a)] Let $\gamma$ be the following $2$-cell, with $j$ even and $j \leq i$:
\[\begin{tikzcd}[ampersand replacement=\&]
	A \&\& {A'} \\
	{A'} \& {X_{j+1}} \& {X_{i+1}.}
	\arrow["h", from=1-1, to=1-3]
	\arrow[""{name=0, anchor=center, inner sep=0}, "h"', from=1-1, to=2-1]
	\arrow["{f//h}"', from=2-1, to=2-2]
	\arrow["x_{j+1,i+1}"', from=2-2, to=2-3]
	\arrow[""{name=1, anchor=center, inner sep=0}, "s", from=1-3, to=2-3]
	\arrow["\gamma"{yshift=0.075cm}, shorten <=40pt, shorten >=40pt, Rightarrow, from=0, to=1]
\end{tikzcd}\]

Since $p_{j+1}(f//h)\cong(p_jf)/h$ is a left Kan extension, see diagram (\ref{diag:xi_pif}), there exists a unique 2-cell $\bar{\gamma} \colon p_{j+1}(f//h) \Rightarrow p_{i+1}s$ such that
\begin{equation}\label{diag:bar_gamma}
\begin{tikzcd}[ampersand replacement=\&]
	A \& {A'} \\
	{X_{j}} \& {A'} \\
	\& {X_{j+1}} \& {X_{i+1}} \\
	\& P
	\arrow["h", from=1-1, to=1-2]
	\arrow["h", from=1-1, to=2-2]
	\arrow["{f//h}", from=2-2, to=3-2]
	\arrow["{x_{j+1,i+1}}", from=3-2, to=3-3]
	\arrow[""{name=0, anchor=center, inner sep=0}, "s", curve={height=-12pt}, from=1-2, to=3-3]
	\arrow[""{name=1, anchor=center, inner sep=0}, "{p_{i+1}}", curve={height=-6pt}, from=3-3, to=4-2]
	\arrow[""{name=2, anchor=center, inner sep=0}, "{p_{j+1}}"{description}, from=3-2, to=4-2]
	\arrow["f"', from=1-1, to=2-1]
	\arrow["{x_{j,j+1}}"{description}, from=2-1, to=3-2]
	\arrow["{\tilde{\xi}^h_f}", shift right=1, shorten <=8pt, shorten >=8pt, Rightarrow, from=2-1, to=2-2]
	\arrow[""{name=3, anchor=center, inner sep=0}, "{p_j}"', curve={height=12pt}, from=2-1, to=4-2]
	\arrow["\gamma", shorten <=8pt, shorten >=8pt, Rightarrow, from=2-2, to=0]
	\arrow["{\mathbf{p}_{j,j+1}}"{xshift=0.35cm,yshift=0.1cm}, shorten <=7pt, shorten >=5pt, Rightarrow, from=3, to=3-2]
	\arrow["{\mathbf{p}_{j+1,i+1}}"{xshift=-0.25cm,yshift=0.1cm}, shorten <=10pt, shorten >=8pt, Rightarrow, from=2, to=1]
\end{tikzcd} \hspace{0.5cm}=\hspace{0.5cm}
\begin{tikzcd}[ampersand replacement=\&]
	A \& {A'} \\
	{X_j} \& {X_{j+1}} \& {X_{i+1}} \\
	\\
	\& P
	\arrow["s", curve={height=-6pt}, from=1-2, to=2-3]
	\arrow["{p_{j+1}}"{description}, from=2-2, to=4-2]
	\arrow["{p_{i+1}}", curve={height=-6pt}, from=2-3, to=4-2]
	\arrow[""{name=0, anchor=center, inner sep=0}, "{f//h}"{description}, from=1-2, to=2-2]
	\arrow["h", from=1-1, to=1-2]
	\arrow["{\overline{\gamma}}", shorten <=4pt, shorten >=4pt, Rightarrow, from=2-2, to=2-3]
	\arrow[""{name=1, anchor=center, inner sep=0}, "f"', from=1-1, to=2-1]
	\arrow["{x_{j,j+1}}", from=2-1, to=2-2]
	\arrow[""{name=2, anchor=center, inner sep=0}, "{p_j}"', from=2-1, to=4-2]
	\arrow["{\tilde{\xi}^h_f}", shorten <=15pt, shorten >=15pt, Rightarrow, from=1, to=0]
	\arrow["{\mathbf{p}_{j,j+1}}"{pos=0.8,yshift=-0.15cm}, shorten <=8pt, shorten >=8pt, Rightarrow, from=2, to=2-2]
\end{tikzcd}
\end{equation}

Therefore, by the 1-dimensional aspect of the universality of the bicoequinserter, there is a unique (up-to-iso) $1$-cell $p_\gamma\colon C_\gamma \to  P$ with invertible 2-cells $ \mathbf{p}^c_\gamma\colon p_{i+1}\Rightarrow p_\gamma c_\gamma $ such that
\begin{equation}\label{eq:def-p_gamma}
\scalebox{0.9}{
\begin{tikzcd}[ampersand replacement=\&]
	\& {X_{j+1}} \& {X_{i+1}} \\
	{A'} \&\& {C_\gamma} \& P \\
	\& {X_{i+1}}
	\arrow["s"', from=2-1, to=3-2]
	\arrow["{f//h}", from=2-1, to=1-2]
	\arrow["{x_{j+1,i+1}}", from=1-2, to=1-3]
	\arrow["{c_\gamma}"', from=1-3, to=2-3]
	\arrow["{c_\gamma}"{description}, from=3-2, to=2-3]
	\arrow["{p_\gamma}"{description}, from=2-3, to=2-4]
	\arrow[""{name=0, anchor=center, inner sep=0}, "{p_{i+1}}", curve={height=-6pt}, from=1-3, to=2-4]
	\arrow[""{name=1, anchor=center, inner sep=0}, "{p_{i+1}}"', curve={height=6pt}, from=3-2, to=2-4]
	\arrow["{\chi_\gamma}", shorten <=13pt, shorten >=13pt, Rightarrow, from=1-2, to=3-2]
	\arrow["{(\mathbf{p}^c_\gamma)^{-1}}"{pos=0.2,yshift=0.4cm,xshift=0.1cm}, shorten <=2pt, shorten >=4pt, Rightarrow, from=2-3, to=1]
	\arrow["{\mathbf{p}^c_\gamma}"', shorten <=6pt, shorten >=6pt, Rightarrow, from=0, to=2-3]
\end{tikzcd} \hspace{0.25cm}=\hspace{0.25cm}
\begin{tikzcd}[ampersand replacement=\&]
	\& {X_{j+1}} \& {X_{i+1}} \\
	{A'} \&\&\& P \\
	\& {X_{i+1}}
	\arrow["{x_{j+1,i+1}}", from=1-2, to=1-3]
	\arrow["{p_{i+1}}", curve={height=-6pt}, from=1-3, to=2-4]
	\arrow[""{name=0, anchor=center, inner sep=0}, "{p_{j+1}}"', curve={height=6pt}, from=1-2, to=2-4]
	\arrow["{p_{i+1}}"', from=3-2, to=2-4]
	\arrow["s"', from=2-1, to=3-2]
	\arrow["{f//h}", from=2-1, to=1-2]
	\arrow["{\overline{\gamma}}", shorten <=13pt, shorten >=13pt, Rightarrow, from=1-2, to=3-2]
	\arrow["{\mathbf{p}_{j+1,i+1}}"{xshift=0.1cm,yshift=0.1cm}, shift right=1, shorten <=4pt, shorten >=4pt, Rightarrow, from=1-3, to=0]
\end{tikzcd}}
\end{equation}

\item[(b)] \blue Let $\gamma=\{\sigma, \tau\}$ with $\sigma$ and $\tau$ two 2-cells as in 2.(b) of Construction \ref{kansmallobject}, thus $\sigma \circ h=\tau\circ h$.
\[\begin{tikzcd}[ampersand replacement=\&]
	\&\& {X_{j+1}} \\
	A \& {A'} \&\& {X_{i+1}} \& P
	\arrow["h", from=2-1, to=2-2]
	\arrow["{f//h}", curve={height=-6pt}, from=2-2, to=1-3]
	\arrow[""{name=0, anchor=center, inner sep=0}, "s"', from=2-2, to=2-4]
	\arrow["{p_{i+1}}", from=2-4, to=2-5]
	\arrow["{x_{j+1,i+1}}", curve={height=-6pt}, from=1-3, to=2-4]
	\arrow["\sigma", shift right=5, shorten <=8pt, shorten >=5pt, Rightarrow, from=1-3, to=0]
	\arrow["\tau"', shift left=5, shorten <=8pt, shorten >=5pt, Rightarrow, from=1-3, to=0]
\end{tikzcd}\]

Then, $p_{i+1}\circ \sigma\circ h=p_{i+1}\circ \tau\circ h$, and, since $(p_jf)/h\cong p_{j+1}(f//h)\cong p_{i+1} x_{j+1,i+1}(f//h)$, it follows that $p_{i+1}\circ \sigma=p_{i+1}\circ \tau$. Consequently, for $\gamma=\{\sigma, \tau\}$,  by the 1-dimensional universality of the bicoequifier, there is  a unique (up-to-iso) $1$-cell $p_\gamma\colon C_\gamma \to  P$ with an invertible 2-cell $ \mathbf{p}^c_\gamma\colon p_{i+1}\Rightarrow p_\gamma c_\gamma$:
\begin{equation}\label{eq-coeq}
\adjustbox{scale=0.7}{\begin{tikzcd}
	{X_{i+1}} && {C_{\gamma}} \\
	&& {} \\
	&& P
	\arrow["{c_{\gamma}}", from=1-1, to=1-3]
	\arrow[""{name=0, anchor=center, inner sep=0}, "{p_{i+1}}"', from=1-1, to=3-3]
	\arrow["{p_{\gamma}}", from=1-3, to=3-3]
	\arrow["{\mathbf{p}_{\gamma}^c}"', shorten <=6pt, Rightarrow, from=0, to=1-3]
\end{tikzcd}}
\end{equation}

   \end{enumerate}
   
 Combining (a), see (\ref{eq:def-p_gamma}), and (b), we have a diagram as in (\ref{eq-coeq}) for every $\gamma \in \Gamma$.
\blue Hence, by the universal property of the wide bipushout $X_{i+2}$, we get \black a unique (up-to-iso) morphism $p_{i+2}\colon X_{i+2} \to P$ \blue with invertible 2-cells $\mathbf{p}_{i+1,i+2}\colon p_{i+1}\Rightarrow p_{i+2}x_{i+1,i+2}$ and $\mathbf{p}^d_\gamma\colon p_{i+2}d_\gamma\Rightarrow p_\gamma$ such that 

\begin{equation}\label{eq:def-p_i+2}
    \mathbf{p}_\gamma^c=
\begin{tikzcd}[ampersand replacement=\&]
	{X_{i+1}} \& {C_\gamma} \\
	\& {X_{i+2}} \\
	\&\& P
	\arrow[""{name=0, anchor=center, inner sep=0}, "{p_{i+1}}"', curve={height=50pt}, from=1-1, to=3-3]
	\arrow["{c_\gamma}", from=1-1, to=1-2]
	\arrow[""{name=1, anchor=center, inner sep=0}, "{p_\gamma}", curve={height=-20pt}, from=1-2, to=3-3]
	\arrow[""{name=2, anchor=center, inner sep=0}, from=1-1, to=2-2]
	\arrow[""{name=3, anchor=center, inner sep=0}, "{d_\gamma}"{description}, from=1-2, to=2-2]
	\arrow["{p_{i+2}}"{description}, from=2-2, to=3-3]
	\arrow["{\mathbf{p}_{i+1,i+2}}"{pos=0.7,yshift=0cm,xshift=0.1cm},shift right=1, shorten <=6pt, shorten >=6pt, Rightarrow, from=0, to=2-2]
	\arrow["{\mathbf{p}_{\gamma}^d}", shorten <=7pt, shorten >=7pt, Rightarrow, from=2-2, to=1]
	\arrow["{\delta_\gamma}"{yshift=0.1cm}, shorten <=8pt, shorten >=8pt, Rightarrow, from=2, to=3]
\end{tikzcd}
\end{equation}
\black 
\blue  Put $\mathbf{p}_{0,i+2}=(\mathbf{p}_{i+1,i+2}\circ x_{0,i+1})\cdot \mathbf{p}_{0,i+1}$. We now would like to conclude that  $(p_{i+2},\mathbf{p}_{0,i+2})$ is a Kan extension of $p_0$ along $x_{0,i+2}$ and so \black  $p_{i+2} \cong p_0/x_{0,i+2}$. In order to do so, we consider any $2$-cell $(r,\mu)$ as below:

\[\begin{tikzcd}[ampersand replacement=\&]
	{X_0} \& {X_{i+2}} \& {X_0} \& {X_{i+2}} \\
	P \&\& P
	\arrow[""{name=0, anchor=center, inner sep=0}, "{p_{i+2}}", dotted, from=1-2, to=2-1]
	\arrow[""{name=1, anchor=center, inner sep=0}, "{p_0}"', from=1-1, to=2-1]
	\arrow["{x_{0i}}", from=1-1, to=1-2]
	\arrow["{x_{0,i+2}}", from=1-3, to=1-4]
	\arrow[""{name=2, anchor=center, inner sep=0}, "{p_0}"', from=1-3, to=2-3]
	\arrow[""{name=3, anchor=center, inner sep=0}, "r", from=1-4, to=2-3]
	\arrow["\mu"{yshift=0.1cm}, shorten <=6pt, shorten >=6pt, Rightarrow, from=2, to=3]
	\arrow["{\mathbf{p}_{0,i+2}}"{yshift=0.1cm}, shorten <=6pt, shorten >=6pt, Rightarrow, from=1, to=0]
\end{tikzcd}\]

Since $p_{i+1}\cong p_{0}/x_{0,i+1}$, there is a unique $2$-cell $\overline{\mu}\colon p_{i+1} \Rightarrow rx_{i+1,i+2}$ such that
\begin{equation}\label{mu-barr}
\begin{tikzcd}
	{X_0} & {X_{i+1}} \\
	&& {X_{i+2}} \\
	P
	\arrow["{x_{0,i+1}}", from=1-1, to=1-2]
	\arrow["{x_{i+1,i+2}}", from=1-2, to=2-3]
	\arrow[""{name=0, anchor=center, inner sep=0}, "{p_0}"', from=1-1, to=3-1]
	\arrow["{p_{i+2}}", from=2-3, to=3-1]
	\arrow["\mu"', shorten <=22pt, shorten >=22pt, Rightarrow, from=0, to=2-3]
\end{tikzcd}
\hspace{0.5cm}=\hspace{0.5cm}
\begin{tikzcd}[ampersand replacement=\&]
	{X_0} \& {X_{i+1}} \\
	\&\& {X_{i+2}} \\
	P
	\arrow["{x_{0,i+1}}", from=1-1, to=1-2]
	\arrow["{x_{i+1,i+2}}", from=1-2, to=2-3]
	\arrow[""{name=0, anchor=center, inner sep=0}, "{p_0}"', from=1-1, to=3-1]
	\arrow["{p_{i+2}}", from=2-3, to=3-1]
	\arrow[""{name=1, anchor=center, inner sep=0}, "{p_{i+1}}"{description}, from=1-2, to=3-1]
	\arrow["{\overline{\mu}}"{yshift=0.1cm}, shorten <=20pt, shorten >=20pt, Rightarrow, from=1, to=2-3]
	\arrow["{\mathbf{p}_{0,i+1}}"{pos=0.6,yshift=0.1cm}, shift left=5, shorten <=6pt, shorten >=6pt, Rightarrow, from=0, to=1]
\end{tikzcd}
\end{equation} 
So, for \blue every $\gamma\in \Gamma$, \black we obtain a $2$-cell $\tilde{\mu}$ defined as the pasting diagram below.
\begin{equation}\label{mu-tilde}
\tilde{\mu}:=
\begin{tikzcd}[ampersand replacement=\&]
	\&\& {C_\gamma} \\
	{X_{i+1}} \&\&\&\& P \\
	\& {C_\gamma} \&\& {X_{i+2}}
	\arrow["{c_\gamma}", from=2-1, to=1-3]
	\arrow["{c_\gamma}"', from=2-1, to=3-2]
	\arrow["{p_\gamma}", from=1-3, to=2-5]
	\arrow[""{name=0, anchor=center, inner sep=0}, "{p_{i+1}}"{description}, curve={height=-12pt}, from=2-1, to=2-5]
	\arrow["r"', from=3-4, to=2-5]
	\arrow["{d_\gamma}"', from=3-2, to=3-4]
	\arrow[""{name=1, anchor=center, inner sep=0}, "x"{description, pos=0.8}, from=2-1, to=3-4]
	\arrow["{({\mathbf{p}_\gamma^c})^{-1}}"{yshift=-0.2cm,xshift=0.1cm}, shorten <=2pt, shorten >=4pt, Rightarrow, from=1-3, to=0]
	\arrow["{\delta_\gamma}"{pos=0.3}, shorten <=6pt, shorten >=2pt, Rightarrow, from=1, to=3-2]
	\arrow["{\bar{\mu}}", shorten <=10pt, shorten >=10pt, Rightarrow, from=0, to=1]
\end{tikzcd}
\end{equation}

\blue First, in (a$^{\prime}$) and (b$^{\prime}$) below, we are going to prove that, for every $\gamma\in \Gamma$, there is a unique 2-cell $\hat{\mu}_{\gamma}\colon p_{\gamma} \Rightarrow rd_{\gamma}$ such that $\tilde{\mu}=\hat{\mu}_{\gamma}\circ c_{\gamma}$.

\begin{enumerate}
    \item[(a$^{\prime}$)] For $\gamma$ a 2-cell as in (a), \black we want to show that $\tilde{\mu}$ satisfies the required condition for the 2-dimensional universal property of the bicoequinserter, i.e. to prove that the two pasting diagrams below are equal.

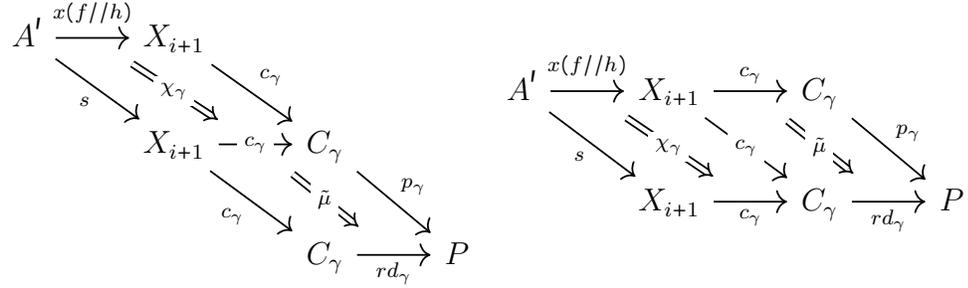
\begin{center}
\begin{tikzcd}[ampersand replacement=\&]
	{A'} \& {X_{i+1}} \\
	\& {X_{i+1}} \& {C_\gamma} \\
	\&\& {C_\gamma} \& P
	\arrow[""{name=0, anchor=center, inner sep=0}, "{c_\gamma}"{description}, from=2-2, to=2-3]
	\arrow["{c_\gamma}"', from=2-2, to=3-3]
	\arrow["{p_\gamma}", from=2-3, to=3-4]
	\arrow[""{name=1, anchor=center, inner sep=0}, "{rd_\gamma}"', from=3-3, to=3-4]
	\arrow[""{name=2, anchor=center, inner sep=0}, "{x(f//h)}"{yshift=0.1cm}, from=1-1, to=1-2]
	\arrow["s"', from=1-1, to=2-2]
	\arrow["{c_\gamma}", from=1-2, to=2-3]
	\arrow["{\chi_\gamma}"{description}, shorten <=17pt, shorten >=16pt, Rightarrow, from=2, to=0]
	\arrow["{\tilde{\mu}}"{description}, shorten <=17pt, shorten >=16pt, Rightarrow, from=0, to=1]
\end{tikzcd}
\begin{tikzcd}[ampersand replacement=\&]
	{A'} \& {X_{i+1}} \& {C_\gamma} \\
	\& {X_{i+1}} \& {C_\gamma} \& P
	\arrow[""{name=0, anchor=center, inner sep=0}, "{c_\gamma}", from=1-2, to=1-3]
	\arrow["{c_\gamma}"{description}, from=1-2, to=2-3]
	\arrow["{p_\gamma}", from=1-3, to=2-4]
	\arrow[""{name=1, anchor=center, inner sep=0}, "{rd_\gamma}"', from=2-3, to=2-4]
	\arrow[""{name=2, anchor=center, inner sep=0}, "{c_\gamma}"', from=2-2, to=2-3]
	\arrow[""{name=3, anchor=center, inner sep=0}, "{x(f//h)}"{yshift=0.1cm}, from=1-1, to=1-2]
	\arrow["s"', from=1-1, to=2-2]
	\arrow["{\tilde{\mu}}"{description}, shorten <=17pt, shorten >=16pt, Rightarrow, from=0, to=1]
	\arrow["{\chi_\gamma}"{description}, shorten <=17pt, shorten >=16pt, Rightarrow, from=3, to=2]
\end{tikzcd}
\end{center}
Since these two 2-cells have as domain
$$p_\gamma c_\gamma x_{j+1,i+1} (f//h)\cong p_{i+1}x_{j+1,i+1} (f//h)\cong p_{j+1}(f//h)\cong (p_jf)/h$$
which is a left Kan extension, then to show that they are equal it sufficies to show that precomposing with $h$ we obtain the same 2-cell. Indeed, the next three pasting diagrams give all the same 2-cell, by the  definition of $\chi_\gamma$ applied twice:

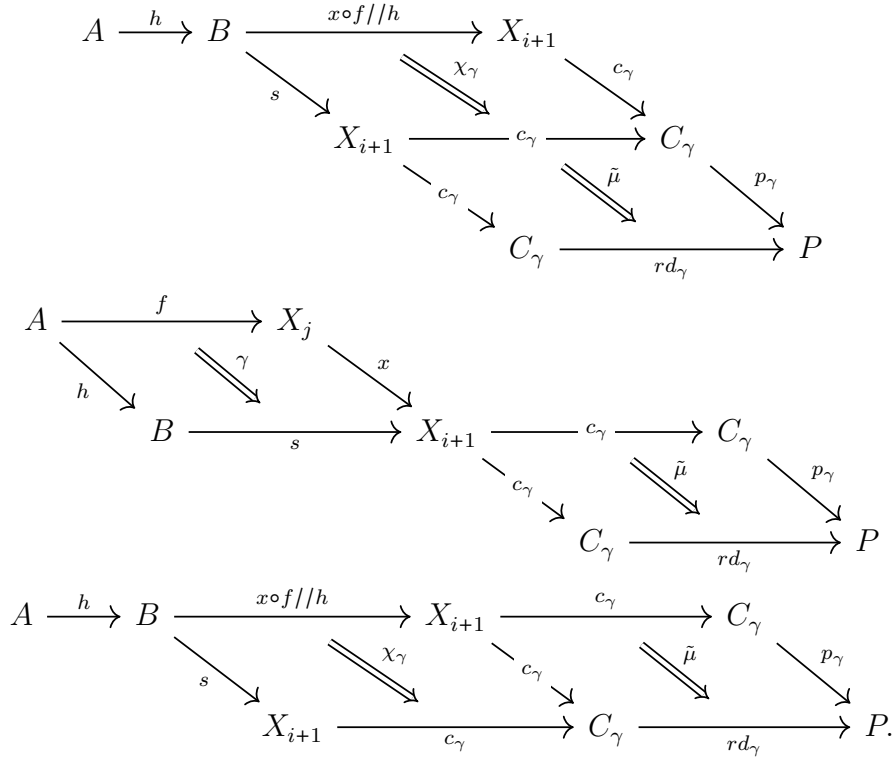
\begin{center}
\begin{tikzcd}[ampersand replacement=\&]
	A \& B \&\& {X_{i+1}} \\
	\&\& {X_{i+1}} \&\& {C_\gamma} \\
	\&\&\& {C_\gamma} \&\& P
	\arrow[""{name=0, anchor=center, inner sep=0}, "{c_\gamma}"{description}, from=2-3, to=2-5]
	\arrow["{c_\gamma}"{description}, from=2-3, to=3-4]
	\arrow[""{name=1, anchor=center, inner sep=0}, "{rd_\gamma }"', from=3-4, to=3-6]
	\arrow["{p_\gamma}", from=2-5, to=3-6]
	\arrow["h", from=1-1, to=1-2]
	\arrow["{c_\gamma}", from=1-4, to=2-5]
	\arrow[""{name=2, anchor=center, inner sep=0}, "{x\circ f//h}", from=1-2, to=1-4]
	\arrow["s"', from=1-2, to=2-3]
	\arrow["{\tilde{\mu}}", shorten <=16pt, shorten >=16pt, Rightarrow, from=0, to=1]
	\arrow["{\chi_\gamma}", shorten <=16pt, shorten >=16pt, Rightarrow, from=2, to=0]
\end{tikzcd} \\
\begin{tikzcd}[ampersand replacement=\&]
	A \&\& {X_j} \\
	\& B \&\& {X_{i+1}} \&\& {C_\gamma} \\
	\&\&\&\& {C_\gamma} \&\& P
	\arrow[""{name=0, anchor=center, inner sep=0}, "{c_\gamma}"{description}, from=2-4, to=2-6]
	\arrow["{c_\gamma}"{description}, from=2-4, to=3-5]
	\arrow[""{name=1, anchor=center, inner sep=0}, "{rd_\gamma }"', from=3-5, to=3-7]
	\arrow["{p_\gamma}", from=2-6, to=3-7]
	\arrow["h"', from=1-1, to=2-2]
	\arrow[""{name=2, anchor=center, inner sep=0}, "s"', from=2-2, to=2-4]
	\arrow[""{name=3, anchor=center, inner sep=0}, "f", from=1-1, to=1-3]
	\arrow["x", from=1-3, to=2-4]
	\arrow["{\tilde{\mu}}", shorten <=16pt, shorten >=16pt, Rightarrow, from=0, to=1]
	\arrow["\gamma", shorten <=16pt, shorten >=16pt, Rightarrow, from=3, to=2]
\end{tikzcd}\\
\begin{tikzcd}[ampersand replacement=\&]
	A \& B \&\& {X_{i+1}} \&\& {C_\gamma} \\
	\&\& {X_{i+1}} \&\& {C_\gamma} \&\& P.
	\arrow[""{name=0, anchor=center, inner sep=0}, "{c_\gamma}", from=1-4, to=1-6]
	\arrow["{c_\gamma}"{description}, from=1-4, to=2-5]
	\arrow[""{name=1, anchor=center, inner sep=0}, "{rd_\gamma }"', from=2-5, to=2-7]
	\arrow["{p_\gamma}", from=1-6, to=2-7]
	\arrow["h", from=1-1, to=1-2]
	\arrow[""{name=2, anchor=center, inner sep=0}, "{x\circ f//h}", from=1-2, to=1-4]
	\arrow["s"', from=1-2, to=2-3]
	\arrow[""{name=3, anchor=center, inner sep=0}, "{c_\gamma}"', from=2-3, to=2-5]
	\arrow["{\tilde{\mu}}", shorten <=16pt, shorten >=16pt, Rightarrow, from=0, to=1]
	\arrow["{\chi_\gamma}", shorten <=16pt, shorten >=16pt, Rightarrow, from=2, to=3]
\end{tikzcd}
\end{center}

Consequently, we may apply the 2-dimensional aspect of the universality of each bicoequinserter $c_\gamma$, obtaining a unique $2$-cell $\hat{\mu}_\gamma$ such that $\hat{\mu}_\gamma\circ c_\gamma=\tilde{\mu}$.

\item[(b$^{\prime}$)] \blue For $\gamma=\{\sigma,\tau\}$ as in (b), $c_{\gamma}$ is the bicoequifier of the two 2-cells $\sigma$ and $\tau$. Thus, the 2-dimensional aspect of its universality ensures that there is a unique
$\hat{\mu}_{\gamma}\colon p_{\gamma}\Rightarrow rd_{\gamma}$ such that
\begin{equation}\label{mu-hat}\tilde{\mu}=\hat{\mu}_{\gamma}\circ c_{\gamma}.\end{equation} 

\end{enumerate}

\blue Combining (a$^{\prime}$) and (b$^{\prime}$), \black the equality $\hat{\mu}_\gamma\circ c_\gamma=\tilde{\mu}$ holds for all $\gamma\in \Gamma$. Hence, taking into account the pasting diagram
\[\begin{tikzcd}[ampersand replacement=\&]
	\& {X_{i+2}} \\
	{C_\gamma} \&\& P
	\arrow[""{name=0, anchor=center, inner sep=0}, "{rd_\gamma}"', curve={height=12pt}, from=2-1, to=2-3]
	\arrow[""{name=1, anchor=center, inner sep=0}, "{p_\gamma}"{description}, curve={height=-12pt}, from=2-1, to=2-3]
	\arrow["{d_\gamma}", curve={height=-6pt}, from=2-1, to=1-2]
	\arrow["{p_{i+2}}", curve={height=-6pt}, from=1-2, to=2-3]
	\arrow["{\hat{\mu_\gamma}}"{xshift=0.1cm}, shorten <=6pt, shorten >=6pt, Rightarrow, from=1, to=0]
	\arrow["{\mathbf{p}_\gamma^d}"{xshift=0.1cm}, shorten <=4pt, shorten >=6pt, Rightarrow, from=1-2, to=1]
\end{tikzcd}\]
and the 2-dimensional universal property of the wide bipushout, there exists a unique \[\stackrel{\bullet}{\mu}\colon  p_{i+2} \Rightarrow r\] such that \blue $\stackrel{\bullet}{\mu}\circ d_\gamma = \hat{\mu}_\gamma\cdot \mathbf{p}^d_{\gamma}$. Consequently,  using this last equality and the equalities (\ref{eq:def-p_i+2}), (\ref{mu-barr}), (\ref{mu-tilde}) and (\ref{mu-hat}), we obtain:

$\begin{array}{rl}
  (\stackrel{\bullet}{\mu} \circ x_{0,i+2})\cdot \mathbf{p}_{0,i+2} &=(((\stackrel{\bullet}{\mu} \circ x_{i+1,i+2})\cdot \mathbf{p}_{i+1,i+2})\circ x_{0,i+1})\cdot \mathbf{p}_{0,i+1}\\
  &=(((r\circ \delta_{\gamma})\cdot ((\hat{\mu}_{\gamma}\cdot \mathbf{p}_{\gamma}^d)\circ c_{\gamma})\cdot(p_{i+2}\circ \delta_{\gamma}^{-1})\cdot \mathbf{p}_{i+1,i+2})\circ x_{0,i+1})\cdot \mathbf{p}_{0,i+1}\\
   &=(\bar{\mu}\circ x_{0,i+1})\cdot \mathbf{p}_{0,i+1}=\mu.
\end{array}
$\\
\black
The unicity of $\stackrel{\bullet}{\mu}$ is a routine check.

\end{enumerate}

Concerning the Kan injectivity of 1-cells, let $u\colon P\to Q$ be in $\LInj(\maps)$ and set $q:=up$. We want to show that $up_{i+2}\cong q_{i+2}$. \blue For each 2-cell $\gamma$ of type 2.(a), \black in diagram (\ref{diag:bar_gamma}),
put $\bar{\gamma}_P:=\bar{\gamma}$ and, analogously, use the notation $\bar{\gamma}_Q$ for the $\bar{\gamma}$ corresponding to $Q$, instead of $P$. Since $u\colon P\to Q$ preserves left Kan extensions along maps in $\maps$,  and $up_{i+1}\cong q_{i+1}$,  we have that $u\bar{\gamma}_P$ coincides with $\bar{\gamma}_Q$ up to an invertible 2-cell. Then, as pasting  diagram (\ref{eq:def-p_gamma}) with u the equation still holds, we get that  $(up_\gamma,u\mathbf{p}_\gamma^c)$ satisfies the same property that $(q_\gamma,\mathbf{q}_\gamma^c)$ does. Hence,  $up_\gamma\cong q_\gamma$. 
This equality is also verified for $\gamma=\{\sigma,\tau\}$ as in 2.(b); this follows immediately from the definition of $p_{\gamma}$ and $q_{\gamma}$ in this case. \black Since the equality holds for all $\gamma$, by the universality of the wide bipushout (\ref{diag:w-bipu0}), we conclude that $up_{i+2}\cong q_{i+2}$.
\end{proof}

\section{KZ-monadicity via the pseudochain}\label{TheTheorem}
In this section we consider a fixed 2-category $\catk$ with (weighted) bicolimits. A key requisite in the classical Small Object Argument and Orthogonal Subcategory Problem is a convenient concept of smallness for objects. Here we make use of the notion below.

\begin{definition}\label{def:small} An object $A$ is \textbf{$\lambda$-small}, for $\lambda$ an infinite regular cardinal, if the 2-functor
  $$\catk(A,-)\colon\catk\to\Cat$$
  preserves bicolimits of $\lambda$-pseudochains.

  Explicitly: For every $\lambda$-pseudochain $(X_i)_{i<\lambda}$, with bicolimit coprojections $l_i\colon X_i\to$~$L$, we have:
  \begin{enumerate}
    \item every morphism $a\colon A\to L$ factorises through some $X_i$ (up-to-iso);
    $$\xymatrix{&L\\
    A\ar[ru]^a_{\; \cong}\ar[r]_{a'}&X_i\ar[u]_{l_i}}$$
    \item for every 2-cell of the form
\[\begin{tikzcd}
	& {X_i} \\
	A && L \\
	& {X_{i'}}
	\arrow["f", from=2-1, to=1-2]
	\arrow["g"', from=2-1, to=3-2]
	\arrow["{l_i}", from=1-2, to=2-3]
	\arrow["{l_{i'}}"', from=3-2, to=2-3]
	\arrow["\alpha"{xshift=0.1cm}, shorten <=15pt, shorten >=15pt, Rightarrow, from=1-2, to=3-2]
\end{tikzcd}\]
    there is some $j\geq i,i'$ and a 2-cell $\overline{\alpha}$ such that
    \begin{center}
\begin{tikzcd}
	& {X_i} \\
	A && {X_j} & L \\
	& {X_{i'}}
	\arrow["f", from=2-1, to=1-2]
	\arrow["g"', from=2-1, to=3-2]
	\arrow[from=1-2, to=2-3]
	\arrow[from=3-2, to=2-3]
	\arrow["{\overline{\alpha}}"{xshift=0.1cm}, shorten <=15pt, shorten >=15pt, Rightarrow, from=1-2, to=3-2]
	\arrow["{l_j}"{description}, from=2-3, to=2-4]
	\arrow[""{name=0, anchor=center, inner sep=0}, "{l_i}", curve={height=-12pt}, from=1-2, to=2-4]
	\arrow[""{name=1, anchor=center, inner sep=0}, "{l_{i'}}"', curve={height=12pt}, from=3-2, to=2-4]
	\arrow["\cong"{description}, Rightarrow, draw=none, from=2-3, to=1]
	\arrow["\cong"{description}, Rightarrow, draw=none, from=2-3, to=0]
\end{tikzcd} $=\;\alpha$.
\end{center}
\item \blue For every equality of 2-cells
\[\begin{tikzcd}
	A && {X_i} & L & {=} & A && {X_{i'}} & L
	\arrow[""{name=0, anchor=center, inner sep=0}, "f", curve={height=-12pt}, from=1-1, to=1-3]
	\arrow[""{name=1, anchor=center, inner sep=0}, "g"', curve={height=12pt}, from=1-1, to=1-3]
	\arrow["{l_i}", from=1-3, to=1-4]
	\arrow["{l_{i'}}", from=1-8, to=1-9]
	\arrow[""{name=2, anchor=center, inner sep=0}, "{f'}", curve={height=-12pt}, from=1-6, to=1-8]
	\arrow[""{name=3, anchor=center, inner sep=0}, "{g'}"', curve={height=12pt}, from=1-6, to=1-8]
	\arrow["\alpha"{xshift=0.1cm}, shorten <=3pt, shorten >=3pt, Rightarrow, from=0, to=1]
	\arrow["\beta"{xshift=0.1cm}, shorten <=3pt, shorten >=3pt, Rightarrow, from=2, to=3]
\end{tikzcd}\]
there is $j\geq i,i'$ such that $x_{ij}\circ \alpha=x_{i'j}\circ \beta$:
\[\begin{tikzcd}
	A && {X_i} & {X_j} & {=} & A && {X_{i'}} & {X_j.}
	\arrow[""{name=0, anchor=center, inner sep=0}, "f", curve={height=-12pt}, from=1-1, to=1-3]
	\arrow[""{name=1, anchor=center, inner sep=0}, "g"', curve={height=12pt}, from=1-1, to=1-3]
	\arrow["{x_{ij}}", from=1-3, to=1-4]
	\arrow["{x_{i'j}}", from=1-8, to=1-9]
	\arrow[""{name=2, anchor=center, inner sep=0}, "{f'}", curve={height=-12pt}, from=1-6, to=1-8]
	\arrow[""{name=3, anchor=center, inner sep=0}, "{g'}"', curve={height=12pt}, from=1-6, to=1-8]
	\arrow["\alpha"{xshift=0.1cm}, shorten <=3pt, shorten >=3pt, Rightarrow, from=0, to=1]
	\arrow["\beta"{xshift=0.1cm}, shorten <=3pt, shorten >=3pt, Rightarrow, from=2, to=3]
\end{tikzcd}\]
\end{enumerate}
\black
An object $X$ in $\catk$ is said to be \textbf{small} if it is $\lambda$-small for some infinite regular cardinal~$\lambda$.
\end{definition}

\begin{remark} An example of $\lambda$-small object is given by the notion of $\lambda$-bipresentable object studied in great detail in \cite{dilib-osm_biacc}.
Recall that an object $A$ of $\catk$ is said to be $\mathbf{\lambda}$\textbf{-bipresentable}  if the 2-functor
  $\catk(A,-)\colon\catk\to\Cat$
  preserves filtered bicolimits in the sense of \cite[2.1.3]{dilib-osm_biacc}. Notice that in $1$-dimensional category theory the two notions collapse due to \cite[1.6]{adamek_rosicky_1994}. The 2-dimensional aspects of such a result are unknown at the current state of art. 
  
  \blue Other generalisations of the notion of presentable object have been studied in \cite{street1976limits, kelly1982structures}, we refer the interested reader to \cite[Remark 3.1.9]{dilib-osm_biacc} for a more detailed analysis.
  \end{remark}

\begin{theorem}
\label{thm:linj-kz-adj}
Let $\catk$ be a locally small 2-category with small bicolimits and such that all objects are small.  Then, for any set $\maps$ of 1-cells in $\catk$, the inclusion 2-functor
$$\slinj(\maps)\hookrightarrow\catk$$
is the right part of a KZ-adjunction. Moreover, $\slinj(\maps)$ is 2-equivalent to the corresponding Eilenberg-Moore 2-category.
\end{theorem}

\begin{proof}
Since $\maps$ is a set and every object of $\catk$ is small, there is some infinite regular cardinal $\kappa$ such that all domains and codomains of morphisms of $\maps$ are $\kappa$-small.


We will use Theorem \ref{thm:marm-wood}, setting $\cata:=\slinj(\maps)$, $DX:=X_\kappa$ and $d_X$ as the 1-cells $x_{0,\kappa}\colon X=X_0\to X_\kappa$, to prove that the inclusion 2-functor
$$\slinj(\maps)\hookrightarrow\catk$$
is the right part of a KZ-adjunction. In Lemma \ref{lem:ext-to-P}, we have already proved that $\slinj(\maps)$ is a sub-2-category of $\slinj(\lbrace x_{0\kappa}\mid X\in \catk\rbrace)$. Therefore, we just need to prove the following two properties:
\begin{enumerate}
    \item[(1)] For all $X\in\catk$, $X_\kappa\in\slinj(\maps)$ and, for any $p\colon X\to P$ with $P\in\slinj(\maps)$, the morphism $p/x_{0,\kappa}$ belongs to $\slinj(\maps)$.

    \item[(2)] Every $x_{0,\kappa}$ is dense (see Definition \ref{def:dense}).

\end{enumerate}

Let us prove these properties.

\begin{enumerate}
    \item[(1)] First, we will prove that $X_\kappa\in\slinj(\maps)$. Given $h\colon A\to A'\in\maps$ and $f\colon A\to X_\kappa$, since $A$ is $\kappa$-small and $X_\kappa=\bicolim_{j<\kappa}X_j$, there is some even ordinal $i$   and an invertible 2-cell $\phi^{f}_{ik}\colon f\Rightarrow  x_{i,\kappa}\circ f'$ with $f'\colon A\to X_i$.  We claim that $f/h:=x_{i+1,\kappa}\circ f'//h$ and the invertible 2-cell (see (\ref{diag:w-bipu}))
\begin{center}
    $\xi_f^h:=\quad$
\begin{tikzcd}[ampersand replacement=\&]
	A \&\& {A'} \\
	{X_i} \&\& {X_{i+1}} \\
	{X_\kappa}
	\arrow["h", from=1-1, to=1-3]
	\arrow[""{name=0, anchor=center, inner sep=0}, "{f'}"{description}, from=1-1, to=2-1]
	\arrow["{x_{i,i+1}}", from=2-1, to=2-3]
	\arrow[""{name=1, anchor=center, inner sep=0}, "{f'//h}", from=1-3, to=2-3]
	\arrow[""{name=2, anchor=center, inner sep=0}, "{x_{i+1,\kappa}}", from=2-3, to=3-1]
	\arrow[""{name=3, anchor=center, inner sep=0}, "{x_{i,\kappa}}"{description}, from=2-1, to=3-1]
	\arrow[""{name=4, anchor=center, inner sep=0}, "f"', curve={height=30pt}, from=1-1, to=3-1]
	\arrow["{\phi^{f}_{ik}}"{pos=0.6,xshift=0.1cm,yshift=0.1cm}, shorten <=5pt, shorten >=1pt, Rightarrow, from=4, to=2-1]
	\arrow["{\tilde{\xi}^h_{f'}}", shorten <=26pt, shorten >=26pt, Rightarrow, from=0, to=1]
	\arrow["{({\mathbf{x}_{i,\kappa}^{i+1}})^{-1}}", shift left=1, shorten <=12pt, shorten >=12pt, Rightarrow, from=3, to=2]
\end{tikzcd}
\end{center}
provides a left Kan extension of $f$ along $h$. To prove this, let us consider a 2-cell
\[\begin{tikzcd}
	A & {A'} \\
	& {X_\kappa.}
	\arrow["h", from=1-1, to=1-2]
	\arrow[""{name=0, anchor=center, inner sep=0}, "f"', from=1-1, to=2-2]
	\arrow[""{name=1, anchor=center, inner sep=0}, "g", from=1-2, to=2-2]
	\arrow["\beta", shorten <=5pt, shorten >=5pt, Rightarrow, from=0, to=1]
\end{tikzcd}\]
Since $A'$ is $\kappa$-small, $g$ factorises through some $X_j$, i.e. there is an ordinal $j$ and an invertible 2-cell $\phi^{g}_{jk}\colon g\Rightarrow x_{j,\kappa}g'$ with $g'\colon A'\to X_j$. Without loss of generality, we may assume that this $j$ is even and $i\leq j$. This way, we can consider the 2-cell $\beta'$ given by the pasting below.
\begin{equation}\label{eq:def-beta'}
\beta':=\quad
\begin{tikzcd}[ampersand replacement=\&]
	\& {X_i} \&\& {X_j} \\
	A \& {} \&\&\& {X_\kappa} \\
	\& {A'} \&\& {X_j}
	\arrow["h"', from=2-1, to=3-2]
	\arrow["{g'}"', from=3-2, to=3-4]
	\arrow["{x_{j,\kappa}}"', from=3-4, to=2-5]
	\arrow["{f'}", from=2-1, to=1-2]
	\arrow["{x_{i,j}}", from=1-2, to=1-4]
	\arrow["{x_{j,\kappa}}", from=1-4, to=2-5]
	\arrow[""{name=0, anchor=center, inner sep=0}, "g"{description}, curve={height=-6pt}, from=3-2, to=2-5]
	\arrow["f"{description}, from=2-1, to=2-5]
	\arrow["\beta"{xshift=0.1cm}, shorten <=4pt, shorten >=6pt, Rightarrow, from=2-2, to=3-2]
	\arrow[""{name=1, anchor=center, inner sep=0}, "{x_{i\kappa}}"{description}, curve={height=6pt}, from=1-2, to=2-5]
	\arrow["{(\phi^{f}_{ik})^{-1}}"{xshift=0.1cm}, shorten <=7pt, shorten >=7pt, Rightarrow, from=1-2, to=2-2]
	\arrow["{\phi^{g}_{jk}}"{pos=0.8}, shorten <=7pt, shorten >=4pt, Rightarrow, from=0, to=3-4]
	\arrow["{\mathbf{x}^j_{i\kappa}}"', shorten <=1pt, shorten >=8pt, Rightarrow, from=1-4, to=1]
\end{tikzcd}
\end{equation}
Therefore, since $A$ is $\kappa$-small, there is some $m\geq j$ (which we may assume even) and a 2-cell $\overline{\beta}$ such that (see part 2 of Definition \ref{def:small})
\begin{center}
    $\beta'=$
\begin{tikzcd}[ampersand replacement=\&]
	\& {X_i} \&\& {X_j} \\
	A \&\&\&\& {X_{m+1}} \& {X_\kappa} \\
	\& {A'} \& {X_j}
	\arrow[""{name=0, anchor=center, inner sep=0}, "{x_{i,j}}", from=1-2, to=1-4]
	\arrow["{x_{j,m+1}}"{description}, from=3-3, to=2-5]
	\arrow["{x_{j,m+1}}"{description, pos=0.4}, curve={height=-6pt}, from=1-4, to=2-5]
	\arrow[""{name=1, anchor=center, inner sep=0}, "{g'}"', from=3-2, to=3-3]
	\arrow["h"', from=2-1, to=3-2]
	\arrow["{f'}", from=2-1, to=1-2]
	\arrow["x"{description}, from=2-5, to=2-6]
	\arrow[""{name=2, anchor=center, inner sep=0}, "{x_{j,\kappa}}", curve={height=-12pt}, from=1-4, to=2-6]
	\arrow[""{name=3, anchor=center, inner sep=0}, "{x_{j,\kappa}}"', curve={height=12pt}, from=3-3, to=2-6]
	\arrow[""{name=4, anchor=center, inner sep=0}, "x"{description}, from=1-2, to=2-5]
	\arrow["{\mathbf{x}^{-1}}"{xshift=0.05cm,yshift=0.1cm}, shorten <=7pt, shorten >=4pt, Rightarrow, from=2, to=2-5]
	\arrow["{\mathbf{x}}"{pos=0.4,xshift=0.05cm,yshift=0.1cm}, shorten <=5pt, shorten >=8pt, Rightarrow, from=2-5, to=3]
	\arrow["{\overline{\beta}}"{xshift=0.1cm,yshift=0.2cm}, shorten <=30pt, shorten >=25pt, Rightarrow, from=0, to=1]
	\arrow["{\mathbf{x}}"{pos=0.2}, shorten <=3pt, shorten >=5pt, Rightarrow, from=1-4, to=4]
\end{tikzcd}
\end{center}
The 2-cell $\overline{\beta}$
is of the form of the 2-cells $\gamma$  considered in {\blue 2.(a) of Construction \ref{kansmallobject} for obtaining $X_{m+2}$}, so let us consider the bicoequinserter associated to it, which we denote with
\[\begin{tikzcd}[ampersand replacement=\&]
	\& {X_{i+1}} \& {X_{m+1}} \\
	{A'} \&\&\& {C_{\overline{\beta}}} \\
	\& {X_j} \& {X_{m+1}}
	\arrow["{c_{\overline{\beta}}}", from=1-3, to=2-4]
	\arrow["{c_{\overline{\beta}}}"', from=3-3, to=2-4]
	\arrow["{f//h}", from=2-1, to=1-2]
	\arrow["{g'}"', from=2-1, to=3-2]
	\arrow[""{name=0, anchor=center, inner sep=0}, "{x_{j,m+1}}"'{yshift=-0.1cm}, from=3-2, to=3-3]
	\arrow[""{name=1, anchor=center, inner sep=0}, "{x_{i+1,m+1}}"{yshift=0.1cm}, from=1-2, to=1-3]
	\arrow["{\chi_{\overline{\beta}}}"{xshift=0.2cm}, shorten <=25pt, shorten >=25pt, Rightarrow, from=1, to=0]
\end{tikzcd}\]
We set $\tilde{\beta}:=(\phi_{j\kappa}^g)^{-1}\cdot\beta''\colon x_{i+1,\kappa}\circ f'//h\Rightarrow g$ where we define $\beta''$ as the pasting below. 
\[\beta'':=\quad\begin{tikzcd}[ampersand replacement=\&]
	\& {X_{i+1}} \& {X_{m+1}} \\
	{A'} \&\&\& {C_{\overline{\beta}}} \& {X_{m+2}} \& {X_\kappa} \\
	\& {X_j} \& {X_{m+1}}
	\arrow["{c_{\overline{\beta}}}"{description}, from=1-3, to=2-4]
	\arrow["{c_{\overline{\beta}}}"{description}, from=3-3, to=2-4]
	\arrow["{f'//h}", from=2-1, to=1-2]
	\arrow["{g'}"', from=2-1, to=3-2]
	\arrow[""{name=0, anchor=center, inner sep=0}, "x", from=3-2, to=3-3]
	\arrow[""{name=1, anchor=center, inner sep=0}, "x"', from=1-2, to=1-3]
	\arrow["{d_{\overline{\beta}}}"{description}, from=2-4, to=2-5]
	\arrow["x"{description}, from=2-5, to=2-6]
	\arrow[""{name=2, anchor=center, inner sep=0}, "x"{description}, curve={height=-12pt}, from=1-3, to=2-5]
	\arrow[""{name=3, anchor=center, inner sep=0}, "x"{description}, curve={height=12pt}, from=3-3, to=2-5]
	\arrow[""{name=4, anchor=center, inner sep=0}, "{x_{i+1,\kappa}}", curve={height=-50pt}, from=1-2, to=2-6]
	\arrow[""{name=5, anchor=center, inner sep=0}, "{x_{j,\kappa}}"', curve={height=50pt}, from=3-2, to=2-6]
	\arrow["{\delta_{\overline{\beta}}}"{yshift=-0.1cm,xshift=0.1cm}, shorten <=4pt, shorten >=6pt, Rightarrow, from=2-4, to=3]
	\arrow["{{\delta_{\overline{\beta}}}^{-1}}"{yshift=0.1cm,xshift=0.1cm}, shorten <=6pt, shorten >=4pt, Rightarrow, from=2, to=2-4]
	\arrow["{\mathbf{x}^{-1}}"{yshift=0.3cm,xshift=0.1cm}, shorten <=7pt, shorten >=7pt, Rightarrow, from=4, to=2]
	\arrow["{\mathbf{x}}"{xshift=0.2cm,yshift=0.1cm}, shorten <=7pt, shorten >=7pt, Rightarrow, from=3, to=5]
	\arrow["{\chi_{\overline{\beta}}}"{xshift=0.2cm}, shorten <=25pt, shorten >=25pt, Rightarrow, from=1, to=0]
\end{tikzcd}\]




\blue Now, we need to prove that $\tilde{\beta}$ is the only 2-cell such that $(\tilde{\beta}\circ h)\cdot \xi_f^h=\beta$; to see the last equality, we equivalently show that $\beta''$ satisfies the following equation:
\begin{equation}
\label{eq:X_k-kan-ext}
\scalebox{0.9}{
\begin{tikzcd}[ampersand replacement=\&]
	A \& {A'} \\
	{X_i} \& {X_{i+1}} \& {X_j} \\
	\& {X_\kappa}
	\arrow[""{name=0, anchor=center, inner sep=0}, "{f'//h}"{description}, from=1-2, to=2-2]
	\arrow[""{name=1, anchor=center, inner sep=0}, "x"{description}, from=2-2, to=3-2]
	\arrow["h", from=1-1, to=1-2]
	\arrow[""{name=2, anchor=center, inner sep=0}, "f'"', from=1-1, to=2-1]
	\arrow[from=2-1, to=2-2]
	\arrow[""{name=3, anchor=center, inner sep=0}, "{x_{i,\kappa}}"', from=2-1, to=3-2]
	\arrow["{g'}", curve={height=-6pt}, from=1-2, to=2-3]
	\arrow["{x_{j,\kappa}}", curve={height=-6pt}, from=2-3, to=3-2]
	\arrow["{\beta''}", shift right=2, shorten <=6pt, shorten >=6pt, Rightarrow, from=2-2, to=2-3]
	\arrow["{\mathbf{x}^{-1}}"{yshift=0.1cm}, shorten <=7pt, shorten >=7pt, Rightarrow, from=3, to=1]
	\arrow["{\tilde{\xi}_{f'}^h}",shift right=1, shorten <=17pt, shorten >=17pt, Rightarrow, from=2, to=0]
\end{tikzcd}}\quad=\quad
\scalebox{0.9}{
\begin{tikzcd}[ampersand replacement=\&]
	A \& {A'} \\
	\&\& {X_j} \\
	{X_i} \& {X_\kappa}
	\arrow[""{name=0, anchor=center, inner sep=0}, "f"{description, pos=0.4}, from=1-1, to=3-2]
	\arrow["h", from=1-1, to=1-2]
	\arrow[""{name=1, anchor=center, inner sep=0}, "g"{description, pos=0.6}, from=1-2, to=3-2]
	\arrow["{g'}", from=1-2, to=2-3]
	\arrow["{x_{j,\kappa}}", from=2-3, to=3-2]
	\arrow[""{name=2, anchor=center, inner sep=0}, "{f'}"', from=1-1, to=3-1]
	\arrow["{x_{i,\kappa}}"', from=3-1, to=3-2]
	\arrow["{{(\phi_{i\kappa}^f})^{-1}}"'{pos=0.7,yshift=-0.1cm}, shift right=4, shorten <=6pt, shorten >=6pt, Rightarrow, from=2, to=0]
	\arrow["\beta", shift left=3, shorten <=6pt, shorten >=6pt, Rightarrow, from=0, to=1]
	\arrow["{{\phi_{j\kappa}^g}}"{yshift=0.1cm}, shift right=1, shorten <=10pt, shorten >=10pt, Rightarrow, from=1, to=2-3]
\end{tikzcd}}
\end{equation}
Starting from the left hand side above, \black we get, by successively using the definitions of $\chi_{\overline{\beta}}$ and $\overline{\beta}$:
\begin{center}
\adjustbox{scale=0.9}{
\begin{tikzcd}[ampersand replacement=\&]
	\& {X_i} \& {X_{i+1}} \& {X_{m+1}} \\
	A \& {A'} \&\&\& {C_{\overline{\beta}}} \& {X_{m+2}} \& {X_\kappa} \\
	\&\& {X_j} \& {X_{m+1}}
	\arrow["{c_{\overline{\beta}}}"{description}, from=1-4, to=2-5]
	\arrow["{c_{\overline{\beta}}}"{description}, from=3-4, to=2-5]
	\arrow["{f//h}"{description}, from=2-2, to=1-3]
	\arrow["{g'}"', from=2-2, to=3-3]
	\arrow[""{name=0, anchor=center, inner sep=0}, from=3-3, to=3-4]
	\arrow[""{name=1, anchor=center, inner sep=0}, "{x_{i+1,m+1}}", from=1-3, to=1-4]
	\arrow["{d_{\overline{\beta}}}"{description}, from=2-5, to=2-6]
	\arrow["x"{description}, from=2-6, to=2-7]
	\arrow[""{name=2, anchor=center, inner sep=0}, "{x_{m+1,m+2}}"{description}, curve={height=-12pt}, from=1-4, to=2-6]
	\arrow[""{name=3, anchor=center, inner sep=0}, "{x_{m+1,m+2}}"{description}, curve={height=12pt}, from=3-4, to=2-6]
	\arrow["h"', from=2-1, to=2-2]
	\arrow["{f'}", from=2-1, to=1-2]
	\arrow["{x}", from=1-2, to=1-3]
	\arrow["{\tilde{\xi}_f^h}"', shift left=2, shorten <=4pt, shorten >=4pt, Rightarrow, from=1-2, to=2-2]
	\arrow[""{name=4, anchor=center, inner sep=0}, curve={height=50pt}, from=3-3, to=2-7]
	\arrow[""{name=5, anchor=center, inner sep=0}, "{x_{i,\kappa}}", curve={height=-70pt}, from=1-2, to=2-7]
	\arrow["{\delta_{\overline{\beta}}}", shorten <=6pt, shorten >=4pt, Rightarrow, from=2, to=2-5]
	\arrow["{{\delta_{\overline{\beta}}}^{-1}}", shorten <=4pt, shorten >=6pt, Rightarrow, from=2-5, to=3]
	\arrow["{\chi_{\overline{\beta}}}", shorten <=25pt, shorten >=25pt, Rightarrow, from=1, to=0]
	\arrow["{\mathbf{x}^{-1}}", shorten <=7pt, shorten >=7pt, Rightarrow, from=5, to=1-4]
	\arrow["{\mathbf{x}}", shorten <=7pt, shorten >=7pt, Rightarrow, from=3, to=4]
\end{tikzcd}}
{\large $=$}  \\
\adjustbox{scale=0.9}{
\begin{tikzcd}[ampersand replacement=\&]
	\& {X_i} \& {X_{i+1}} \\
	A \&\&\& {X_{m+1}} \& {C_{\overline{\beta}}} \& {X_{m+2}} \& {X_\kappa} \\
	\& {A'} \& {X_j}
	\arrow["{c_{\overline{\beta}}}"{description}, from=2-4, to=2-5]
	\arrow[""{name=0, anchor=center, inner sep=0}, "{g'}"', from=3-2, to=3-3]
	\arrow["{x_{i+1,m+1}}"{description}, from=1-3, to=2-4]
	\arrow["{d_{\overline{\beta}}}"{description}, from=2-5, to=2-6]
	\arrow["x"{description}, from=2-6, to=2-7]
	\arrow[""{name=1, anchor=center, inner sep=0}, "x", curve={height=-30pt}, from=2-4, to=2-6]
	\arrow["h"', from=2-1, to=3-2]
	\arrow["{f'}", from=2-1, to=1-2]
	\arrow[""{name=2, anchor=center, inner sep=0}, "x", from=1-2, to=1-3]
	\arrow[""{name=3, anchor=center, inner sep=0}, "{x_{j,\kappa}}"', curve={height=40pt}, from=3-3, to=2-7]
	\arrow[""{name=4, anchor=center, inner sep=0}, "{x_{i,\kappa}}", curve={height=-50pt}, from=1-2, to=2-7]
	\arrow[""{name=5, anchor=center, inner sep=0}, "{x_{m+1,m+2}}"{description}, curve={height=30pt}, from=2-4, to=2-6]
	\arrow["{x_{j,m+1}}"{description}, from=3-3, to=2-4]
	\arrow["{\delta_{\overline{\beta}}}"{xshift=0.1cm,yshift=-0.1cm}, shorten <=4pt, shorten >=1pt, Rightarrow, from=1, to=2-5]
	\arrow["{\mathbf{x}^{-1}}"{pos=0.3}, shorten <=12pt, shorten >=20pt, Rightarrow, from=4, to=2-4]
	\arrow["{{\delta_{\overline{\beta}}}^{-1}}"{pos=0.3,yshift=0.1cm,xshift=0.1cm}, shorten <=1pt, shorten >=4pt, Rightarrow, from=2-5, to=5]
	\arrow["{\mathbf{x}}"{yshift=-0.2cm,xshift=0.1cm}, shorten <=8pt, shorten >=8pt, Rightarrow, from=5, to=3]
	\arrow["{\overline{\beta}}"{xshift=0.2cm}, shorten <=25pt, shorten >=25pt, Rightarrow, from=2, to=0]
\end{tikzcd}}
{\large $=$}  \\
\adjustbox{scale=0.9}{
\begin{tikzcd}[ampersand replacement=\&]
	\& {X_i} \& {X_{i+1}} \\
	A \&\&\& {X_{m+1}} \&\& {X_\kappa} \\
	\& {A'} \& {X_j}
	\arrow[""{name=0, anchor=center, inner sep=0}, "{g'}"', from=3-2, to=3-3]
	\arrow["{x_{i+1,m+1}}"{description}, from=1-3, to=2-4]
	\arrow["h"', from=2-1, to=3-2]
	\arrow["{f'}", from=2-1, to=1-2]
	\arrow[""{name=1, anchor=center, inner sep=0}, "x"{description}, from=1-2, to=1-3]
	\arrow[""{name=2, anchor=center, inner sep=0}, "{x_{j,\kappa}}"', curve={height=12pt}, from=3-3, to=2-6]
	\arrow[""{name=3, anchor=center, inner sep=0}, "{x_{i,\kappa}}", curve={height=-30pt}, from=1-2, to=2-6]
	\arrow["{x_{j,m+1}}"{description}, from=3-3, to=2-4]
	\arrow["{x_{m+1,\kappa}}"{description}, from=2-4, to=2-6]
	\arrow["{\mathbf{x}^{-1}}"{xshift=0.1cm,yshift=-0.2cm}, shorten <=8pt, shorten >=8pt, Rightarrow, from=3, to=2-4]
	\arrow["{\mathbf{x}}", shorten <=7pt, shorten >=7pt, Rightarrow, from=2-4, to=2]
	\arrow["{\overline{\beta}}"{xshift=0.2cm}, shorten <=25pt, shorten >=25pt, Rightarrow, from=1, to=0]
\end{tikzcd} }  {\large $=$} $\;\beta'\cdot((\mathbf{x}_{i,\kappa}^j)^{-1}\circ f').$
\end{center}
\blue
Using the definition of $\beta'$ \eqref{eq:def-beta'}, we see $\beta'\cdot(\mathbf{x}_{i,\kappa}^j\circ f')$ is exactly the right hand side of the equation \eqref{eq:X_k-kan-ext}. \black

\blue To conclude that $(\xi^h_f,x_{i+1,\kappa}\circ f'//h)$ is a Kan extension, it remains to prove that $\tilde{\beta}$ is the unique 2-cell satisfying the equality $(\tilde{\beta}\circ h)\cdot \xi_f^h=\beta$; this is to say that, given two 2-cells 
\[\adjustbox{scale=0.9}{\begin{tikzcd}
	&& {X_{i+1}} \\
	{A'} & {} &&& {X_\kappa} \\
	&& {X_j}
	\arrow[""{name=0, anchor=center, inner sep=0}, "{x_{i+1,\kappa}}", curve={height=-12pt}, from=1-3, to=2-5]
	\arrow[""{name=1, anchor=center, inner sep=0}, "{f'//h}", curve={height=-12pt}, from=2-1, to=1-3]
	\arrow[""{name=2, anchor=center, inner sep=0}, "{g'}"', curve={height=12pt}, from=2-1, to=3-3]
	\arrow[""{name=3, anchor=center, inner sep=0}, "{x_{j\kappa}}"', curve={height=12pt}, from=3-3, to=2-5]
	\arrow["{\sigma_1}"{xshift=0.1cm}, shift left=5, shorten <=10pt, shorten >=10pt, Rightarrow, from=1, to=2]
	\arrow["{\sigma_2}"'{xshift=-0.1cm}, shift right=5, shorten <=10pt, shorten >=10pt, Rightarrow, from=0, to=3]
\end{tikzcd}}\]
such that 
\begin{equation}\label{eq:**}(\sigma_1\circ h)\cdot \xi^h_f=(\sigma_2\circ h)\cdot \xi^h_f
\end{equation}
then $\sigma_1=\sigma_2$. 
Indeed, with the two 2-cells $\sigma_1$ and $\sigma_2$ as above, since $A'$ is $\kappa$-small, by part 2 of Definition \ref{def:small}, there is  some $n\geq j, i+1$ in $\kappa$ and 2-cells $\bar{\sigma}_1,\, \bar{\sigma}_2\colon x_{i+1, n}\circ f'//h\Rightarrow x_{jn}\circ g'$ such that each $\sigma_r$ is just the composition of $x_{kn}$ with $\bar{\sigma}_r$, up to isomorphism:
\begin{equation}\label{eq:sigmas}\sigma_r=(\mathbf{x}_{j\kappa}^n\circ g')\cdot (x_{n\kappa}\circ \bar{\sigma}_r)\cdot ((\mathbf{x}_{i+1,\kappa}^n)^{-1}\circ (f'//h)),\quad r=1,2.
\end{equation}
Hence, the equality (\ref{eq:**}) and the $\kappa$-smallness of $A$ imply that, using part 3 of Definition~\ref{def:small}, there is an even $p$ between $n$ and $\kappa$ such that
$$x_{n,p+1}\circ \bar{\sigma}_1\circ h=x_{n,p+1}\circ \bar{\sigma}_2\circ h.$$
The two 2-cells $x_{n,p+1}\circ \bar{\sigma}_r$, $r=1,2$, are of the type considered in 2.(b) of Construction \ref{kansmallobject}. Thus, by the definition of the connecting map $x_{p+1, p+2}$ (recall that $x_{p+1, p+2}\cong d_{\gamma}c_{\gamma}$, for $c_{\gamma}$ the bicoequifier of the two 2-cells), we have that $x_{p+1, p+2}\circ x_{n,p+1}\circ \bar{\sigma}_1=x_{p+1, p+2}\circ x_{n,p+1}\circ \bar{\sigma}_2$. Therefore, using (\ref{eq:sigmas}), we obtain $\sigma_1=\sigma_2$.
\black

\vskip1.5mm

    Second, let us consider a 1-cell $p\colon X_0\to P$ with $P\in\slinj(\maps)$. We know that $p$ gives rise to a pseudococone $p_i\colon X_i\to P$ satisfying the conditions in Lemma \ref{lem:ext-to-P}. Now, we want to show that the morphism $p_\kappa\colon X_\kappa\to P$ belongs to $\slinj(\maps)$, i.e. for every $f\colon A\to X_\kappa$ and $h\colon A\to A'\in\maps$
    $$p_\kappa \circ f/h\cong (p_\kappa f)/h.$$
    Since $A$ is $\kappa$-small, there is some even ordinal $i$ and $f'\colon A\to X_i$ such that $f\cong x_{i,\kappa}f'$. From the first part, we know that $f/h\cong x_{i+1,\kappa}\circ f'//h$. Therefore,
\begin{align*}
    p_\kappa\circ f/h
    &  \cong p_\kappa \circ x_{i+1,\kappa} \circ f'//h
    & 
    \\
    & \cong p_{i+1}\circ f'//h
    & (\textrm{by pseudococone condition})
    \\
    & \cong (p_if')/h
    & (\textrm{by construction of $p_{i+1}$, see diagram (\ref{diag:xi_pif})}) 
    \\
    & \cong (p_\kappa x_{i,\kappa}f')/h
    & (\textrm{by pseudococone condition})
    \\
    & \cong (p_\kappa f)/h
    & (\textrm{by }f\cong x_{i,\kappa}f').
\end{align*}

\item[(2)] Let us now consider $X_\kappa$, which is in $\slinj(\maps)$ as we proved in the previous point. Setting $p_0:=x_{0,\kappa}$ and $P:=X_\kappa$, by Lemma~\ref{lem:ext-to-P}, we get a pseudococone $p_i\colon X_i\to X_\kappa$ with $p_i\cong p_0/x_{0,i}$. We will show that, for any $i<\kappa$, \blue we can choose $p_i:=x_{i,\kappa}$ in the construction of Lemma~\ref{lem:ext-to-P}, \black which implies that $p_\kappa\cong 1_{X_\kappa}$ and $1_{X_\kappa}\cong x_{0,\kappa}/x_{0,\kappa}$. As usual, we will proceed inductively.

\vskip2mm

\noindent  \emph{Limit step.} It follows directly by bicolimit properties.

\vskip2mm

\noindent  \emph{Step $i+1$.}
\blue We recall that by construction, $p_{i+1}$ is the unique (up-to-iso) 1-cell equipped with invertible 2-cells $\mathbf{p}_{i,i+1}\colon p_i\to p_{i+1}x_{i,i+1}$ and $\pi\colon p_{i+1}\circ f//h \Rightarrow (p_if)/h$ \black such that
\[\begin{tikzcd}[ampersand replacement=\&]
	\& A \& {A'} \\
	{\xi_{p_if}^h=} \& {X_i} \& {X_{i+1}} \\
	\&\&\& {P=X_\kappa .}
	\arrow["h", from=1-2, to=1-3]
	\arrow[""{name=0, anchor=center, inner sep=0}, "{f//h}"{description}, from=1-3, to=2-3]
	\arrow[""{name=1, anchor=center, inner sep=0}, "{p_{i+1}}"{description}, from=2-3, to=3-4]
	\arrow[""{name=2, anchor=center, inner sep=0}, "f"', from=1-2, to=2-2]
	\arrow["x"{description}, from=2-2, to=2-3]
	\arrow[""{name=3, anchor=center, inner sep=0}, "{p_i}"', curve={height=12pt}, from=2-2, to=3-4]
	\arrow[""{name=4, anchor=center, inner sep=0}, "{(p_if)/h}", curve={height=-12pt}, from=1-3, to=3-4]
	\arrow["{\tilde{\xi}_f^h}", shift right=2, shorten <=15pt, shorten >=15pt, Rightarrow, from=2, to=0]
	\arrow["\pi", shorten <=8pt, shorten >=8pt, Rightarrow, from=2-3, to=4]
	\arrow["{\mathbf{p}_{i,i+1}}"{pos=0.3,xshift=0.2cm,yshift=0.1cm}, shorten <=11pt, shorten >=11pt, Rightarrow, from=3, to=1]
\end{tikzcd}\]
    Now, we will prove that $x_{i+1,\kappa}$ has the same universal property of $p_{i+1}$. Let us recall that in point (1) we proved that $x_{i+1,\kappa}\circ f//h$ together with the 2-cell below is a Kan extension of $x_{i\kappa}f$ along $h$.
\[\begin{tikzcd}[ampersand replacement=\&]
	\& A \& {A'} \\
	{\xi_{x_{i,\kappa}f}^h=} \& {X_i} \& {X_{i+1}} \\
	\&\&\& {X_\kappa}
	\arrow["h", from=1-2, to=1-3]
	\arrow[""{name=0, anchor=center, inner sep=0}, "{f//h}"{description}, from=1-3, to=2-3]
	\arrow[""{name=1, anchor=center, inner sep=0}, "{x_{i+1,\kappa}}"{description}, from=2-3, to=3-4]
	\arrow[""{name=2, anchor=center, inner sep=0}, "f"', from=1-2, to=2-2]
	\arrow["x"{description}, from=2-2, to=2-3]
	\arrow[""{name=3, anchor=center, inner sep=0}, "{x_{i,\kappa}}"', curve={height=12pt}, from=2-2, to=3-4]
	\arrow[""{name=4, anchor=center, inner sep=0}, "{(x_{i,\kappa}f)/h}", curve={height=-12pt}, from=1-3, to=3-4]
	\arrow["{\tilde{\xi}_f^h}", shift right=2, shorten <=15pt, shorten >=15pt, Rightarrow, from=2, to=0]
	\arrow["{\mathbf{x}^{-1}}"{pos=0.3}, shift left=2, shorten <=6pt, shorten >=12pt, Rightarrow, from=3, to=1]
	\arrow["{=}"{description}, draw=none, from=4, to=2-3]
\end{tikzcd}\]
    Since by inductive hypothesis we have \blue $p_i=x_{i,\kappa}$, we can set $\mathbf{p}_{i,i+1}:=(\mathbf{x}_{i,\kappa}^{i+1})^{-1}$ and $\pi=1$ and the diagram above shows that \black $x_{i+1,\kappa}$ satisfies the same universal property as $p_{i+1}$.

\vskip2mm

\noindent \emph{Step $i+2$.} We need to show that \blue following the construction of Lemma~\ref{lem:ext-to-P} we can choose \black $p_{i+2}:=x_{i+2,\kappa}$. \blue We recall that $p_{i+2}$ is defined through the universal property described in \eqref{eq:def-p_i+2}. By inductive hypothesis we have $p_{i+1}=x_{i+1,\kappa}$ and if we show that $(\,x_{i+2,\kappa}d_\gamma,\,(\mathbf{x}_{i+1,\kappa}^{i+2})^{-1}\cdot (x_{i+2,\kappa}\circ\delta_\gamma)\,)$ has the universal property of $(p_\gamma,\,\mathbf{p}_\gamma^c)$, then we can assume $p_\gamma=x_{i+2,\kappa}d_{\gamma}$ and in particular
\begin{center}
\begin{tikzcd}[ampersand replacement=\&]
	{X_{i+1}} \\
	{C_\gamma} \& {X_{i+2}} \& {X_\kappa.}
	\arrow["{c_\gamma}"', from=1-1, to=2-1]
	\arrow[""{name=0, anchor=center, inner sep=0}, "{x_{i+1,\kappa}}", curve={height=-12pt}, from=1-1, to=2-3]
	\arrow["{d_\gamma}"', from=2-1, to=2-2]
	\arrow["x"', from=2-2, to=2-3]
	\arrow["{\mathbf{p}_\gamma^c}"'{pos=0.6,yshift=0.1cm,xshift=0.1cm}, shift left=2, shorten <=12pt, shorten >=20pt, Rightarrow, from=0, to=2-1]
\end{tikzcd} $\quad=\quad$
\begin{tikzcd}[ampersand replacement=\&]
	{X_{i+1}} \\
	{C_\gamma} \& {X_{i+2}} \& {X_\kappa.}
	\arrow["{c_\gamma}"', from=1-1, to=2-1]
	\arrow[""{name=0, anchor=center, inner sep=0}, "{x_{i+1,\kappa}}", curve={height=-12pt}, from=1-1, to=2-3]
	\arrow["{d_\gamma}"', from=2-1, to=2-2]
	\arrow["x"', from=2-2, to=2-3]
	\arrow[""{name=1, anchor=center, inner sep=0}, "x"{description}, curve={height=-6pt}, from=1-1, to=2-2]
	\arrow["\delta\gamma"', shift left=2, shorten <=8pt, shorten >=5pt, Rightarrow, from=1, to=2-1]
	\arrow["{\mathbf{x}^{-1}}"{pos=0.6}, shorten <=5pt, shorten >=5pt, Rightarrow, from=0, to=2-2]
\end{tikzcd}
\end{center}
This would show that $(x_{i+2,\kappa},(\mathbf{x}_{i+1,\kappa}^{i+2})^{-1},1)$ has the same universal property as $(p_{i+2},\mathbf{p}_{i+1,\kappa},\mathbf{p}_\gamma^d)$ and so we could conclude that we can choose $p_{i+2}=x_{i+2,\kappa}$. 

{\blue We consider the two cases for 
$\gamma$,  one leading to a bicoequinserter, the other to a bicoequifier. For the first case,} let us now show that $$(\,x_{i+2,\kappa}d_\gamma,\,(\mathbf{x}_{i+1,\kappa}^{i+2})^{-1}\cdot (x_{i+2,\kappa}\circ\delta_\gamma)\,)$$ satisfy the wanted property, i.e. \eqref{eq:def-p_gamma}. This equation becomes the one below, which is true by definition of $X_{i+2}$ and pseudofunctoriality of the pseudochain. 
\begin{center}
\scalebox{0.8}{
\begin{tikzcd}[ampersand replacement=\&]
	\& {X_{j+1}} \& {X_{i+1}} \\
	{A'} \&\&\& {C_\gamma} \&\& {X_{i+2}} \& {X_\kappa} \\
	\&\& {X_{i+1}}
	\arrow["s"', from=2-1, to=3-3]
	\arrow["{f//h}", from=2-1, to=1-2]
	\arrow["{x_{j+1,i+1}}", from=1-2, to=1-3]
	\arrow["{c_\gamma}"', from=1-3, to=2-4]
	\arrow["{c_\gamma}"{description}, from=3-3, to=2-4]
	\arrow[""{name=0, anchor=center, inner sep=0}, "{p_{i+1}=x_{i+1,\kappa}}"{pos=0.8}, curve={height=-18pt}, from=1-3, to=2-7]
	\arrow["{\chi_\gamma}", shorten <=25pt, shorten >=25pt, Rightarrow, from=1-2, to=3-3]
	\arrow[""{name=1, anchor=center, inner sep=0}, "{d_\gamma}"{description}, from=2-4, to=2-6]
	\arrow["x"{description}, from=2-6, to=2-7]
	\arrow[""{name=2, anchor=center, inner sep=0}, "x"{description}, curve={height=-6pt}, from=1-3, to=2-6]
	\arrow[""{name=3, anchor=center, inner sep=0}, "x"{description}, curve={height=6pt}, from=3-3, to=2-6]
	\arrow[""{name=4, anchor=center, inner sep=0}, "{p_{i+1}=x_{i+1,\kappa}}"'{pos=0.8}, curve={height=24pt}, from=3-3, to=2-7]
	\arrow["{\delta_\gamma}"'{pos=0.2}, shift right=1, shorten <=9pt, shorten >=9pt, Rightarrow, from=2, to=1]
	\arrow["{\mathbf{x}^{-1}}"{pos=0.8,xshift=0.05cm,yshift=-0.1cm}, shorten <=13pt, shorten >=8pt, Rightarrow, from=0, to=2-6]
	\arrow["{\mathbf{x}}"{xshift=0.1cm}, shorten <=9pt, shorten >=9pt, Rightarrow, from=3, to=4]
	\arrow["{{\delta_\gamma}^{-1}}", shorten <=2pt, shorten >=4pt, Rightarrow, from=2-4, to=3]
\end{tikzcd} 
= 
\begin{tikzcd}[ampersand replacement=\&]
	\& {X_{j+1}} \& {X_{i+1}} \\
	{A'} \&\&\& {X_\kappa} \\
	\& {X_{i+1}}
	\arrow["{x_{j+1,i+1}}", from=1-2, to=1-3]
	\arrow["{p_{i+1}}", curve={height=-6pt}, from=1-3, to=2-4]
	\arrow[""{name=0, anchor=center, inner sep=0}, "{x_{j+1,\kappa}}"', curve={height=6pt}, from=1-2, to=2-4]
	\arrow["{x_{i+1,\kappa}}"', from=3-2, to=2-4]
	\arrow["s"', from=2-1, to=3-2]
	\arrow["{f//h}", from=2-1, to=1-2]
	\arrow["{\overline{\gamma}}", shorten <=15pt, shorten >=15pt, Rightarrow, from=1-2, to=3-2]
	\arrow["{\mathbf{x}}"{yshift=0.1cm,xshift=0.1cm}, shift right=1, shorten <=4pt, shorten >=4pt, Rightarrow, from=1-3, to=0]
\end{tikzcd}}
\end{center}
\black 
\blue Regarding the second case, with $\gamma=\{\sigma, \tau\}$ as in 2.(b) of Construction \ref{kansmallobject}, the universal property of $(p_{\gamma}, \mathbf{p}_{\gamma}^c)$ is
 given by diagram (\ref{eq-coeq}), and it is clear that $(x_{i+2,\kappa},(\mathbf{x}_{i+1,\kappa}^{i+2})^{-1})$ also satisfies the same property.

\end{enumerate}
\black 
This concludes the proof of the first part of the theorem, i.e. the inclusion 2-functor $\slinj(\maps)\hookrightarrow\catk$ is the right part of a KZ-adjunction.

\vskip2mm

Finally, it follows from Corollary \ref{cor:Linj(H)} that $\slinj(\maps)$ is the corresponding 2-category of pseudoalgebras of the induced KZ-pseudomonad.
\end{proof}

 \section{Lex colimits, distributive laws and Kan injectivity} \label{sec5}

In this section we are given a  pseudomonad
$(S,s,m)$ and a KZ-pseudomonad
$(T,t,n)$ on a $2$-category $\mathcal{K}$ with weighted bicolimits and a pseudodistributive law 
\[d\colon ST \Rightarrow TS.\]

\blue
The aim of this section will be to study and compare some natural Kan injectivity classes that arise in this context, motivated mostly by the theory of KZ-pseudomonads. \black 

For the theory of pseudodistributive laws over KZ doctrines, we refer to \cite{walker2019distributive,MarmolejoF:dislp}. Recall that, as shown in \cite[Thm. 35 (a)(b), or Cor. 50]{walker2019distributive}, this amounts to a lift $(\hat{T},\hat{t}, \hat{n})$ of $T$ to the category of pseudoalgebras for $S$ as in the diagram below.

\[\begin{tikzcd}[ampersand replacement=\&]
	{S\text{-Alg}} \& {S\text{-Alg}} \\
	\catk \& \catk
	\arrow["T"', from=2-1, to=2-2]
	\arrow[from=1-1, to=2-1]
	\arrow[from=1-2, to=2-2]
	\arrow["{\hat{T}}", from=1-1, to=1-2]
\end{tikzcd}\]

Moreover $\hat{T}$ is KZ too  and its unit $\hat{t}$ coincides with the unit of $T$. Furthermore, because we are assuming that $T$ is KZ, there is at most one such a $d$ \cite[Def 33, Thm. 44 and Cor. 49]{walker2019distributive} and it has to coincide with the left Kan extension below.

\[\begin{tikzcd}[ampersand replacement=\&]
	\& T \\
	ST \&\& TS
	\arrow["Ts", from=1-2, to=2-3]
	\arrow["{d \equiv Ts/sT}"', from=2-1, to=2-3]
	\arrow["sT"', from=1-2, to=2-1]
\end{tikzcd}\]

This situation is pretty common in practice. Our guiding example is that in which $S$ is the free completion under finite limits in Cat, while $T$ is a completion under a family of colimits. The reader might observe that in this specific example $S$ is coKZ, and we have not listed this one as a working assumption. We will come back to this later.

\begin{definition}[Three interesting classes of maps]
Consider the diagrams below collecting the data of the distributivity law $d$ between the pseudomonads $S$ and $T$.
\[\begin{tikzcd}[ampersand replacement=\&]
	\& 1 \&\&\& T \&\&\& S \\
	S \&\& T \& ST \&\& TS \& ST \&\& TS
	\arrow["s"', from=1-2, to=2-1]
	\arrow["t", from=1-2, to=2-3]
	\arrow["Ts", from=1-5, to=2-6]
	\arrow["d"', from=2-4, to=2-6]
	\arrow["sT"', from=1-5, to=2-4]
	\arrow["d"', from=2-7, to=2-9]
	\arrow["St"', from=1-8, to=2-7]
	\arrow["tS", from=1-8, to=2-9]
\end{tikzcd}\]

We can then define three classes of maps in $S\text{-Alg}$.

 \begin{enumerate}
     \item $\mathcal{H}_{St}$ contains the maps $St: S \to ST$.
     \item $\mathcal{H}_{tS}$ contains the maps $tS: S \to TS$. Notice that those coincide with the unit $\hat{t}$ on free $S$-algebras and thus are in the $2$-category $S\text{-Alg}$.
     \item  $\mathcal{H}_{\hat{t}}$ contains the maps $\hat{t}: 1 \to \hat{T}$.
 \end{enumerate}
\end{definition}


\begin{remark}
On a technical level, the rest of the section will be devoted to discuss the diagram below, that is to discuss the relation between the $2$-categories of Kan injectives with respect to these classes of $1$-cells.

\[\begin{tikzcd}
	{\slinj(\maps_{St})} & {\slinj(\maps_{{tS}})} & {\slinj(\maps_{\hat{t}})} & {\hat{T}\text{-Alg}} \\
	\\
	& {S\text{-Alg}}
	\arrow["{d^*}"', from=1-2, to=1-1]
	\arrow[from=1-3, to=1-2]
	\arrow[from=1-1, to=3-2]
	\arrow[from=1-2, to=3-2]
	\arrow[from=1-3, to=3-2]
	\arrow["\simeq"', from=1-4, to=1-3]
	\arrow[from=1-4, to=3-2]
\end{tikzcd}\]

As hinted by the diagram we will show that $d$ yields a forgetful functor from $\slinj(\maps_{St})$ to $\slinj(\maps_{tS})$ and that the two classes $\maps_{St}$ and $\maps_{\hat{t}}$ specify the relatable $2$-categories of Kan injectives. Already at this stage we can infer that $\slinj(\maps_{\hat{t}})$ is equivalent to $\hat{T}{\text{-Alg}}$. Indeed, the lift of a KZ monad will be KZ and thus its category of pseudoalgebras coincides with the Kan injectives with respect to the unit as observed in Theorem \ref{thm:units}.
\end{remark}

\begin{proposition}
The precomposition with $d$ defines a forgetful $2$-functor  \[d^\ast\colon \slinj(\maps_{tS}) \to  \slinj(\maps_{St}).\]

\end{proposition}
\begin{proof}
The key aspect of this proof is to show that in the diagram below, when $X$ is Kan injective with respect to $tS$, the precomposition with $d$ gives us the left Kan extension $f/St$.

\[\begin{tikzcd}[ampersand replacement=\&]
	S \&\& ST \\
	\&\& TS \\
	X
	\arrow["d"', from=1-3, to=2-3]
	\arrow["St", from=1-1, to=1-3]
	\arrow["tS"', from=1-1, to=2-3]
	\arrow["f"', from=1-1, to=3-1]
	\arrow[dashed, from=2-3, to=3-1]
\end{tikzcd}\]

Indeed, if we show this, we have that every $\slinj(\maps_{tS})$ lies in $\slinj(\maps_{St})$ and we will see that a routine idea delivers the functoriality on the spot. Now, recall that by \cite[Thm 44, Cor. 49]{walker2019distributive} $d$ must coincide with $tS/St$. We are thus left with the composition \[(f/tS) \circ (tS/St)\] and we want to show that it coincides with $f/St$. But recall that $tS$ is the unit of $\hat{T}$ on a free algebra, and that $\hat{T}$ is KZ. We can thus apply the Kancellation rule \cite[Rem~2.20]{di2018accessibility} and deduce that \[(f/tS) \circ (tS/St) \cong  f/St,\] which is the thesis. Notice that to apply \cite[Rem~ 2.20]{di2018accessibility}, we need $St$ to be admissible in the sense of Bunge and Funk \cite[Def~ 4.3.2]{bunge2006singular}, this is true by \cite[Lemma 41]{walker2019distributive}.
\end{proof}

\begin{remark}
Because $\mathcal{H}_{tS} \subset \mathcal{H}_{\hat{t}}$, it follows from subsection \ref{subsec:saturation} that we have $\slinj(\maps_{\hat{t}}) \subseteq \slinj(\maps_{tS})$.
\end{remark}

\begin{proposition}
\label{prop:distr-KZ-KZ}
If $S$ is KZ, then  $\slinj(\maps_{tS}) = \slinj(\maps_{\hat{t}})$.
\end{proposition}
\begin{proof}




 Together with the previous remark, it is enough to show that $\maps_{\hat{t}} \subset \maps^{\text{sat}}_{tS}$. By inspecting the diagram below, that is witnessing the fact that every $S$-algebra is reflective in its free completion, 

\[\begin{tikzcd}[ampersand replacement=\&]
	Y \&\& SY \\
	\\
	TY \&\& TSY
	\arrow[""{name=0, anchor=center, inner sep=0}, "{s_Y}"{description}, from=1-1, to=1-3]
	\arrow[""{name=1, anchor=center, inner sep=0}, "a"{description}, curve={height=18pt}, from=1-3, to=1-1]
	\arrow[""{name=2, anchor=center, inner sep=0}, "{Ts_Y}"{description}, from=3-1, to=3-3]
	\arrow[""{name=3, anchor=center, inner sep=0}, "Ta"{description}, curve={height=-18pt}, from=3-3, to=3-1]
	\arrow["{\hat{t}_Y=t_Y}"', from=1-1, to=3-1]
	\arrow["{t_{SY}}", from=1-3, to=3-3]
	\arrow["\dashv"{anchor=center, rotate=-90}, draw=none, from=1, to=0]
	\arrow["\dashv"{anchor=center, rotate=92}, draw=none, from=3, to=2]
\end{tikzcd}\]

 We see that $\hat{t}$ can be obtained from $tS$ via one of the steps that saturates $\maps_{tS}$ in Proposition \ref{prop:saturation-irr} and thus the two classes have the same saturation.
\end{proof}

\subsection{Lex colimits}
The technology of lex colimits was introduced by Garner and Lack \cite{lex} to describe a large class of structures where colimits interact with limits \textit{lex}-ly. The paradigmatic example of this behavior is Grothendieck topoi, where this phenomenon is called descent.

Let us recall very briefly their definitions to set the notation of the subsection and introduce the reader to the topic. We work in $\text{Cat}$, the $2$-category of small (but possibly large) categories, functors between them and natural transformations. This choice will rule out the very interesting case of Grothendieck topoi, and could be avoided by paying the price of a very detailed foundational analysis. We prefer to stick to small categories because the treatment will be largely cleaner.

\begin{constr}
For $\Phi$ a set of weights $ W \colon I^{\op} \rightarrow \text{Set}$ and $C$ a category, we can consider the category $ \Phi_l(C)$ as the full subcategory of the presheaf category $\text{Psh}(C)$ consisting of all $ \Phi$-weighted colimits of representables (see \cite[Sec. 3]{lex}). This construction defines a KZ-pseudomonad on $\text{Cat}$, which lifts to a KZ-pseudomonad over $\text{Lex}$, the $2$-category of small categories with finite limits and functors preserving them.
\[\begin{tikzcd}[ampersand replacement=\&]
	{\text{Lex}} \& {\text{Lex}} \\
	{\text{Cat}} \& {\text{Cat}}
	\arrow["{\Phi_l}"', from=2-1, to=2-2]
	\arrow[from=1-1, to=2-1]
	\arrow[from=1-2, to=2-2]
	\arrow["{\hat{\Phi}_l}", from=1-1, to=1-2]
\end{tikzcd}\]

Of course, this perfectly fits with the narrative of the previous subsection, indeed $\text{Lex}$ is coKZ doctrinal over $\text{Cat}$.
\end{constr}

\begin{definition}[$\Phi$-lex-cocompleteness, \cite{lex}]
$C$ is \emph{$\Phi$-lex-cocomplete} if it is lex and has colimits of shape $\Phi$, i.e. if it is lex and is a $\Phi_l$ algebra. This gives us the $2$-category $\Phi\text{-LexAlg}$, of lex categories supporting a $\Phi_l$ structure and functors preserving both finite limits and $\Phi$-colimits.
\end{definition}

\begin{definition}[$\Phi$-exactness, \cite{lex}]
Now a $\Phi$-lex-cocomplete category is said to be \emph{$\Phi$-exact} if its algebra structure $\Phi_l(C) \to C$
is lex, which amounts to saying that $C$ bears a structure of pseudoalgebra for the pseudomonad $ \Phi_l$ on $\text{Lex}$. This gives us the $2$-category $\Phi\text{-Ex}$, which is another name for $\hat{\Phi}_l\text{-Alg}$.
\end{definition}

Of course, it is easy to see that we have a fully faithful forgetful $2$-functor which is only acknowledging that the requirement of being $\Phi$-exact is stronger than being lex and $\Phi$-cocomplete.

\[\begin{tikzcd}[ampersand replacement=\&]
	{\Phi\text{-LexAlg}} \&\& {\Phi{\text{-Ex}}} \\
	\& {\text{Lex}}
	\arrow[from=1-1, to=2-2]
	\arrow[from=1-3, to=2-2]
	\arrow["U"', from=1-3, to=1-1]
\end{tikzcd}\]

We are now ready to apply our machinery to this situation, and get the following theorem.

\begin{theorem} \label{characteriation of exactness}
The following are equivalent.
\begin{enumerate}
    \item $C$ is $\Phi$-exact.
    \item $C$ is Kan injective in $\text{Lex}$ to all the maps $D \to \Phi_l(D)$.
\end{enumerate}
\end{theorem}

\begin{remark}
This theorem follows directly from the previous subsection, but we shall compare it with the content of \cite[3.4]{lex}. The theorem above may seem to say the same thing of \cite[3.4]{lex}. But their Kan extensions are taken in Cat, while in our paper we compute them in Lex. The forgetful functor from Lex to Cat does not seem to preserve left Kan extensions in general, their result is thus surprising from this point of view.
\end{remark}

\begin{theorem}
$\Phi\text{-LexAlg}$ is equivalent to $\slinj(\maps_{St})$ and the forgetful functor $U$ is precisely $d^*$. 
\end{theorem}

\section*{ Acknowledgments}

We are indebted to the anonymous referee for many valuable comments which helped us to improve the paper.
\black

\bibliography{thebib}
 \bibliographystyle{alpha}

\end{document}